\documentclass[12pt]{amsart}

\usepackage{hyperref}
\usepackage{amsmath,amsthm,amsfonts,amssymb}
\usepackage{mathbbol}
\usepackage{geometry}
\usepackage[backend=bibtex, maxnames=99]{biblatex}
\usepackage[all]{xy}
\usepackage{ytableau}

\usepackage{graphicx}
\usepackage{tikz}
\definecolor{mediumseagreen}{rgb}{0.24, 0.7, 0.44}
\definecolor{green(ryb)}{rgb}{0.4, 0.69, 0.2}
\definecolor{darktangerine}{rgb}{1.0, 0.66, 0.07}
\definecolor{denim}{rgb}{0.08, 0.38, 0.74}
\definecolor{lavender(floral)}{rgb}{0.71, 0.49, 0.86}

\tikzstyle{vertex}=[circle, draw, inner sep=0pt, minimum size=4pt]
\newcommand{\vertex}{\node[vertex]}

\tikzstyle{netnode}=[circle, draw, inner sep=0pt, minimum size=3pt]
\newcommand{\nnode}{\node[netnode]}

\def\cross[#1](#2,#3){
    \draw[fill, color=gray!10] (#2-1,#3-1) rectangle (#2,#3);
    \draw [black, #1] (#2-1,#3-0.5) -- (#2, #3-0.5);
    \draw [black, #1] (#2-0.5, #3-1) -- (#2-0.5, #3);
}

\def\elbow[#1](#2,#3){
    \draw[fill, color=gray!10] (#2-1,#3-1) rectangle (#2,#3);
    \draw [black,#1,domain=180:270] plot ({#2+.5*cos(\x)}, {#3+.5*sin(\x)});
	\draw [black,#1,domain=0:90] 	plot ({#2-1+.5*cos(\x)}, {#3-1+.5*sin(\x)});
}

\def\pivotelbow[#1](#2,#3){
	\draw [black,#1,domain=180:270] plot ({#2+.5*cos(\x)}, {#3+.5*sin(\x)});
}

\usepackage[textwidth=2.5cm, textsize=small, disable]{todonotes}
\definecolor{titaniumyellow}{rgb}{0.93, 0.9, 0.0}

\definecolor{bobcatorange}{rgb}{1, 0.55, 0.0}

\newtheorem{theorem}{Theorem}[section]
\newtheorem{lemma}[theorem]{Lemma}
\newtheorem{proposition}[theorem]{Proposition}
\newtheorem{corollary}[theorem]{Corollary}
\theoremstyle{definition}
\newtheorem{definition}[theorem]{Definition}
\newtheorem{example}[theorem]{Example}
\newtheorem{remark}[theorem]{Remark}

\definecolor{frenchblue}{rgb}{0.0, 0.45, 0.73}
\newcommand{\defn}[1]{{\color{frenchblue} #1}}

\newcommand\bbk{\mathbb{k}}
\newcommand\bbR{\mathbb{R}}

\newcommand\bbP{\mathbb{P}}

\newcommand\bc{\mathbf{c}}
\newcommand\br{\mathbf{r}}

\newcommand\calB{\mathcal{B}}

\newcommand\calF{\mathcal{F}}

\newcommand\calN{\mathcal{N}}
\newcommand\calP{\mathcal{P}}

\newcommand\calR{\mathcal{R}}
\newcommand\calS{\mathcal{S}}

\newcommand\fD{\mathfrak{D}}
\newcommand\fS{\mathfrak{S}}

\newcommand{\GL}{\mathrm{GL}}
\newcommand{\Gr}{\mathrm{Gr}}
\newcommand\Coinv{\mathrm{Coinv}}
\newcommand\col{\mathrm{col}}
\newcommand\Cov{\mathrm{Cov}}
\newcommand\Fl{\mathrm{Fl}}
\newcommand\FPP{\mathrm{FPP}}

\newcommand\Le{\reflectbox{\textup{L}}}
\newcommand\sLe{\scalebox{.5}{\reflectbox{\textup{L}}}}
\newcommand\Mat{\mathrm{Mat}}
\newcommand\nep{\mathrm{nep}}

\newcommand\rank{\mathrm{rank}}
\newcommand\Rothe{\mathrm{Rothe}}
\newcommand\RowS{\mathrm{Row}}
\newcommand\setFPP{\mathcal{FPP}}
\newcommand\sgn{\mathrm{sgn}}

\newcommand\st{\mathrm{st}}

\title{Flag positroid pipe dreams}

\author[Rizer]{Williem Rizer}
\address[Rizer]{Department of Mathematics\\ 
    University of Kentucky
    Lexington, KY 40506-0027
}
\email{williem.rizer@uky.edu}

\author[Yip]{Martha Yip}
\address[Yip]{Department of Mathematics\\
         University of Kentucky\\
        Lexington, KY 40506--0027
}
\email{mpyi222@uky.edu}
\urladdr{https://www.ms.uky.edu/~myip/}

\thanks{MY is partially supported by Simons Collaboration grant 964456.}

\date{\today}

\begin{document}
\begin{abstract}
We introduce flag positroid pipe dreams (FPPs), whose role in the study of complete flag positroids is analogous to the role of $\Le$-diagrams in the study of positroids.
We develop the combinatorics of these diagrams and highlight some of their properties. 
FPPs are in bijection with intervals in the Bruhat order of the symmetric group, and the number of elbows in an FPP is the dimension of the corresponding Richardson cell in the decomposition of the nonnegative flag variety.
We show how complete flag positroids can be built rank by rank via FPPs, and how the $\Le$-diagrams of the positroid constituents of the flag can be obtained from the FPP via a simple standardization operation.
Using partial FPPs, we give an alternative proof of a conjecture of Benedetti, Chavez, and Tamayo on the problem of characterizing elementary positroid quotients via cyclic shifts of decorated permutations, in the nonnegatively representable case.  
Our proof partially addresses a problem of Chen et al. regarding an explicit characterization of the cyclic shift operators purely in terms of decorated permutations.
We show that the poset of nonnegatively representable elementary positroid quotients is self-dual.
The maximal chains of this poset are in bijection with FPPs.
\end{abstract}
\maketitle
\tableofcontents
\parskip=1pt

\section{Introduction}
The grassmannian $\Gr_{k,n}$ of $k$-dimensional subspaces of $\bbR^n$ has a stratification indexed by matroids and it is known by Mn\"ev's Universality Theorem that its geometric structure is complicated.
In contrast, Postnikov~\cite{Postnikov06} showed that the intersection of the matroid stratification of $\Gr_{k,n}$ with the nonnegative part $\Gr_{k,n}^{\geq0}$ (Section~\ref{subsec.Grassmannians}) yields a cell decomposition of $\Gr_{k,n}^{\geq0}$. 
These cells are indexed by {\em positroids}, which are matroids that are representable over $\bbR$ whose maximal minors are nonnegative (Section~\ref{subsec.positroids_flagpositroids}). 
We note that positroids were previously studied by da Silva~\cite{daSilva} in the guise of {\em oriented matroids}, and it was shown by Ardila, Rincon and Williams~\cite{ARW17} that every positively oriented matroid is a positroid.

Ordering the positroid cells by containment gives rise to the {\em circular Bruhat order}, so called due to its similarities to the usual Bruhat order of the symmetric group.
This hints at the fact that positroids are rich in combinatorics.
Postnikov presented many families of combinatorial objects which are in bijection with positroids, for example {\em Grassmann necklaces}, {\em $\Le$-diagrams} (Definition~\ref{defn.Le_diagram}), and {\em decorated permutations} (Definition~\ref{defn.decperm}), to name a few.
In the order we have listed these objects, the data (by which we mean bases) of the corresponding positroid is progressively condensed.

One can similarly study the restriction of the complete flag variety $\Fl_n$ to the nonnegative part $\Fl_n^{\geq0}$ (Section~\ref{subsec.Grassmannians}).
Rietsch~\cite{Rietsch98} showed that $\Fl_n^{\geq0}$ has a cell decomposition indexed by intervals $u\leq v$ in the {\em Bruhat order}, and the cells appearing in this decomposition are the {\em positive Richardson cells} $\calR_{u,v}^{>0}$.
Each point in the nonnegative flag variety is a flag $V_1\subset \cdots\subset V_{n-1}$ of subspaces of $\bbR^n$ where $\dim V_k = k$, and is representable by a real matrix with nonnegative minors whose row span of the first $k$ rows is $V_k$.
Therefore, each point in $\Fl_n^{\geq0}$ corresponds to a {\em complete flag positroid}, which is a sequence of positroids $(P_1, \ldots, P_{n-1})$ such that $\mathrm{rank}\,P_k=k$, $P_k$ is a quotient of $P_{k+1}$, and the flag is representable by an $n\times n$ matrix whose {\em flag minors} are nonnegative (Section~\ref{subsec.positroids_flagpositroids}).
Boretsky, Eur, and Williams~\cite{BEW24} showed that positive Richardson cells are in bijection with complete flag positroids.

It has been pointed out for example in~\cite{BCT22, BK24, BlochKarp23, BEW24, CFGSZ26, JLLO23} that characterizing positroid quotients can be a tricky endeavour.
For starters, not all sequences of positroid quotients correspond to a point in the nonnegative flag variety (Example~\ref{eg.3poset}), and unlike the case of matroids, there exists partial flag positroids that cannot be extended to a complete flag positroid (Remark~\ref{rem.completable_flag_counterexample}).
Hence it is reasonable to focus one's attention on characterizing {\em nonnegatively representable elementary quotients}, which are two-step flag positroids of consecutive ranks. 
We denote this by $P \unlhd_q Q$.

A primary motivation of this article is to extend some of the combinatorial theory of positroids to the flag positroid setting.
We introduce the family of {\em flag positroid pipe dreams} (FPPs), first as a certain reduced configuration of elbow and cross tiles (Definition~\ref{defn.FPP}), then equivalently as a $\Gamma$ pattern-avoiding filling (Definition~\ref{defn.gammapatt}) which allows us to draw a direct connection to $\Le$-diagrams for positroids.
We remark that it is convenient for us to think of $\Le$-diagrams as $\Le$-pipe dreams~\cite[Section 19]{Postnikov06}; $\Le$-pipe dreams are a minor variation of $\Le$-diagrams in which zeros are replaced by cross tiles and ones are replaced by elbow tiles. 
Its virtue lies in the fact that it is easier to visualize the connection between $\Le$-pipe dreams and certain pairs of permutations (Section~\ref{subsec.le_parallel}), which is a key ingredient in the combinatorics of flag positroids.

Pipe dreams were introduced by Bergeron and Billey~\cite{BB93} in their study of Schubert polynomials.
More recently, Lam, Lee, and Shimozono~\cite{LLS21} introduced bumpless pipe dreams (BPDs) as a model for computing double Schubert polynomials, and Knutson and Zinn-Justin~\cite{KZ24} introduced generic pipe dreams (GPDs) for the computation of equivariant cohomology classes of lower-upper varieties. 
Our FPPs share similarities with BPDs and GPDs.

Our first main result (Theorem~\ref{thm.main_FPP}) shows that the set $\setFPP(n)=\{\FPP(u,v) \mid u\leq v \}$ of $n\times n$ flag positroid pipe dreams are in bijection with intervals $u\leq v$ in Bruhat order, and the number of elbow tiles in $\FPP(u,v)$ is the length of the interval $u\leq v$.
This is also the dimension of the positive Richardson cell $\calR_{u,v}^{>0}$, which is in complete analogy with the role of $\Le$-pipe dreams in the theory of positroids~\cite[Theorem 19.1]{Postnikov06}.
We show in Theorem~\ref{thm.rpd_is_le} that the subset of FPPs such that $u$ has at most one ascent can be identified with $\Le$-pipe dreams $\Le(w_0v, w_0u)$, which diagrammatically amounts to a $180^\circ$ rotation (see Figure~\ref{fig.le2}, and note that $\Gamma$ is a $180^\circ$ rotation of $\Le$).

Recently, there has been interest in obtaining combinatorial characterizations of positroid quotients~\cite{BCT22, BK24, BEW24, CFGSZ26, OX22}.
For instance, Benedetti, Chavez, and Tamayo~\cite[Theorem 28]{BCT22} studied elementary quotients of uniform matroids (which are positroids) in terms of their decorated permutations, Benedetti and Knauer~\cite[Theorem 19]{BK24} studied the class of {\em lattice path matroids} (LPMs, which are positroids) and characterized all quotients of a given LPM in terms of good pairs, and Boretsky, Eur, and Williams studied the geometry of flag positroid polytopes and gave a characterization of flag positroids of consecutive ranks via Grassmann necklaces~\cite[Theorem 7.14]{BEW24}.

Our contribution to this body of work is our second main result (Theorem~\ref{thm.char_Q} and its corollaries), which is a combinatorial characterization of two-step flag positroids of consecutive ranks in terms of partial FPPs.
Given a rank $k$ positroid $P$ with associated partial FPP $D^{\sLe}(P)$ (Definition~\ref{defn.DLP}), we can compute all rank $k+1$ positroids $Q$ such that $P\unlhd_q Q$ by appending below $D^{\sLe}(P)$ a row of tiles such that  elbow tiles occur only in some nonempty subset of the {\em unblocked columns} of $D^{\sLe}(P)$ (Definition~\ref{defn.Leblocked}).
This yields another partial FPP from which we can compute the bases of $Q$.
A curious consequence is that there are $2^{|U|}-1$ nonnegatively representable elementary quotients $P\unlhd_q Q$ for each fixed $P$, where $|U|$ is the number of unblocked columns of $D^{\sLe}(P)$.
These results are obtained by studying a map $\Phi$ from two-step nonnegative flags $\Fl_{(k,k+1);n}^{\geq0}$ onto its image inside $\Gr_{k+1,n+1}^{\geq0}$, and the induced bijection $\phi$ from the set of two-step flag positroids $\calP_{(k,k+1)}$ on $[n]$ to a certain subset $\calR_{k+1}$ of rank $k+1$ positroids on the ground set $[0,n]$ (see Theorems~\ref{thm:projection-bijection} and~\ref{thm.positroid-quotient-bijection}; we remark that these are similar to the constructions of Boretsky, Eur, and Williams~\cite[Section 7]{BEW24}).

The third main result in this article reveals the true beauty of FPPs and explains the novelty of our combinatorial approach.
In Theorem~\ref{thm.fpp_richardson} we prove that the flag positroid pipe dream $\FPP(u,v)$ corresponding to the interval $u\leq v$ indeed encodes the flag positroid $P_\bullet(u,v) = (P_1(u,v),\ldots, P_{n-1}(u,v))$ associated to any point of the positive Richardson cell $\calR_{u,v}^{>0}$.
Specifically, for each $k=1,\ldots, n$, the restriction of $D=\FPP(u,v)$ to its first $k$ rows yields a partial FPP we denote by $D|_k$.
Generalizing a construction of Postnikov~\cite[Section 4]{Postnikov06}, the partial FPP defines an acyclic directed graph $G(D|_k)$ (Definition~\ref{defn.GofD}) and the bases of $P_k(u,v)$ are computed by the sinks of $k$-tuples of non-intersecting paths in $G(D|_k)$.
Additionally, the $\Le$-diagram of $P_k(u,v)$ can be easily obtained from $D$ via a standardization operation (Theorem~\ref{thm.sti_preserves_bases}); that is, $D^{\sLe}(P_k) = \st(D|_k)$.
We remark that unlike the theory of acyclic directed graphs for positroids, the graph $G(D|_k)$ may be nonplanar (see Example~\ref{eg.crossingGD}).
This subtle difference is reflected in FPPs in the sense that a $\Gamma$ pattern in an FPP is a generalization of a $\Le$ pattern in a $\Le$-pipe dream.

From our perspective, flag positroid pipe dreams are a novel combinatorial construction in the theory of positroids because a single FPP encodes an entire flag positroid from which we can easily recover the data of the $k$-th constituent positroid by restricting to the first $k$ rows.
Put another way, we can construct a flag positroid simply by building an FPP row by row such that for each $k$, the partial FPP avoids the $\Gamma$ pattern.

Our fourth main result addresses an open problem of Chen et al.~\cite[Remark 3.6]{CFGSZ26}.
We begin by recounting the history of the problem.
Based on extensive computational evidence, Benedetti, Chavez, and Tamayo put forth the following conjecture~\cite[Conjecture 37]{BCT22} regarding a combinatorial characterization of elementary positroid quotients: If $P$ is an elementary quotient of $Q$ and $\pi$ and $\sigma$ are their respective decorated permutations, then there is a subset $A \subset [n]$ such that $\sigma= \overrightarrow{\rho_A}(\pi)$, 
where $\overrightarrow{\rho_A}$ is a {\em right cyclic shift} operator (Definition~\ref{defn.rcshift}).

The set $A$ denotes the frozen positions of $\pi$, and $\sigma$ is obtained by cyclically shifting all non-frozen values of $\pi$ one place to the right and decorating any new fixed point with an overline. 
This conjecture was proved by Chen et al.~\cite[Theorem 3.4]{CFGSZ26} using the concept of CW-arrows of Oh and Xiang~\cite{OX22}.
They left open the problem of finding ``a concise characterization - based solely on $\pi$ - of the sets $A$''.

Having developed the combinatorial theory of FPPs, we leverage the close relationship between $\Le$-pipe dreams and decorated permutations to give in Theorem~\ref{thm.covering_decperms} a characterization of nonnegatively representable elementary  positroid quotients in terms of the {\em unblocked positions} (Definition~\ref{defn.unblockedposn}) of a decorated permutation.
A strength of our approach is the explicit characterization (Definition~\ref{defn.TC}) of the frozen sets $A$ in the cyclic shift operator on decorated permutations purely in terms of $\pi$.
This partially addresses the open problem of Chen et al. in the case of nonnegatively representable elementary quotients.

The final object of study in this article is the poset $(\Pi_n, \unlhd_q)$ of nonnegatively representable positroid quotients (Figure~\ref{fig.3NNposet}), which is a subposet of the poset $(\Pi_n, \leq_q)$ of positroid quotients (Figure~\ref{fig.3poset}) that was introduced by Benedetti, Chavez, and Tamayo~\cite{BCT22}. 
As already noted, each positroid $P$ in $(\Pi_n, \unlhd_q)$ is covered by $2^{|U|}-1$ elements, where $|U|$ is the number of unblocked columns of the partial FPP $D^{\sLe}(P)$.
Furthermore, the maximal chains in $(\Pi_n, \unlhd_q)$ correspond with the flag positroid pipe dreams $\setFPP(n)$ and hence with the complete flag positroids on $[n]$.
Using the combinatorics we have developed, we prove that $(\Pi_n, \unlhd_q)$ is self-dual via matroid duals (equivalently, taking inverses of decorated permutations).
Since $(\Pi_n, \unlhd_q)$ contains the quotient poset of lattice path matroids~\cite{BK24} as an induced subposet, it follows that the quotient poset of LPMs is also self-dual.

We summarize the organization of this article. 
In Section~\ref{sec.background} we provide the relevant background on positroids, flag positroids, and the nonnegative Grassmannian and flag variety. 
In Section~\ref{sec.fpp} we develop the combinatorics of FPPs, and explain how to identify $\Le$-pipe dreams as special cases of FPPs.
Section~\ref{sec.positroid_quotients} is the heart of this article, where we develop the connection between FPPs, positroid quotients, and complete flag positroids. 
As an application, in Section~\ref{sec.decperm} we give a characterization of nonnegatively representable elementary positroid quotients in terms of decorated permutations and give a partial answer to the problem of Chen et al..
We conclude with a quick study of some properties of the poset $(\Pi_n, \unlhd_q)$.

\noindent\subsubsection*{Acknowledgements}
We thank Carolina Benedetti, Sergio Fernandez de soto Guerrero, Rafael Gonz\'alez D'Le\'on, Daoji Huang, and Melissa Sherman-Bennett for enlightening conversations on positroids and combinatorics.

\section{Background}\label{sec.background}
In this section, we provide some background on positroids, flag positroids, and their combinatorics.
See~\cite{Postnikov06}, \cite{BEW24}, \cite{BK24} for more details.

\subsection{Positroids and flag positroids}\label{subsec.positroids_flagpositroids}
Let $[n]=\{1,\ldots, n\}$.
A \defn{matroid} $M=(E,\calB)$ consists of a ground set $E=[n]$ and a nonempty collection $\calB$ of subsets of $E$ satisfying the basis exchange property: for any $B, B'\in \calB$ and $b\in B\backslash B'$, there exists $b'\in B'\backslash B$ such that $B \backslash \{b\} \cup \{b'\} \in \calB$.

An element in $\calB$ is a \defn{basis} of $M$ and every basis of $M$ has the same size; this is the \defn{rank} of $M$.
The \defn{dual} of the matroid $M$ is the matroid $M^*$ whose bases are the complements of the bases of $M$ in $E$. 
Given two matroids $M$ and $M'$ on the ground set $E$ with bases $\calB$ and $\calB'$ respectively, $M$ is a \defn{quotient} of $M'$ and we write \defn{$M \leq_q M'$} if for each $B'\in \calB'$ and $j\notin B'$, there exists $B\in \calB$ such that $B\subseteq B'$ and if $B \cup \{j\} \backslash \{i\} \in \calB$ then $B' \cup \{j\}\backslash \{i\} \in \calB'$ for all $i\in B$.
A matroid quotient is \defn{elementary} if the matroids are of consecutive ranks.

Two matroids $(E,\calB)$ and $(E',\calB')$ are \defn{isomorphic} if there exists a bijection $\varphi: E\rightarrow E'$ that preserves independent sets.
A \defn{linear matroid} of rank $k$ is a matroid arising from a $k\times n$ matrix $A$ over a field $\bbk$, where $E$ is the index set of the columns of $A$ and $\calB$ consists of the $k$-subsets of $E$ whose corresponding columns are linearly independent.
A matroid $M$ is \defn{realizable} or \defn{representable} over a field $\bbk$ if it is isomorphic to a linear matroid over $\bbk$.

\begin{definition}
A \defn{positroid} is a matroid that is representable by a matrix whose maximal minors are nonnegative.
\end{definition}
Benedetti and Knauer note in~\cite[Example 1]{BK24} that being a positroid is not preserved under matroid isomorphism, and that the ordering of the ground set of a positroid is a part of the definition.

Let $\br = (r_1,\ldots, r_t)$ be an increasing sequence of integers $1\leq r_1 < \cdots < r_t \leq n-1$.
A \defn{flag matroid} of ranks $\br$ is a sequence of matroids $M_\bullet=(M_1,\dots, M_t)$ such that $\rank\,M_i = r_i$ and $M_i\leq_q M_{i+1}$ for all $i=1, \ldots, t-1$.
A flag matroid $M_\bullet=(M_1,\dots, M_t)$ of ranks $\br$ is \defn{representable} over $\bbk$ if there exists an $r_t \times n$ matrix $A$ such that the $r_i \times n$ submatrix formed by the first $r_i$ rows realizes $M_i$, for each $i=1,\ldots, t$.
The \defn{flag minors} of $A$ are the minors formed by the first $r_i$ rows of $A$, for $i=1,\ldots, t$.

\begin{definition}
A \defn{flag positroid} is a flag matroid that is representable over $\bbR$ by a $r_t\times n$ matrix $A$ whose flag minors are nonnegative.
\end{definition}
In~\cite{BK24}, the authors call such a flag matroid \defn{nonnegatively representable}.
Equivalently, this means that the flag positroid is a point in the nonnegative flag variety $\Fl_n^{\geq0}$, see Section~\ref{subsec.Grassmannians}.
For positroids $P, Q$ on $[n]$, we write \defn{$P\unlhd_q Q$} if $P$ is a nonnegatively representable quotient of $Q$.
That is, $(P,Q)$ is a two-step flag positroid.
If $P$ and $Q$ are of consecutive rank, then $(P,Q)$ is an \defn{elementary two-step flag positroid}.

Contrast the definition of a flag positroid with that of a \defn{positroid flag matroid}, which is a flag matroid whose constituents are all positroids.  
By definition, every flag positroid is a positroid flag matroid, but not every positroid flag matroid is a flag positroid, as we discuss in the next example.

\begin{example}\label{eg.3poset}
The \defn{partial order of positroid quotients} $(\Pi_n,\leq_q)$ was introduced by Benedetti, Chavez, and Tamayo~\cite{BCT22}.  
For a fixed positive integer $n$, the elements of the poset are the positroids on $[n]$, and $P$ is covered by $Q$ in the poset if and only if $P$ is an elementary quotient of $Q$.
Maximal chains in $(\Pi_n, \leq_q)$ are complete positroid flag matroids.

Figure~\ref{fig.3poset} shows the Hasse diagram of $\Pi_3$ where the positroids are indexed by their associated decorated permutations (see Section~\ref{sec.decperm}). 
From the figure, we see that there are $22$ positroid flag matroids, but three of them are not flag positroids/nonnegatively representable.  
These are indicated by dotted lines in the figure between the positroids of ranks one and two.
The remaining $19$ maximal chains of flag positroids correspond to the $19$ intervals in the Bruhat order of $\fS_3$.

For example, the positroid flag matroid $(P_1, P_2)$ on $[3]$ whose constituents have bases $\calB(P_1) = \{1,3\}$ and $\calB(P_2)=\{12,23\}$ is not nonnegatively representable.  
This flag is depicted in Figure~\ref{fig.3poset} in terms of its decorated permutations as $(3\underline{2}1, 3\overline{2}1)$.
\end{example}

\begin{remark}\label{rem.completable_flag_counterexample}
We note that not all flag positroids can be extended to a complete flag positroid.
See Boretsky, Eur, and Williams~\cite[Example 4.6]{BEW24} for details, where they discuss the example of the flag positroid $(P_1,P_3)$ on $[4]$ whose constituents have bases $\calB(P_1)=\{1,2,3,4\}$ and $\calB(P_3)=\{123,234\}$.
\end{remark}

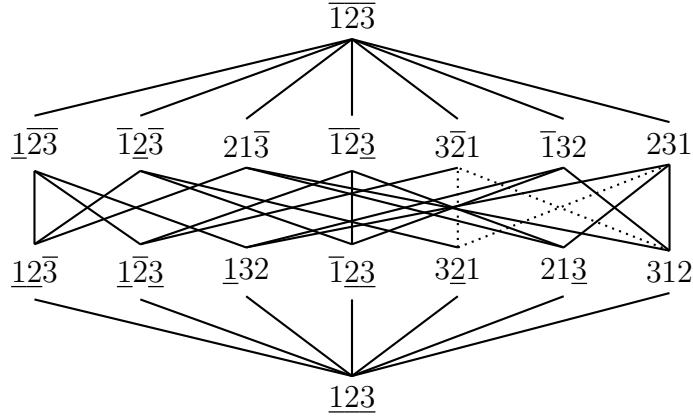
\begin{figure}[ht!]
\begin{tikzpicture}[xscale=1.4, yscale=1.7]
\node[](p1bar2bar3bar) at (0,3){$\overline{1}\overline{2}\overline{3}$};

\node[](p231) at (3,2){$231$};
\node[](p1bar32) at (2,2){$\overline{1}32$};
\node[](p32bar1) at (1,2){$3\overline{2}1$};
\node[](p1bar2bar3) at (0,2){$\overline{1}\overline{2}\underline{3}$};
\node[](p213bar) at (-1,2){$21\overline{3}$};
\node[](p1bar23bar) at (-2,2){$\overline{1}\underline{2}\overline{3}$};
\node[](p12bar3bar) at (-3,2){$\underline{1}\overline{2}\overline{3}$};

\node[](p312) at (3,1){$312$};
\node[](p213) at (2,1){$21\underline{3}$};
\node[](p321) at (1,1){$3\underline{2}1$};
\node[](p1bar23) at (0,1){$\overline{1}\underline{2}\underline{3}$};
\node[](p132) at (-1,1){$\underline{1}32$};
\node[](p12bar3) at (-2,1){$\underline{1}\overline{2}\underline{3}$};
\node[](p123bar) at (-3,1){$\underline{1}\underline{2}\overline{3}$};

\node[](p123) at (0,0){$\underline{1}\underline{2}\underline{3}$};

\draw[thick] (p231.north)--(p1bar2bar3bar.south);
\draw[thick] (p32bar1.north)--(p1bar2bar3bar.south);
\draw[thick] (p1bar32.north)--(p1bar2bar3bar.south);
\draw[thick] (p1bar2bar3.north)--(p1bar2bar3bar.south);
\draw[thick] (p213bar.north)--(p1bar2bar3bar.south);
\draw[thick] (p1bar23bar.north)--(p1bar2bar3bar.south);
\draw[thick] (p12bar3bar.north)--(p1bar2bar3bar.south);

\draw[thick] (p312.north)--(p231.south);
\draw[thick, dotted] (p312.north)--(p32bar1.south);
\draw[thick] (p312.north)--(p1bar32.south);
\draw[thick] (p312.north)--(p213bar.south);

\draw[thick] (p213.north)--(p231.south);
\draw[thick] (p213.north)--(p1bar2bar3.south);
\draw[thick] (p213.north)--(p213bar.south);

\draw[thick, dotted] (p321.north)--(p231.south);
\draw[thick, dotted] (p321.north)--(p32bar1.south);
\draw[thick] (p321.north)--(p1bar23bar.south);

\draw[thick] (p1bar23.north)--(p1bar32.south);
\draw[thick] (p1bar23.north)--(p1bar2bar3.south);
\draw[thick] (p1bar23.north)--(p1bar23bar.south);

\draw[thick] (p132.north)--(p231.south);
\draw[thick] (p132.north)--(p1bar32.south);
\draw[thick] (p132.north)--(p12bar3bar.south);

\draw[thick] (p12bar3.north)--(p32bar1.south);
\draw[thick] (p12bar3.north)--(p1bar2bar3.south);
\draw[thick] (p12bar3.north)--(p12bar3bar.south);

\draw[thick] (p123bar.north)--(p213bar.south);
\draw[thick] (p123bar.north)--(p1bar23bar.south);
\draw[thick] (p123bar.north)--(p12bar3bar.south);

\draw[thick] (p123.north)--(p312.south);
\draw[thick] (p123.north)--(p213.south);
\draw[thick] (p123.north)--(p321.south);
\draw[thick] (p123.north)--(p1bar23.south);
\draw[thick] (p123.north)--(p132.south);
\draw[thick] (p123.north)--(p12bar3.south);
\draw[thick] (p123.north)--(p123bar.south);

\end{tikzpicture}
\caption{The poset of positroid quotients $(\Pi_3,\leq_q)$.
A dotted covering relation indicates a quotient that is not nonnegatively representable.
See Example~\ref{eg.3poset}.  
Also see Figure~\ref{fig.3NNposet}.}
\label{fig.3poset}
\end{figure}

\subsection{Grassmannians, flag varieties}\label{subsec.Grassmannians}
A motivation for studying positroids comes from studying the matroid stratification of Grassmannians.
For $0\leq k \leq n$, the (real) \defn{Grassmannian} $\Gr_{k,n}$ is the set of $k$-dimensional vector subspaces of $\bbR^n$.
A \defn{representation} of a vector space $V\in \Gr_{k,n}$ is a $k\times n$ matrix $A(V)\in \Mat_{k,n}$ whose rows are a basis for $V$.

Given a $k\times n$ matrix $A$ and $S\in \binom{[n]}{k}$, let $\Delta_S(A)$ denote the $k\times k$ maximal minor of $A$ whose columns are indexed by $S$.
The matrix $A(V)$ defines a matroid $M(V)=([n],\calB)$ where $S\in\calB$ if and only if $\Delta_S(A(V))\neq0$.

The \defn{Pl\"ucker embedding} of the Grassmannian $\Gr_{k,n}$ into projective space is given by
\[\Gr_{k,n} \rightarrow \bbP^{\binom{n}{k}-1}: V \mapsto (\Delta_S( A(V)) )_{S\in \binom{[n]}{k}},
\]
consisting of all maximal minors of $A(V)$.

The \defn{nonnegative Grassmannian} $\Gr_{k,n}^{\geq0}$ is the subset of vector spaces $V\in \Gr_{k,n}$ for which there exists a matrix $A(V)\in\Mat_{k,n}$ whose rows are a basis for $V$ and whose Pl\"ucker coordinates are nonnegative.
Postnikov~\cite[Theorem 3.5]{Postnikov06} showed that the nonnegative Grassmannian has a cell decomposition 
\[\Gr_{k,n}^{\geq0} = \bigsqcup_P \calS_P
\]
into \defn{nonnegative Grassmann cells} or \defn{positroid cells} $\calS_P = \{V \in \Gr_{k,n}^{\geq0} \mid M(V)= P \}$ indexed by rank $k$ positroids $P$ on $[n]$.
Furthermore, he showed that the decomposition of the nonnegative Grassmannian into positroid cells has the structure of a CW-complex.

Given a sequence of integers $\br = (r_1,\ldots, r_t)$ where $1\leq r_1 < \cdots < r_t \leq n-1$, let $P_{\br;n}$ denote the parabolic subgroup of $\GL_n(\bbR)$ consisting of block upper-triangular matrices with block sizes $r_1, r_2-r_1,\ldots, n-r_t$.
The \defn{flag variety} is
\[\Fl_{\br;n}(\bbR) = \GL_n(\bbR)/P_{\br;n},\]
and the \defn{complete flag variety} is $\Fl_n = \Fl_{(1,\ldots, n-1);n}(\bbR) = \GL_n(\bbR)/B_n$, where $B_n$ is the Borel subgroup of $\GL_n(\bbR)$ of upper-triangular matrices.

The flag variety may be identified with flags of subspaces in $\bbR^n$:
\[\Fl_{\br;n}(\bbR) = \left\{V_\bullet =  V_1 \subset \cdots \subset V_t \mid \dim V_i = r_i \hbox{ for each } i \right\},
\]
and $\Fl_n$ is a subvariety of $\Gr_{1,n}\times \Gr_{2,n}\times \cdots \times \Gr_{n-1,n}$.
A \defn{representation} of a flag $V_\bullet\in \Fl_{\br;n}$ is a $r_t\times n$ matrix $A(V_\bullet)$ whose first $r_i$ rows is a basis for $V_i$.
The \defn{Pl\"ucker embedding} of the flag variety $\Fl_{\br;n}$ is
\[\Fl_{\br;n}(\bbR) \rightarrow \bbP^{\binom{n}{r_1}-1} \times \cdots \times \bbP^{\binom{n}{r_t}-1}: V_\bullet \mapsto \left(\Delta_{S_1}(A(V_\bullet))_{S_1\in \binom{[n]}{r_1}}, \ldots, \Delta_{S_t}(A(V_\bullet))_{S_t\in \binom{[n]}{r_t}} \right).
\]

The \defn{nonnegative flag variety} $\Fl_{\br;n}^{\geq 0}$ is the subset of $\Fl_{\br;n}$ with nonnegative Pl\"ucker coordinates, so it has a \defn{nonnegative representation} $A\in \Mat_{r_t\times n}$ with flag minors $\Delta_S (A)\geq0$ for $S\in \binom{[n]}{r_i}$, $i=1,\ldots, t$.
Rietsch~\cite{Rietsch98} showed that the nonnegative flag variety has a cell decomposition
\[\Fl_n^{\geq0} = \bigsqcup_{u\leq v} \calR_{u,v}^{>0} \]
into \defn{positive Richardson cells} $\calR_{u,v}^{>0} = BvB/B \cap B^-uB/B \cap \Fl_n^{\geq0}$ where $B=B_n$ and $B^-$ is the subgroup of lower-triangular matrices in $\GL_n(\bbR)$.
This decomposition is indexed by pairs of permutations $u\leq v$ in Bruhat order (see Section~\ref{subsec.Bruhat}).
Riestch also showed that $\Fl_{\br;n}^{\geq0}$ has a cell decomposition, obtained by projecting that of $\Fl_n^{\geq0}$.

If we consider the Grassmannian $\Gr_{k,n}$ as the flag variety $\Fl_{(k);n}$, then the decompositions of Postnikov and Rietsch agree. 
It turns out that positroids are in bijection with pairs of permutations $t\leq w$ in Bruhat order where $w$ is a permutation with at most one descent, see for example~\cite[Theorem 19.1]{Postnikov06}.

\section{Flag positroid pipe dreams}\label{sec.fpp}
In this section we introduce the family $\setFPP(n)$ of flag positroid pipe dreams (FPPs, Definition~\ref{defn.FPP}) which serve a similar purpose for flag positroids as the $\Le$-diagrams/$\Le$-pipe dreams do for positroids.

We show in Theorem~\ref{thm.main_FPP} that FPPs are in bijection with intervals $[u,v]$ in the Bruhat order on the symmetric group $\fS_n$, and moreover the number of elbows in an FPP is the length of its corresponding Bruhat interval.
Hence, FPPs index positive Richardson cells in the decomposition of the nonnegative flag variety $\Fl_n^{\geq0}$, and the number of elbows in an FPP gives the dimension of its corresponding cell.
 
In analogy with $\Le$-diagrams, we show in Proposition~\ref{prop:staircase-equivalence} that FPPs have a characterization as a filling that avoids the $\Gamma$ pattern (Definition~\ref{defn.gammapatt}).
Lastly, we make a direct connection with $\Le$-pipe dreams in Theorem~\ref{thm.rpd_is_le} by showing how $\Le$-pipe dreams can be obtained as special cases of flag positroid pipe dreams.

\subsection{Flag positroid pipe dreams}
\begin{definition}
Let $u=u_1 \cdots u_n \in \fS_n$. 
The \defn{(dual) Rothe diagram} of $u$ is 
\[
\Rothe(u) = \{ (i,j) \mid 1\leq i,j\leq n, u_i < j, u_j^{-1} > i\}.
\]
\end{definition}
We may visualize a permutation $u$ by an $n\times n$ array of boxes (indexed using the conventions of a matrix), putting a dot in the $i$-th row and $u_i$-th column for $i=1,\ldots, n$.
These boxes are the \defn{pivots} of $u$.
The Rothe diagram is the set of boxes which remain after crossing out any box that lies below or to the left of a dot in the array.
See Figure~\ref{fig.rothe} for an illustration.

\begin{remark}
We use the convention where the boxes of the Rothe diagram $\Rothe(u)$ correspond to the \defn{coinversions} of $u$, where $\Coinv(u) = \{ (u_i, u_j) \mid i < j \hbox{ and } u_i < u_j\}$.
\end{remark}

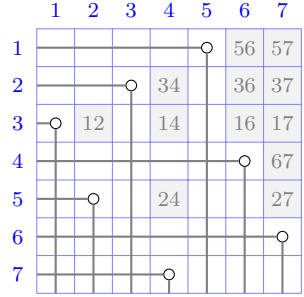
\begin{figure}[ht!]
\begin{tikzpicture}
\begin{scope}[scale=.5, xshift=0]
	\vertex(w1) at (4.5, 6.5) {};
	\vertex(w2) at (2.5, 5.5) {};
	\vertex(w3) at (0.5, 4.5) {};
	\vertex(w4) at (5.5, 3.5) {};
	\vertex(w5) at (1.5, 2.5) {};
	\vertex(w6) at (6.5, 1.5) {};
	\vertex(w7) at (3.5, 0.5) {};

	\fill[gray!10] (1,4) rectangle (2,5);
	\fill[gray!10] (3,2) rectangle (4,3);
	\fill[gray!10] (3,4) rectangle (4,6);
	\fill[gray!10] (5,4) rectangle (6,7);
	\fill[gray!10] (6,2) rectangle (7,7);

	\draw[very thin, color=blue!50] (0,0) grid (7,7);

	\draw[thick, color=gray] (0,6.5)--(w1)--(4.5,0);
	\draw[thick, color=gray] (0,5.5)--(w2)--(2.5,0);
	\draw[thick, color=gray] (0,4.5)--(w3)--(0.5,0);
	\draw[thick, color=gray] (0,3.5)--(w4)--(5.5,0);
	\draw[thick, color=gray] (0,2.5)--(w5)--(1.5,0);
	\draw[thick, color=gray] (0,1.5)--(w6)--(6.5,0);
	\draw[thick, color=gray] (0,0.5)--(w7)--(3.5,0);

	\node[] at (5.5,6.5) {\color{gray}{\tiny$56$}};
	\node[] at (6.5,6.5) {\color{gray}{\tiny$57$}};
	\node[] at (3.5,5.5) {\color{gray}{\tiny$34$}};
	\node[] at (5.5,5.5) {\color{gray}{\tiny$36$}};
	\node[] at (6.5,5.5) {\color{gray}{\tiny$37$}};
	\node[] at (1.5,4.5) {\color{gray}{\tiny$12$}};
	\node[] at (3.5,4.5) {\color{gray}{\tiny$14$}};
	\node[] at (5.5,4.5) {\color{gray}{\tiny$16$}};
	\node[] at (6.5,4.5) {\color{gray}{\tiny$17$}};
	\node[] at (6.5,3.5) {\color{gray}{\tiny$67$}};
    \node[] at (3.5,2.5) {\color{gray}{\tiny$24$}};
    \node[] at (6.5,2.5) {\color{gray}{\tiny$27$}};
	
	\node at (-.5,6.5) {\color{blue}{\tiny$1$}};
	\node at (-.5,5.5) {\color{blue}{\tiny$2$}};
	\node at (-.5,4.5) {\color{blue}{\tiny$3$}};
	\node at (-.5,3.5) {\color{blue}{\tiny$4$}};
	\node at (-.5,2.5) {\color{blue}{\tiny$5$}};
	\node at (-.5,1.5) {\color{blue}{\tiny$6$}};
	\node at (-.5,.5) {\color{blue}{\tiny$7$}};
	\node at (.5,7.5) {\color{blue}{\tiny$1$}};
	\node at (1.5,7.5) {\color{blue}{\tiny$2$}};
	\node at (2.5,7.5) {\color{blue}{\tiny$3$}};
	\node at (3.5,7.5) {\color{blue}{\tiny$4$}};
	\node at (4.5,7.5) {\color{blue}{\tiny$5$}};
	\node at (5.5,7.5) {\color{blue}{\tiny$6$}};
	\node at (6.5,7.5) {\color{blue}{\tiny$7$}};
\end{scope}
\end{tikzpicture}
\caption{The Rothe diagram of $u=5316274$, where the shaded boxes of $\Rothe(u)$ show the correspondence with the coinversion pairs of $u$. 
}
\label{fig.rothe}
\end{figure}

\begin{definition}\label{defn.rrpd}
Let $u = u_1\cdots u_n\in \fS_n$.  
A \defn{Rothe pipe dream of $u$} is a tiling of an $n\times n$ array of boxes by the six tiles
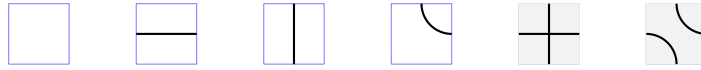
\begin{center}
\begin{tikzpicture}
\begin{scope}[scale=.8, xshift=0]
    \draw[color=blue!50] (0,0) grid (1,1);
\end{scope}

\begin{scope}[scale=.8, xshift=60]
    \draw[color=blue!50] (0,0) grid (1,1);
    \draw [black, thick] (0,0.5) -- (1,0.5);
\end{scope}

\begin{scope}[scale=.8, xshift=120]
    \draw[color=blue!50] (0,0) grid (1,1);
    \draw [black, thick] (0.5,0) -- (0.5,1);  
\end{scope}

\begin{scope}[scale=.8, xshift=180]
    \draw[color=blue!50] (0,0) grid (1,1);
    \draw [black, thick, domain=180:270] plot ({1+.5*cos(\x)}, {1+.5*sin(\x)});  
\end{scope}

\begin{scope}[scale=.8, xshift=240]
    \draw[color=blue!50] (0,0) grid (1,1);
    \cross[thick](1,1);  
\end{scope}

\begin{scope}[scale=.8, xshift=300]
    \draw[color=blue!50] (0,0) grid (1,1);
    \elbow[thick](1,1); 
\end{scope}
\end{tikzpicture}
\end{center}
(respectively empty, horizontal pipe, vertical pipe, pivot elbow, cross, and elbow)
such that:
\begin{enumerate}
    \item[(i)] there are $n$ pipes that begin at the top of the array and end at the right boundary of the array,
    \item[(ii)] a pivot $(i,u_i)$ is tiled by a pivot elbow,
    \item[(iii)] a box that is below some pivot of $u$ and left of some pivot of $u$ is an empty tile,
    \item[(iv)] a box $(i,j)$ with $j>u_i$ that is below some pivot of $u$ is a horizontal pipe, and
    \item[(v)] a box $(i,j)$ with $j<u_i$ that is above some pivot of $u$ is a vertical pipe.
\end{enumerate}
As a consequence of these conditions, $u$ completely determines the placements of pivot elbows, horizontal pipes, vertical pipes, and empty tiles, and each box in the Rothe diagram of $u$ can be tiled by a cross or an elbow.

A Rothe pipe dream is \defn{reduced} if in addition to the above five conditions it also satisfies:
\begin{enumerate}
    \item[(vi)] two pipes may not cross more than once.
\end{enumerate}

Given a pipe in a reduced Rothe pipe dream, we consider it directed from the northwest to the  southeast.
The \defn{exit permutation} of a reduced Rothe pipe dream $R$ is the order in which the pipes exit the array on the right boundary.
That is, if the pipe beginning in column $j$ exits in row $i$, then $i$ is the \defn{exit row} of pipe $j$.
\end{definition}

If every box in the Rothe diagram of $u$ is tiled by a cross, then the reduced Rothe pipe dream has exit permutation $u$.  
At the other extreme, if every box in the Rothe diagram of $u$ is tiled by an elbow, then the exit permutation is the \defn{longest permutation} $w_0 = n(n-1)\cdots 321$.

Note that if a reduced Rothe pipe dream of $u$ with exit permutation $v$ exists, it is not necessarily unique.
For the class of reduced Rothe pipe dreams of $u$ with exit permutation $v$, we will choose a preferred representative as follows. 

\begin{definition}\label{defn.FPP}
Two pipes in a reduced Rothe pipe dream are said to \defn{touch} if they traverse through the same box via an elbow tile.
Let \defn{$\FPP(u,v)$} denote the reduced Rothe pipe dream of $u$ with exit permutation $v$ such that if a pair of pipes cross, then they do not touch again (southeast of the crossing pipes).
We call this a \defn{flag positroid pipe dream}.
Let $\setFPP(n)$ denote the set of $n\times n$ flag positroid pipe dreams.
\end{definition}
In other words, every cross tile in a flag positroid pipe dream is justified to the southeast direction.

\begin{figure}[ht!]
\begin{tikzpicture}[scale=0.5]
\begin{scope}[xshift=0]
	\vertex[color=gray!40](v1) at (0.5, 2.5) {};
	\vertex[color=gray!40](v2) at (1.5, 1.5) {};
	\vertex[color=gray!40](v3) at (2.5, 0.5) {};
	
    \draw [black,thick,domain=180:270] plot ({1+.5*cos(\x)}, {3+.5*sin(\x)});
    \draw [black,thick,domain=180:270] plot ({2+.5*cos(\x)}, {2+.5*sin(\x)});
    \draw [black,thick,domain=180:270] plot ({3+.5*cos(\x)}, {1+.5*sin(\x)});
    
    \cross[thick](2,3);
    \elbow[thick](3,3);
    \elbow[thick](3,2); 

	\draw[very thin, color=blue!50] (0,0) grid (3,3);
	
	\node at (0.5,3.5) {\color{black}{\tiny$1$}};
	\node at (1.5,3.5) {\color{black}{\tiny$2$}};
	\node at (2.5,3.5) {\color{black}{\tiny$3$}};		
    \node at (3.5,2.5) {\color{black}{\tiny$3$}};
    \node at (3.5,1.5) {\color{black}{\tiny$1$}};    
    \node at (3.5,0.5) {\color{black}{\tiny$2$}};
\end{scope}

\begin{scope}[xshift=180]
	\vertex[color=gray!40](v1) at (0.5, 2.5) {};
	\vertex[color=gray!40](v2) at (1.5, 1.5) {};
	\vertex[color=gray!40](v3) at (2.5, 0.5) {};

    \draw [black,thick,domain=180:270] plot ({1+.5*cos(\x)}, {3+.5*sin(\x)});
    \draw [black,thick,domain=180:270] plot ({2+.5*cos(\x)}, {2+.5*sin(\x)});
    \draw [black,thick,domain=180:270] plot ({3+.5*cos(\x)}, {1+.5*sin(\x)});
    
    \elbow[thick](2,3);
    \elbow[thick](3,3);
    \cross[thick](3,2); 

	\draw[very thin, color=blue!50] (0,0) grid (3,3);
	
	\node at (0.5,3.5) {\color{black}{\tiny$1$}};
	\node at (1.5,3.5) {\color{black}{\tiny$2$}};
	\node at (2.5,3.5) {\color{black}{\tiny$3$}};		
    \node at (3.5,2.5) {\color{black}{\tiny$3$}};
    \node at (3.5,1.5) {\color{black}{\tiny$1$}};    
    \node at (3.5,0.5) {\color{black}{\tiny$2$}};
    
    \node at (1.5, -1) {\footnotesize$\FPP(123,312)$};
\end{scope}   

\end{tikzpicture}
\caption{These two reduced Rothe pipe dreams have $u=123$ and $v=312$.  Only the one on the right is a flag positroid pipe dream.
}
\label{fig.notFPP}
\end{figure}
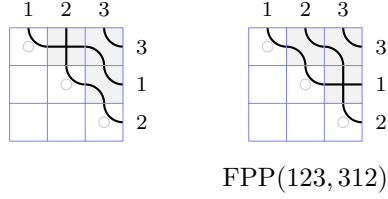

In analogy with $\Le$-diagrams (Definition~\ref{defn.Le_diagram}), there is an alternative way to visualize a flag positroid pipe dream $\FPP(u,v)$ as a $01$-filling of $\Rothe(u)$ in which each cross tile corresponds with $0$ and each elbow tile corresponds with $1$.
Pivot elbows are marked by an open circle.
(In Proposition~\ref{prop:staircase-equivalence}, we will show that a flag positroid pipe dream has a characterization as a certain pattern-avoiding filling.)

\begin{example}\label{eg.uv1}
The left side of Figure~\ref{fig.FPP} shows the flag positroid pipe dream $\FPP(u,v)$ of $u=5316274$ with exit permutation $v=6735142$.
The pipes are labeled $1$ through $7$ in increasing order across the top of the array, and the exit permutation is the order of the pipes on the right boundary from the top to the bottom.
The right side of the figure shows the equivalent $01$-filling of $\FPP(u,v)$.
\end{example}

\begin{figure}[ht!]
\begin{tikzpicture}
\begin{scope}[scale=.5, xshift=0]
	\vertex[color=gray!40](w1) at (4.5, 6.5) {};
	\vertex[color=gray!40](w2) at (2.5, 5.5) {};
	\vertex[color=gray!40](w3) at (0.5, 4.5) {};
	\vertex[color=gray!40](w4) at (5.5, 3.5) {};
	\vertex[color=gray!40](w5) at (1.5, 2.5) {};
	\vertex[color=gray!40](w6) at (6.5, 1.5) {};
	\vertex[color=gray!40](w7) at (3.5, 0.5) {};

        \draw [black,thick,domain=180:270] plot ({5+.5*cos(\x)}, {7+.5*sin(\x)});
        \draw [black,thick,domain=180:270] plot ({3+.5*cos(\x)}, {6+.5*sin(\x)});
        \draw [black,thick,domain=180:270] plot ({1+.5*cos(\x)}, {5+.5*sin(\x)});
        \draw [black,thick,domain=180:270] plot ({6+.5*cos(\x)}, {4+.5*sin(\x)});
        \draw [black,thick,domain=180:270] plot ({2+.5*cos(\x)}, {3+.5*sin(\x)});
        \draw [black,thick,domain=180:270] plot ({7+.5*cos(\x)}, {2+.5*sin(\x)});
        \draw [black,thick,domain=180:270] plot ({4+.5*cos(\x)}, {1+.5*sin(\x)});
        \draw [black,thick] (0.5,5) -- (0.5, 7);
        \draw [black,thick] (1.5,3) -- (1.5, 4); 
            \draw [black,thick] (1.5,5) -- (1.5, 7);
        \draw [black,thick] (2.5,6) -- (2.5, 7);
        \draw [black,thick] (3.5,1) -- (3.5, 2);
            \draw [black,thick] (3.5,3) -- (3.5, 4);
            \draw [black,thick] (3.5,6) -- (3.5, 7);
        \draw [black,thick] (4,0.5) -- (7,0.5);
        \draw [black,thick] (2,2.5) -- (3,2.5);
            \draw [black,thick] (4,2.5) -- (6,2.5);
        \draw [black,thick] (2,4.5) -- (3,4.5);
            \draw [black,thick] (4,4.5) -- (5,4.5);
        \draw [black,thick] (4,5.5) -- (5,5.5);

        \cross[thick](4,3); 
        \cross[thick](7,3);
        \cross[thick](6,5);
        \cross[thick](7,5);
        \cross[thick](7,7);
        
        \elbow[thick](7,4);
        \elbow[thick](2,5); 
        \elbow[thick](4,5); 
        \elbow[thick](4,6);
        \elbow[thick](6,6);
        \elbow[thick](7,6);
        \elbow[thick](6,7);
        
	\draw[very thin, color=blue!50] (0,0) grid (7,7);

	\node at (.5,7.5) {\color{black}{\tiny$1$}};
	\node at (1.5,7.5) {\color{black}{\tiny$2$}};
	\node at (2.5,7.5) {\color{black}{\tiny$3$}};
	\node at (3.5,7.5) {\color{black}{\tiny$4$}};
	\node at (4.5,7.5) {\color{black}{\tiny$5$}};
	\node at (5.5,7.5) {\color{black}{\tiny$6$}};
	\node at (6.5,7.5) {\color{black}{\tiny$7$}};

        \node at (7.5, 6.5) {\color{black}{\tiny$6$}};
        \node at (7.5, 5.5) {\color{black}{\tiny$7$}};
        \node at (7.5, 4.5) {\color{black}{\tiny$3$}};
        \node at (7.5, 3.5) {\color{black}{\tiny$5$}};
        \node at (7.5, 2.5) {\color{black}{\tiny$1$}};
        \node at (7.5, 1.5) {\color{black}{\tiny$4$}};
        \node at (7.5, 0.5) {\color{black}{\tiny$2$}};
\end{scope}
\begin{scope}[scale=.5, xshift=280]
	\vertex[color=gray!40](w1) at (4.5, 6.5) {};
	\vertex[color=gray!40](w2) at (2.5, 5.5) {};
	\vertex[color=gray!40](w3) at (0.5, 4.5) {};
	\vertex[color=gray!40](w4) at (5.5, 3.5) {};
	\vertex[color=gray!40](w5) at (1.5, 2.5) {};
	\vertex[color=gray!40](w6) at (6.5, 1.5) {};
	\vertex[color=gray!40](w7) at (3.5, 0.5) {};

        \draw[fill, color=gray!10] (3,2) rectangle (4,3);
        \draw[fill, color=gray!10] (6,2) rectangle (7,7);
        \draw[fill, color=gray!10] (1,4) rectangle (2,5);
        \draw[fill, color=gray!10] (3,4) rectangle (4,6);
        \draw[fill, color=gray!10] (5,4) rectangle (6,7);

        \node[] at (4.5,6.5){\footnotesize$\circ$};
        \node[] at (2.5,5.5){\footnotesize$\circ$};
        \node[] at (0.5,4.5){\footnotesize$\circ$};
        \node[] at (5.5,3.5){\footnotesize$\circ$};
        \node[] at (1.5,2.5){\footnotesize$\circ$};
        \node[] at (6.5,1.5){\footnotesize$\circ$};
        \node[] at (3.5,0.5){\footnotesize$\circ$};

        \node[] at (3.5,2.5){$0$};
        \node[] at (6.5,2.5){$0$};
        \node[] at (5.5,4.5){$0$};
        \node[] at (6.5,4.5){$0$};
        \node[] at (6.5,6.5){$0$};

        \node[] at (6.5,3.5){$1$};
        \node[] at (1.5,4.5){$1$};
        \node[] at (3.5,4.5){$1$};
        \node[] at (3.5,5.5){$1$};
        \node[] at (5.5,5.5){$1$};
        \node[] at (6.5,5.5){$1$};
        \node[] at (5.5,6.5){$1$};
                
	\draw[very thin, color=blue!50] (0,0) grid (7,7);
\end{scope}
\end{tikzpicture}
\caption{Two representations of the flag positroid pipe dream $\FPP(u,v)$ with $u=5316274$ and exit permutation $v=6735142$, one as a filling with pipes, the other as a $01$-filling.
}
\label{fig.FPP}
\end{figure}
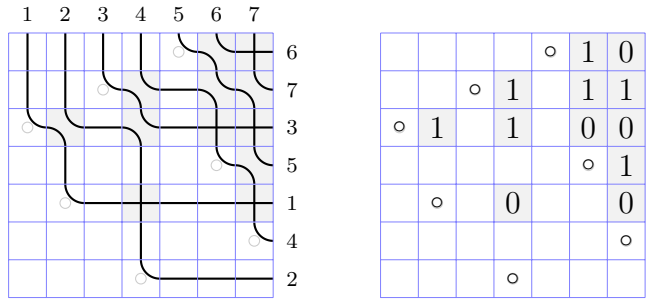

\begin{remark}
Up to a $90^\circ$ rotation, an FPP can be thought of as a generalization of the pipe dream diagrams defined by Bergeron and Billey~\cite{BB93} in their study of Schubert polynomials.  
The original pipe dreams only allowed tiles in a triangular array (i.e. $u$ is the identity permutation).
An FPP can also be interpreted as a variation of the bumpless pipe dreams defined by Lam, Lee, and  Shimozono~\cite{LLS21} in their study of double Schubert polynomials.
In comparison, the bumpless pipe dreams forbid the occurrence of the (double) elbow tile, but allow the use of the southwest single elbow tile.
Lastly, Knutson and Zinn-Justin~\cite{KZ24} defined generic pipe dreams which, up to rotation, allow the use of the single southwest single elbow tile as well.
\end{remark}

\subsection{Intervals in Bruhat order}\label{subsec.Bruhat}
The \defn{Bruhat order} on $\fS_n$ is the partial order where $u\leq v$ if there is a reduced word for $v$ that contains a reduced word for $u$ as a subword.
While this definition will be useful for us later on, at the moment we will instead use the following characterization of the Bruhat order (see Manivel~\cite[Proposition 2.1.11]{Man01} or Bj\"orner--Brenti~\cite[Theorem 2.6.3]{BB05}).

For $u\in \fS_n$, the \defn{key} of $u$, denoted $K(u)$, is a semistandard filling of the Young diagram of shape $\delta_n=(n-1, \ldots, 2,1)$ such that its $j$-th column consists of $u_1,\ldots, u_{n-j}$ in increasing order.
Then $u\leq v$ in Bruhat order if and only if $K(u) \leq K(v)$ entrywise.
\begin{example}
The keys for $u=5316274$ and $v=6735142$ are
\ytableausetup{centertableaux, smalltableaux}
\[K(u)=
    \begin{ytableau}
    1&1&1&1&3&5\\
    2&2&3&3&5\\
    3&3&5&5\\
    5&5&6\\
    6&6\\
    7
    \end{ytableau}
\qquad
K(v)=
    \begin{ytableau}
    1&1&3&3&6&6\\
    3&3&5&6&7\\
    4&5&6&7\\
    5&6&7\\
    6&7\\
    7
    \end{ytableau}.
\]
Each entry in $K(u)$ is less than or equal to each entry in $K(v)$, so $u\leq v$ in Bruhat order.
\end{example}

\begin{proposition} \label{prop.key}
If there exists a Rothe pipe dream of $u$ with exit permutation $v$, then $u\leq v$ in Bruhat order.
\end{proposition}
\begin{proof}
Let $R$ be a Rothe pipe dream of $u$.
We proceed by induction on the number of elbows in $D$.
If every box in $\Rothe(u)$ is tiled by a cross, then its exit permutation is $u$.
Otherwise, suppose $R$ has exit permutation $u'$ with $u\leq u'$ in Bruhat order, and the pipes $a$ and $b$ cross in $R$.  
Assume without loss of generality that $a<b$.
Since $R$ is reduced, then necessarily the horizontal pipe is $a$ and the vertical pipe is $b$.
Replacing this cross tile with an elbow yields a flag positroid pipe dream of $u$ with exit permutation $v=(a\,b)u'$. 
If $u_j'=a$ and $u_k'=b$, then the effect of the exchanging the tile in $R$ on the key of $u'$ is that from columns $n-j+1$ to $n-i$, the $a$'s are replaced by $b$'s, giving the key of $v$.
Since $a<b$, and keys are semistandard fillings, it follows that $K(u')\leq K(v)$ entrywise, so $u'\leq v$.
By induction, we have $u \leq u' \leq v$ in Bruhat order.
\end{proof}

A consequence of Proposition~\ref{prop.key} is that if there is a Rothe pipe dream of $u$ with exit permutation $v$ containing $m$ elbows, then $v = r_m \cdots r_1 u$, where the $r_i$ are the transpositions corresponding to the coinversions of $u$ that are tiled by elbows in the pipe dream, traversed along rows from right to left and bottom to top.
Continuing Example~\ref{eg.uv1}, we see that
\[v= (5\,6)(3\,4)(3\,6)(3\,7)(1\,2)(1\,4)(6\,7)\cdot u.\]

Conversely, we now show that if $u\leq v$ in Bruhat order then we can construct a flag positroid pipe dream of $u$ with exit permutation $v$ via the following algorithm.

\begin{definition}[Algorithm for producing a flag positroid pipe dream for $u\leq v$]\label{defn.RPD_algo}
Let $u,v\in\fS_n$ with $u\leq v$ in Bruhat order. 
Begin with the Rothe diagram of $u$ as a subset of boxes in an $n\times n$ array.
The tiling of the boxes in the array which are not in $\Rothe(u)$ is completely specified by Definition~\ref{defn.rrpd}.
\begin{enumerate}
    \item[(i)] Set $i=n$ and $\col = (0,\ldots, 0) \in \mathbb{Z}_{\geq0}^n$. 
    \item[(ii)] If $i>0$, set $x=v_i$. Otherwise, stop.
    \item[(iii)] If all boxes in the $i$-th row of $\Rothe(u)$ are tiled, then set $\col(u_i)=x$, $i=i-1$, and go to (iii).
    Otherwise, let $j$ be the column index of the rightmost untiled box in the $i$-th row of $\Rothe(u)$. 
    If $(x, \col(j)) \in \Coinv(v)$, then tile box $(i,j)$ with a cross.
    If $(x, \col(j)) \notin \Coinv(v)$, then tile box $(i,j)$ with an elbow and switch the values of $x$ and $\col(j)$.
    \item[(iv)] Repeat (iii).
\end{enumerate}
\end{definition}

Simply put, the boxes of $\Rothe(u)$ are tiled row by row from bottom to top and from right to left.  
At each step of the algorithm, the untiled box $(i,j)$ has incoming pipes from the right boundary (say, labeled $a$) and from the bottom boundary (say, labeled $b$).
If $(a,b) \in \Coinv(v)$ (with $a<b$ implicitly), then tile the box with a cross.
Otherwise, tile the box with an elbow.

After all boxes in the $i$-th row of the pipe dream have been tiled, the $j$-th coordinate of the vector $\col$ records the label of the pipe that is entering the box $(i-1,j)$ from below (and if $\col(j)=0$ then there is not yet a pipe entering the box $(i-1,j)$ from below).
Thus in order to verify that the algorithm at minimum returns a Rothe pipe dream of $u$ with exit permutation $v$, we need to verify the following.

\begin{lemma} \label{lem.validRPD}
Let $u,v\in\fS_n$ with $u\leq v$.   When the algorithm in Definition~\ref{defn.RPD_algo} terminates, the vector $\col = (1,2,\ldots, n)$.
\end{lemma}
\begin{proof}
We will show that at the end of each iteration of the algorithm, the nonzero entries of $\col$ are always in increasing order, and if $\col(j) = a \neq0$ then $j\geq a$ for each $j$.

At the end of the $i$-th iteration of the algorithm, the $n-i+1$ nonzero entries of $\col$ are $\{v_i, \ldots, v_n\}$, and this is the set of labels of the pipes incoming to the $(i-1)$-st row of the array from below.
Likewise, the indices of the nonzero entries of $\col$ are $\{u_i, \ldots, u_n\}$, and this indexes the boxes of $\Rothe(u)$ which must be filled in the $(i-1)$-st row.
Observe that $\{v_i, \ldots, v_n\}$ is the complement of the $(n-i+1)$-st column of the key $K(v)$, and $\{u_i,\ldots, u_n\}$ is the complement of the $(n-i+1)$-st column of $K(u)$.
Since $u\leq v$ in Bruhat order, then $K(u) \leq K(v)$ entrywise, which implies that if we order the set $\{v_i, \ldots, v_n\}$ in increasing order to obtain $a_1< \cdots < a_{n-i+1}$ and order the set $\{u_i, \ldots, u_n\}$ in increasing order to obtain $p_1< \cdots < p_{n-i+1}$, then $p_k\geq a_k$ for each $k$.
Thus, the index of a nonzero entry of $\col$ is greater than or equal to its value.

We now show by induction that the nonzero entries of $\col$ are increasing.
Beginning with $i=n$, the only nonzero entry of $\col$ is $\col(u_n) = v_n$, so the claim holds.
By the induction hypothesis, the $n-i+1$ nonzero entries of $\col$ are in increasing order at the end of the $i$-th iteration of the algorithm.
In the next iteration, $x=v_{i-1}$ is compared against each nonzero entry of $\col$ from right to left: for $j> u_{i-1}$, if $0\neq \col(j)> v_{i-1}$ then $\col(j)$ remains unchanged, otherwise if $0\neq \col(j) < v_{i-1}$ then $\col(j)$ is updated to $v_{i-1}$ and the remaining nonzero entries of $\col$ are successively shifted by one position to the left until $\col(u_{i-1})$ is assigned a value.
Effectively, $v_{i-1}$ is inserted into $\col$ so that its nonzero entries remain in increasing order.

Altogether, the pipes labeled $\{v_{i-1},\ldots, v_n\}$ enter the next row from below in increasing order such that the pipe labeled $v_k$ has not moved past the $k$-th column, so the result follows when the algorithm terminates.
\end{proof}

The following result shows that the Rothe pipe dream produced by the algorithm from Definition~\ref{defn.RPD_algo} is reduced and has all cross tiles justified to the southeast direction and so is in fact a flag positroid pipe dream.
\begin{proposition} \label{prop.algo}
Given $u\leq v$ in Bruhat order, the algorithm in Definition~\ref{defn.RPD_algo} constructs the flag positroid pipe dream $\FPP(u,v)$.   
\end{proposition}
\begin{proof}
The algorithm constructs a pipe dream by extending pipes labeled $v_1,\ldots, v_n$ from the right boundary of the array to the top of the array. 
When the algorithm terminates, the vector $\col$ records the labels of the pipes at the top of the array in order from left to right, so it follows from Lemma~\ref{lem.validRPD} that the algorithm returns a Rothe pipe dream $D$ of $u$ with exit permutation $v$.
It remains to show that no pair of pipes in $D$ cross more than once, and that all cross tiles are justified to the southeast direction.

Let $a<b$. 
Suppose the pair of pipes labeled $a$ and $b$ first cross in box $(i,j)$ in $D$, and as the pipe dream algorithm iterates, the same pair of pipes are about to traverse through the box $(k,l)$, with $k<i$ and $l<j$ necessarily.
Since the pipes already crossed in box $(i,j)$, then the pipe labeled $a$ must enter the box $(k,l)$ from below and the pipe labeled $b$ must enter the box $(k,l)$ from the right, meaning $x=b > a = \col(l)$ is not a coinversion of $v$, and therefore $(k,l)$ is tiled with an elbow.
Therefore, any pair of pipes in $D$ cross at most once.

Lastly, suppose for contradiction that the pair of pipes labeled $a$ and $b$ cross in box $(i,j)$, but there is a box $(k,l)$ that is southeast of $(i,j)$ in which the same pair of pipes touch (in the sense of Definition~\ref{defn.FPP}).  
This is impossible since the algorithm tiles $\Rothe(u)$ row by row from bottom to top and right to left, and the configuration of the pipes in the box $(k,l)$ would have dictated that a cross tile be placed there.
\end{proof}

Given a flag positroid pipe dream $\FPP(u,v)$, we saw as a consequence of Proposition~\ref{prop.key} how the positions of the elbow tiles give an expression for $v$ in terms of multiplying $u$ on the left by transpositions. 
It turns out that the cross tiles also play an important role.

For $u,v\in \fS_n$, define $x=u^{-1}w_0$ and $y=v^{-1}w_0$.
Then $u\leq v$ if and only if $u^{-1} \leq v^{-1}$ if and only if $x=u^{-1}w_0 \geq v^{-1}w_0 = y$, and this means there exists a reduced word for $y$ that is a subword of a reduced word for $x$.
We next show how to obtain these reduced expressions.

For $i=1,\ldots, n-1$, let $s_i$ denote the simple transposition $(i\  i+1)$ in $\fS_n$. 
\begin{lemma} \label{lem.ux}
Let $u\in\fS_n$.
To each box in $\Rothe(u)$ we associate a simple transposition in the following way.  
If a box is in the $i$-th row and the $h$-th column from the right (with respect to the Rothe diagram), we associate $s_{i+h-1}$ to it.
Let $s_{i_1}, s_{i_2}, \ldots, s_{i_m}$ be the sequence of the associated transpositions in $\Rothe(u)$, ordered row by row from bottom to top and from right to left.
Define $x=s_{i_1} s_{i_2} \cdots s_{i_m}$.
Then $ux=w_0$, and the expression for $x$ is reduced.
\end{lemma}
\begin{proof}
Recall that there is a one-to-one correspondence between the boxes in $\Rothe(u)$ and the coinversions of $u$. 
Starting with the lowest nonempty row of $\Rothe(u)$ (say, the $r$-th row of the array), applying the associated sequence of transpositions in that row to $u$ from right to left has the effect of sorting the entries $u_r, u_{r+1}, \ldots, u_{n}$ into decreasing order.
Repeating this process row by row from the bottom to the top then sorts $u_1,\ldots, u_n$ into decreasing order, so $ux=w_0$.
Lastly, there is one simple transposition for each coinversion of $u$, so the given factorization for $x$ is reduced.
\end{proof}

\begin{lemma} \label{lem.vy}
Consider the flag positroid pipe dream $\FPP(u,v)$.
Order the $m$ boxes of $\Rothe(u)$ row by row from bottom to top and from right to left, and suppose the cross tiles in $\FPP(u,v)$ occur in the positions $1\leq p_1< p_2< \cdots <p _l \leq m$ with respect to the linear ordering of the boxes.
Define $y= s_{i_{p_1}} \cdots s_{i_{p_l}}$.  Then $vy=w_0$ and the expression for $y$ is reduced.
\end{lemma}
\begin{proof}
Recall that there is a one-to-one correspondence between the cross tiles in $\Rothe(u)$ and the coinversions of $v$.
With an argument that is completely analogous to the one given in Lemma~\ref{lem.ux}, applying the sequence of simple transpositions associated to the cross tiles sorts $v_1,\ldots, v_n$ into decreasing order, giving $vy =w_0$.  
Since $\FPP(u,v)$ is reduced (no two pipes cross more than once), then there is one simple transposition for each coinversion of $v$, so the given factorization for $y$ is reduced.
\end{proof}

\begin{remark}\label{rem.1213}
Combining Lemmas~\ref{lem.ux} and~\ref{lem.vy}, it follows that the reduced expression for $y$ is a subword of the reduced expression for $x$.
Moreover in the case when $D=\FPP(u,v)$, then the reduced expression for $y$ is lexicographically minimal as a subword of $x$ because the flag positroid pipe dream $\FPP(u,v)$ has all cross tiles justified to the southeast.
\end{remark}

\begin{example}
Figure~\ref{fig.rightaction} shows the Rothe diagram of 
$u=5316274$.
Each box is marked with its corresponding simple transposition.
Then $x=s_5s_6s_4s_3s_4s_5s_6s_2s_3s_4s_1s_2$ is a reduced expression such that $ux =w_0$.

When $v=6735142$, the flag positroid pipe dream $\FPP(u,v)$ has cross tiles in the first, second, fourth, fifth, and eleventh box in $\Rothe(u)$ when traversed along rows from bottom to top and right to left (see Figure~\ref{fig.FPP}).  
Thus $y=s_5s_6s_3s_4s_1$ is a reduced expression such that $vy=w_0$.
\end{example}

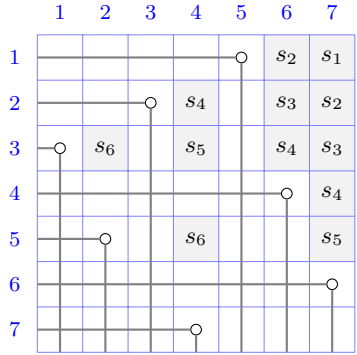
\begin{figure}[ht!]
\begin{tikzpicture}
\begin{scope}[scale=.6, xshift=0]
	\vertex(w1) at (4.5, 6.5) {};
	\vertex(w2) at (2.5, 5.5) {};
	\vertex(w3) at (0.5, 4.5) {};
	\vertex(w4) at (5.5, 3.5) {};
	\vertex(w5) at (1.5, 2.5) {};
	\vertex(w6) at (6.5, 1.5) {};
	\vertex(w7) at (3.5, 0.5) {};

        \draw[fill, color=gray!10] (3,2) rectangle (4,3);
        \draw[fill, color=gray!10] (6,2) rectangle (7,7);
        \draw[fill, color=gray!10] (1,4) rectangle (2,5);
        \draw[fill, color=gray!10] (3,4) rectangle (4,6);
        \draw[fill, color=gray!10] (5,4) rectangle (6,7);

	\draw[very thin, color=blue!50] (0,0) grid (7,7);
	
	\draw[thick, color=gray] (0,6.5)--(w1)--(4.5,0);
	\draw[thick, color=gray] (0,5.5)--(w2)--(2.5,0);
	\draw[thick, color=gray] (0,4.5)--(w3)--(0.5,0);
	\draw[thick, color=gray] (0,3.5)--(w4)--(5.5,0);
	\draw[thick, color=gray] (0,2.5)--(w5)--(1.5,0);
	\draw[thick, color=gray] (0,1.5)--(w6)--(6.5,0);
	\draw[thick, color=gray] (0,0.5)--(w7)--(3.5,0);

	\node[] at (5.5,6.5) {\color{black}{\scriptsize$s_2$}};
	\node[] at (6.5,6.5) {\color{black}{\scriptsize$s_1$}};
    
	\node[] at (3.5,5.5) {\color{black}{\scriptsize$s_4$}};
	\node[] at (5.5,5.5) {\color{black}{\scriptsize$s_3$}};
	\node[] at (6.5,5.5) {\color{black}{\scriptsize$s_2$}};
    
	\node[] at (1.5,4.5) {\color{black}{\scriptsize$s_6$}};
	\node[] at (3.5,4.5) {\color{black}{\scriptsize$s_5$}};
	\node[] at (5.5,4.5) {\color{black}{\scriptsize$s_4$}};
	\node[] at (6.5,4.5) {\color{black}{\scriptsize$s_3$}};
    
	\node[] at (6.5,3.5) {\color{black}{\scriptsize$s_4$}};
    
    \node[] at (3.5,2.5) {\color{black}{\scriptsize$s_6$}};
    \node[] at (6.5,2.5) {\color{black}{\scriptsize$s_5$}};
	
	\node at (-.5,6.5) {\color{blue}{\tiny$1$}};
	\node at (-.5,5.5) {\color{blue}{\tiny$2$}};
	\node at (-.5,4.5) {\color{blue}{\tiny$3$}};
	\node at (-.5,3.5) {\color{blue}{\tiny$4$}};
	\node at (-.5,2.5) {\color{blue}{\tiny$5$}};
	\node at (-.5,1.5) {\color{blue}{\tiny$6$}};
	\node at (-.5,.5) {\color{blue}{\tiny$7$}};
	\node at (.5,7.5) {\color{blue}{\tiny$1$}};
	\node at (1.5,7.5) {\color{blue}{\tiny$2$}};
	\node at (2.5,7.5) {\color{blue}{\tiny$3$}};
	\node at (3.5,7.5) {\color{blue}{\tiny$4$}};
	\node at (4.5,7.5) {\color{blue}{\tiny$5$}};
	\node at (5.5,7.5) {\color{blue}{\tiny$6$}};
	\node at (6.5,7.5) {\color{blue}{\tiny$7$}};
\end{scope}
\end{tikzpicture}
\caption{The correspondence between the boxes of the Rothe diagram of $u=5316274$ and the set of simple transpositions which encode a reduced expression for $u^{-1}w_0$.}
\label{fig.rightaction}
\end{figure}

\begin{theorem}\label{thm.main_FPP}
Let $u,v\in \fS_n$.
The flag positroid pipe dream $\FPP(u,v)$ exists if and only if $u\leq v$ in Bruhat order.
Moreover, the length of the interval $[u,v]$ is the number of elbows in $\FPP(u,v)$.
\end{theorem}
\begin{proof}
The first claim follows from Propositions~\ref{prop.key} and~\ref{prop.algo}.

To verify the second claim, the length of the interval $[u,v]$ is $\ell(v)-\ell(u)$ and it follows from Lemmas~\ref{lem.ux} and~\ref{lem.vy} that
\begin{align*}
\ell(v) - \ell(u) &= \ell(u^{-1}w_0) - \ell(v^{-1}w_0) \\
&= |\Rothe(u)| - \#\,\textrm{cross tiles in $\Rothe(u)$}\\
&= \#\,\textrm{elbow tiles in $\Rothe(u)$}.    
\end{align*}
\end{proof}

\begin{corollary}
Let $u \leq v\in\fS_n$.
The map $[u,v] \mapsto \FPP(u,v)$ is a bijection between intervals in Bruhat order and the set $\setFPP(n)$ of flag positroid pipe dreams.
\qed
\end{corollary}

\begin{example}
Continuing Example~\ref{eg.uv1}, we conclude that the interval $[u,v]$ in Bruhat order has length $7$ because $\FPP(u,v)$ has seven elbow tiles.
\end{example}

\subsection{As a pattern avoiding filling}
In this section, we give an equivalent definition of FPPs as a tiling that avoids a certain configuration of tiles.
This is analogous to the definition for $\Le$-diagrams/$\Le$-pipe dreams, so we first review its definition.

\begin{definition}[{\cite[Definition 6.1]{Postnikov06}}]\label{defn.Le_diagram}
Let $\lambda= (\lambda_1,\ldots, \lambda_k)$ be a partition.
A \defn{$\Le$-diagram} of shape $\lambda$ is a $01$-filling of a Young diagram $Y$ of shape $\lambda$ such that if there are three boxes in $Y$ in positions $(i,j')$, $(i',j)$, $(i',j')$ with $i<i'$ and $j<j'$, and $(i,j')$ and $(i',j)$ contain $1$, then $(i',j')$ must contain $1$.
The configuration of these boxes forms a `$\Le$' shape.
\end{definition}
Instead of a $01$-filling of a Young diagram, we may alternatively visualize a $\Le$-diagram in terms of pipe dreams (see~\cite[Section 19]{Postnikov06}) in which each $0$ corresponds with a cross tile and each $1$ corresponds with an elbow tile.
Thus an alternative definition of a $\Le$-diagram is as a pipe dream that forbids a certain configuration of tiles.
A \defn{$\Le$ pattern} is a set of three tiles such that $(i,j')$ and $(i',j)$ are elbows and $(i',j')$ is a cross, with $i<i'$ and $j<j'$.
Then a $\Le$-diagram is a tiling of a partition shape that is $\Le$-free.
We will refer to this as a \defn{$\Le$-pipe dream}.
See Figure~\ref{fig.le1} for an example.

\begin{figure}[ht!]
\begin{center}
\begin{tikzpicture}
\begin{scope}[scale=0.5]
    \elbow[thick](0,0);
    \cross[thick](2,0);
    \elbow[thick](2,2);	
    \draw[color=blue!30] (-1,-1) grid (2,2);
\node at (-2,-.5) {\textcolor{blue}{$i'$}};
\node at (-2,1.5) {\textcolor{blue}{$i$}};
\node at (-0.5,2.75) {\textcolor{blue}{$j$}};
\node at (1.5,2.75) {\textcolor{blue}{$j'$}};
\end{scope}
\end{tikzpicture}
\end{center}
\caption{The forbidden $\Le$ pattern in a $\Le$-pipe dream.}
\label{fig.Lepattern}
\end{figure}
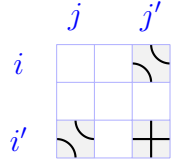

\begin{definition}\label{defn.gammapatt}
A \defn{$\Gamma$ pattern} is a set of three tiles $(i,j)$, $(i',j)$, and $(i,j')$ with $i<i'$ and $j<j'$ such that $(i',j)$ and $(i,j')$ are elbows (possibly pivot elbows) and $(i,j)$ is a cross, and the pipe that passes horizontally through $(i,j)$ exits in a row $i'' \geq i'$. 
A Rothe pipe dream is \defn{$\Gamma$-free} if it is does not contain any $\Gamma$ patterns.
\end{definition}
See Figure~\ref{fig.gammapatt} for an illustration. 
The dashed line in the figure denotes the pipe that runs horizontally through the cross tile in box $(i,j)$ in the $\Gamma$ pattern and which exits the reduced flag positroid pipe dream in a row $i'' \geq i'$.
We emphasize that the case where such a pattern occurs with $i'' < i'$ results in a pattern of tiles that is allowed for our purposes.
See Figure~\ref{fig.OKgamma} for example.

\begin{figure}[ht!]
\centering
\begin{tikzpicture}
\begin{scope}[scale=0.5]
    \elbow[thick](0,1);
    \cross[thick](0,3);
    \elbow[thick](4,3);	

	\draw[dashed, thick] (0,2.5) -- (1.75,2.5);
	
	\draw[dashed, thick, domain=180:270] plot ({5+.5*cos(\x)}, {.5*sin(\x)});
	\draw[dashed, thick, domain=0:90] plot ({5+.5*cos(\x)}, {-1+.5*sin(\x)});
	\draw[dashed, thick, domain=180:270] plot ({6+.5*cos(\x)}, {-1+.5*sin(\x)});
	\draw[dashed, thick] (6,-1.5) -- (7, -1.5);
    \draw[color=blue!30] (-1,-2) grid (7,3);
    
\node at (-2,-1.5) {\textcolor{blue}{$i''$}};
\node at (-2,.5) {\textcolor{blue}{$i'$}};
\node at (-2,2.5) {\textcolor{blue}{$i$}};
\node at (-0.5,3.75) {\textcolor{blue}{$j$}};
\node at (3.5,3.75) {\textcolor{blue}{$j'$}};
\end{scope}
\end{tikzpicture}
	\caption{The forbidden $\Gamma$ pattern in a $\FPP$.
    }
    \label{fig.gammapatt}
\end{figure}
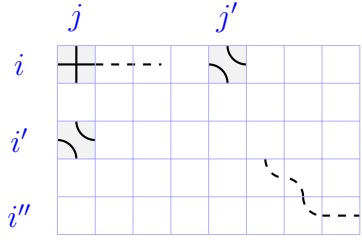
\begin{figure}[ht!]
\begin{tikzpicture}
\begin{scope}[scale=.5, xshift=0]
        \draw [black,thick,domain=180:270] plot ({2+.5*cos(\x)}, {4+.5*sin(\x)});
        \draw [black,thick,domain=180:270] plot ({4+.5*cos(\x)}, {3+.5*sin(\x)});
        \draw [black,thick,domain=180:270] plot ({1+.5*cos(\x)}, {2+.5*sin(\x)});
        \draw [black,thick,domain=180:270] plot ({3+.5*cos(\x)}, {1+.5*sin(\x)});
        \draw [black,thick] (0.5,2) -- (0.5, 4);
        \draw [black,thick] (2.5,2) -- (2.5, 3);
        
        \draw [black,thick] (3,0.5) -- (4,0.5);
        \draw [black,thick] (1,1.5) -- (2,1.5);
        \draw [black,thick] (3,1.5) -- (4,1.5);

        \cross[thick](3,4);
        \elbow[thick](4,4);
        \elbow[thick](3,2); 
        
		\draw[very thin, color=blue!50] (0,0) grid (4,4);

		\node at (0.5,4.5) {\color{black}{\tiny$1$}};
		\node at (1.5,4.5) {\color{black}{\tiny$2$}};
		\node at (2.5,4.5) {\color{black}{\tiny$3$}};
		\node at (3.5,4.5) {\color{black}{\tiny$4$}};

        \node at (4.5, 3.5) {\color{black}{\tiny$4$}};
        \node at (4.5, 2.5) {\color{black}{\tiny$2$}};
        \node at (4.5, 1.5) {\color{black}{\tiny$3$}};
        \node at (4.5, 0.5) {\color{black}{\tiny$1$}};
\end{scope}
\end{tikzpicture}
    \caption{The pipe dream $\FPP(2413,4231)$ is valid; there are no $\Gamma$ patterns here because pipe $2$ exits the diagram in a row above the elbow tile in box $(3,3)$.
    }
    \label{fig.OKgamma}
\end{figure}
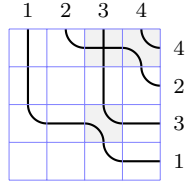

Heuristically, a $\Gamma$ pattern is a $\Le$ pattern rotated $180^\circ$, unless the horizontal pipe through the corner of the $\Gamma$ pattern exits before it reaches the row containing the bottom box of the shape. 
While this definition connects FPPs to $\Le$-pipe dreams directly, the following characterization will be more convenient for our purposes.

\begin{definition}
Given a pipe in a Rothe pipe dream, the \defn{subarea} of the pipe is the set of boxes that lie below and in the same column as any horizontal section of the pipe but do not lie below the exit row of the pipe.
\end{definition}

\begin{lemma}\label{lem:gamma-free-is-subarea-condition}
Let $R$ be a reduced Rothe pipe dream. 
Then $R$ is $\Gamma$-free if and only if there are no elbow tiles or pivot elbow tiles in the subarea of any pipe.
\end{lemma}
\begin{proof}
By the definition of a $\Gamma$ pattern, the elbow or pivot elbow tile at the bottom of the $\Gamma$ shape is in the subarea of the pipe passing horizontally through the box at the corner of the shape.
    
Conversely, suppose there is an elbow or pivot elbow tile in the subarea of some pipe in $R$, say in box $(i',j)$. 
By definition, there is a pipe passing horizontally through a cross tile above this elbow tile, say at $(i,j)$ with $i<i'$, and the exit row of this pipe is $i'' \geq i'$. 
Moreover, since $i<i'\leq i''$, there must be an elbow tile to the right of the cross tile in $(i,j)$ to direct the pipe in question downward, say at box $(i,j')$ with $j<j'$. 
Hence, the tiles in $(i,j)$, $(i',j)$ and $(i,j')$ form a $\Gamma$ pattern.
\end{proof}

\begin{proposition}\label{prop:staircase-equivalence}
Let $R$ be a reduced Rothe pipe dream. 
Then $R=\FPP(u,v)$ for some $u\leq v$ if and only if $R$ is $\Gamma$-free.
\end{proposition}
\begin{proof}
First suppose that $R$ is $\Gamma$-free.
Let $p$ and $q$ be two pipes in $R$ which first cross in position $(i,j)$, say with $p$ being the pipe that continues downward and $q$ being the pipe that continues to the right and exits in the row $i''$. 
We will show that the pipes $p$ and $q$ do not cross or touch again.
In the case $i''=i$ or $j=n$, the claim follows readily, so we may assume $i''> i$ and $j<n$.
After the crossing at $(i,j)$, the rest of the pipe $q$ exists in the rectangle $[i,i'']\times [j+1, n]$.
Since the pipe $q$ exits in row $i''>i$, then there must be an elbow tile at some box $(i,j')$ for some $j'\in [j+1,n]$.
$R$ is $\Gamma$-free so there are no elbow or pivot elbow tiles in the boxes $(i', j)$ for any $i < i' \leq i''$.
This implies that $p$ cannot enter the given rectangle, so $p$ and $q$ do not touch or cross again.
Therefore, $R$ is reduced and has all cross tiles justified to the southeast direction, so it is an $\FPP(u,v)$ for some permutations $u\leq v$.

Now suppose $R$ is not $\Gamma$-free.
By Lemma~\ref{lem:gamma-free-is-subarea-condition}, $R$ has an elbow or a pivot elbow tile in the subarea of some pipe.
Choose the rightmost column $j$ containing an elbow or pivot elbow tile in the subarea of some pipe, and consider the topmost of these elbows in this column, say in row $i'$.
Let $(i,j)$ be the position of the lowest cross tile such that our selected elbow is in its subarea, and let $p$ and $q$ be the pipes passing through this cross vertically and horizontally, respectively.
Since we have chosen the topmost violating elbow, there are no elbows in this column between $i'$ and $i$, so $p$ is turned to the right by our elbow in $(i',j)$.
Now let $j'$ be the column in which $q$ first enters row $i'$.

Since $j$ is the rightmost column containing a violating elbow, there can be no elbows between $(i',j)$ and $(i', j')$, except possibly below a vertical section of $q$.
If there are no such elbows, then $p$ continues to the right in row $i'$ until it touches or crosses $q$, and $R$ would not be an $\FPP$.
Otherwise, suppose there is an elbow or pivot elbow tile under a vertical section of $q$ . 
Consider the leftmost of these elbows, say in column $j''$, and let $i''$ be the row of the lowest elbow above $(i', j'')$.
Then by all of our assumptions, we have identified a pipe $r$ that crosses $p$ at $(i'', j)$, turns downward at $(i'', j'')$, then crosses or touches $p$ again at $(i', j'')$, and $R$ is not an $\FPP$.
\end{proof}

\subsection{A parallel with \texorpdfstring{$\Le$}{Le}-diagrams}\label{subsec.le_parallel}
$\Le$-diagrams were introduced by Postnikov as a set of combinatorial diagrams that index positroidal cells in the totally nonnegative Grassmannian (see~\cite[Theorem 6.5]{Postnikov06}).
In this section we explain in Theorem~\ref{thm.rpd_is_le} how the flag positroid pipe dreams $\FPP(u,v)$ may be seen as a generalization of $\Le$-diagrams (up to a $180^\circ$ rotation, or equivalently, left multiplication by $w_0$), and so Theorem~\ref{thm.main_FPP} may be viewed as a generalization of Postnikov's result for $\Le$-diagrams (see~\cite[Theorem 19.1]{Postnikov06}).

Throughout the rest of this article, we will use the terms `$\Le$-diagram' and `$\Le$-pipe dream' interchangeably (recall Definition~\ref{defn.Le_diagram}).

A permutation $w\in \fS_n$ is \defn{Grassmannian} if it has at most one descent (i.e. there is at most one index $k$ such that $w_k > w_{k+1}$).
Grassmannian permutations are in bijection with partitions contained in the rectangle $(n-k)^k$, which can be seen from the following combinatorial construction.
For the Grassmannian permutation $w=w_1\cdots w_n$ with unique descent at position $k$, its associated partition $\lambda$ is obtained by writing $w_1, \ldots , w_k$ along the left side from the bottom to the top, and $w_{k+1}, \ldots,  w_n$ along the top from left to right.  
Then $\lambda_i$ is the number of inversion pairs $w_{k-i+1} > w_{k+j}$ for $k-i+1 < k+j$, in $w$.
See Postnikov~\cite[Section 19]{Postnikov06}, and also the left side of Figure~\ref{fig.le1} below.

\begin{theorem}[{Postnikov~\cite[Theorem 19.1]{Postnikov06}}]
\label{thm:postnikov19}
Let $w$ be a Grassmannian permutation with corresponding partition $\lambda$.
There is a bijection between $\textrm{\Le}$-pipe dreams of shape $\lambda$ and permutations $t\in \fS_n$ such that $t\leq w$ in Bruhat order.
Moreover, the length of the interval $[t,w]$ is the number of elbows in the $\Le$-pipe dream. 
\end{theorem}
Let \defn{$\Le(t,w)$} denote the $\Le$-pipe dream that corresponds with the pair of permutations $t\leq w$ in the bijection of Theorem~\ref{thm:postnikov19}.

\begin{remark} \label{rem.le_s}
We emphasize a particular feature of the above theorem.
If $w$ is a Grassmannian permutation with corresponding partition $\lambda$, then $\lambda$ specifies a choice of a reduced word for $w$ of length $|\lambda|$ in the following way.
If a box in Young's diagram of $\lambda$ is in the $i$-th row and $j$-th column, we associate the simple transposition $s_{k+j-i}$ to it.

Furthermore, the $\Le$-pipe dream $\Le(t,w)$ of shape $\lambda$ specifies a reduced word for $t$ that is lexicographically maximal as a subword of the chosen reduced word of $w$ as specified by the cross tiles in the $\Le$-pipe dream, and the length of $t$ is the number of crosses in the $\Le$-pipe dream.
\end{remark}

\begin{example}\label{eg.uw}
For the Grassmannian permutation $w=461235$ with unique descent at position $k=2$, its associated partition is $\lambda = (4,3)$, as shown on the left of Figure~\ref{fig.le1}.  

The associated $\Le$-pipe dream is shown in the center.
Following the pipes from the northwest boundary of the pipe dream to the southeast boundary, we see that the $\Le$-diagram corresponds with the permutation $t=143256$.

By reading the diagram on the right of Figure~\ref{fig.le1} row by row from the bottom to top and right to left, we see that the shape $\lambda$ specifies the reduced word $s_3s_2s_1s_5s_4s_3s_2$ for $w$, and the positions of the cross tiles in the $\Le$-diagram specify the reduced word $s_2s_3s_2$ for $t$ that is lexicographically maximal amongst all reduced words for $t$ that are subwords of $w$. 
There are four elbows in the pipe dream, and indeed $\ell(w)-\ell(t) = 7-3= 4$.
\end{example}

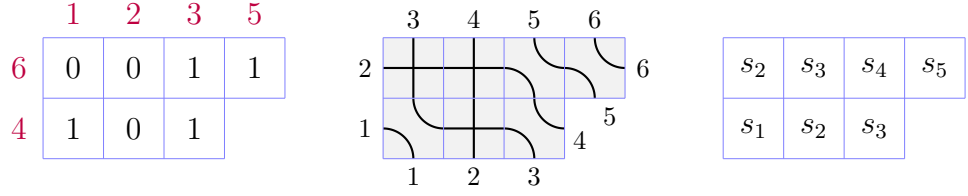
\begin{figure}[ht!]
\begin{tikzpicture}
\begin{scope}[scale=.8, xshift=0]
 	\node at (0.5, 0.5) {$1$};
 	\node at (2.5, 0.5) {$1$};
        \node at (2.5, 1.5) {$1$};
        \node at (3.5, 1.5) {$1$};
        \node at (1.5, 0.5) {$0$};
        \node at (0.5, 1.5) {$0$};
        \node at (1.5, 1.5) {$0$};

	\draw[very thin, color=blue!50] (0,0) -- (3,0);
        \draw[very thin, color=blue!50] (0,1) -- (4,1);
        \draw[very thin, color=blue!50] (0,2) -- (4,2);
        \draw[very thin, color=blue!50] (0,0) -- (0,2);
        \draw[very thin, color=blue!50] (1,0) -- (1,2);
        \draw[very thin, color=blue!50] (2,0) -- (2,2);
        \draw[very thin, color=blue!50] (3,0) -- (3,2);
        \draw[very thin, color=blue!50] (4,1) -- (4,2);
    	
	\node at (-.4,1.5) {\color{purple}{$6$}};
	\node at (-.4,0.5) {\color{purple}{$4$}};

	\node at (.5,2.4) {\color{purple}{$1$}};
	\node at (1.5,2.4) {\color{purple}{$2$}};
	\node at (2.5,2.4) {\color{purple}{$3$}};
	\node at (3.5,2.4) {\color{purple}{$5$}};
\end{scope}
\begin{scope}[scale=.8, xshift=160]
  	
	\node at (-.3,0.5) {\color{black}{\footnotesize$1$}};
	\node at (-.3,1.5) {\color{black}{\footnotesize$2$}};
	\node at (.5,2.3) {\color{black}{\footnotesize$3$}};
	\node at (1.5,2.3) {\color{black}{\footnotesize$4$}};
	\node at (2.5,2.3) {\color{black}{\footnotesize$5$}};
	\node at (3.5,2.3) {\color{black}{\footnotesize$6$}};

        \node at (.5,-.3) {\color{black}{\footnotesize$1$}};
	\node at (1.5,-.3) {\color{black}{\footnotesize$2$}};
	\node at (2.5,-.3) {\color{black}{\footnotesize$3$}};
	\node at (3.25,.3) {\color{black}{\footnotesize$4$}};
	\node at (3.75,0.7) {\color{black}{\footnotesize$5$}};
	\node at (4.3,1.5) {\color{black}{\footnotesize$6$}};

        \elbow[thick](1,1);
        \elbow[thick](3,1);
        \elbow[thick](3,2);
        \elbow[thick](4,2);
        \cross[thick](2,1);
        \cross[thick](1,2);
        \cross[thick](2,2);

	\draw[very thin, color=blue!50] (0,0) -- (3,0);
        \draw[very thin, color=blue!50] (0,1) -- (4,1);
        \draw[very thin, color=blue!50] (0,2) -- (4,2);
        \draw[very thin, color=blue!50] (0,0) -- (0,2);
        \draw[very thin, color=blue!50] (1,0) -- (1,2);
        \draw[very thin, color=blue!50] (2,0) -- (2,2);
        \draw[very thin, color=blue!50] (3,0) -- (3,2);
        \draw[very thin, color=blue!50] (4,1) -- (4,2);
\end{scope}

\begin{scope}[scale=.8, xshift=320]
 	\node at (0.5, 0.5) {$s_1$};
        \node at (0.5, 1.5) {$s_2$};
        \node at (1.5, 0.5) {$s_2$};
        \node at (1.5, 1.5) {$s_3$};
 	\node at (2.5, 0.5) {$s_3$};
        \node at (2.5, 1.5) {$s_4$};
        \node at (3.5, 1.5) {$s_5$};

	\draw[very thin, color=blue!50] (0,0) -- (3,0);
        \draw[very thin, color=blue!50] (0,1) -- (4,1);
        \draw[very thin, color=blue!50] (0,2) -- (4,2);
        \draw[very thin, color=blue!50] (0,0) -- (0,2);
        \draw[very thin, color=blue!50] (1,0) -- (1,2);
        \draw[very thin, color=blue!50] (2,0) -- (2,2);
        \draw[very thin, color=blue!50] (3,0) -- (3,2);
        \draw[very thin, color=blue!50] (4,1) -- (4,2);
    	\end{scope}
\end{tikzpicture}
    \caption{From left to right: a $\Le$-diagram of shape $\lambda =(4,3)$, its equivalent $\Le$-pipe dream $\Le(t,w)$ corresponding to the pair of permutations $t=143256$ and $w=461235$, and the simple transpositions associated to $\lambda$.  
See Example~\ref{eg.uw}.}
    \label{fig.le1}
\end{figure}

We now explain how to identify a $\Le$-pipe dream with a flag positroid pipe dream, under a rotation.

\begin{theorem}\label{thm.rpd_is_le}
Let $w \in\fS_n$ be a Grassmannian permutation. 
For any $t\leq w$ (equivalently $w_0w \leq w_0t$) in the Bruhat order, the flag positroid pipe dream $\FPP(w_0w, w_0t)$ is the $\Le$-pipe dream associated to the interval $[t,w]$, after pushing together the columns in $\Rothe(w_0w)$ and rotating $180^\circ$.
\end{theorem}

\begin{proof}
We have $t\leq w$ if and only if $w_0w \leq w_0t$ in Bruhat order.
By Theorem~\ref{thm.main_FPP}, $\FPP(w_0w,w_0t)$ exists.
We will show that $\FPP(w_0w,w_0t)$ is $\Le(t,w)$, up to pushing together the columns in $\Rothe(w_0w)$ and rotating by $180^\circ$.

If the Grassmannian permutation $w$ has no descents, then $w$ is the identity permutation and $w_0w = w_0$ has no ascents. 
Then $t=w$ and both $\Rothe(w_0w)$ and $\Le(t,w)$ are empty.
Otherwise, $w$ has a unique descent in position $k$, and $w_0w$ has a unique ascent in position $k$, since $(w_0w)_i = n+1-w_i$.
This implies that the Rothe diagram of $w_0w$ has partition shape (after pushing together the columns in $\Rothe(w_0w)$ and rotating $180^\circ$) and has at most $k$ nonempty rows.

Let $\lambda = (\lambda_1,\ldots, \lambda_k)$ be the partition associated to the Grassmannian permutation $w$, where $\lambda_i$ is the number of inversion pairs $(k-i+1, k+j)$ such that $k-i+1 < k+j$ and $w_{k-i+1}> w_{k+j}$.
These inversion conditions are true if and only if 
\[
(w_0w)_{k-i+1} = n+1-w_{k-i+1} < n+1-w_{k+j} = (w_0w)_{k+j}
\]
for $k-i+1 < k+j$.  
These are the coinversions of $w_0w$ that correspond to the boxes in the $(k-i+1)$-th row of the boxes of the Rothe diagram of $w_0w$, for $i=1,\ldots, k$, therefore the Rothe diagram of $w_0w$ has shape $\lambda$.

It remains to show that the tiling of $\Le(t,w)$ is the same as the tiling of $\FPP(w_0w,w_0t)$.
Recall from Remark~\ref{rem.le_s} that a box in the $i$-th row and $j$-th column of $\Le(t,w)$ corresponds to the simple transposition $s_{k+j-i}$, and from Lemma~\ref{lem.ux} that a box in the $i$-th row and $h$-th column (from the right with respect to the Rothe diagram) of $\FPP(w_0w,w_0t)$ is $s_{i+h-1}$.
Under the $180^\circ$ rotation, a box $B$ in position $(i,j)$ in $\Le(t,w)$ is mapped to the box $\rho(B)$ in $\FPP(w_0w,w_0t)$ in the $(k-i+1)$-th row and $h$-th column, and the simple transpositions associated to the boxes $B$ and $\rho(B)$ are identical, since $s_{k-i+j}=s_{(k-i+1)+h-1}$ under the action of the rotation $\rho$.
Hence, the reduced word for $w$ arising from $\Le(t,w)$ is the inverse of the reduced word for $w_0w$ arising from $\FPP(w_0w,w_0t)$.

By Remark~\ref{rem.le_s}, the tiling of $\Le(t,w)$ by cross and elbow tiles specifies a reduced word for $t$ that is lexicographically maximal as a subword of $w$.
By Remark~\ref{rem.1213}, the tiling of $\FPP(w_0w,w_0t)$ by cross and elbow tiles in the Rothe diagram of $w_0w$ specifies a reduced word for $x=t^{-1}w_0$ that is lexicographically minimal as a subword of $y=w^{-1}w_0$.
Taking inverses reverses the reduced expressions, and thus $y^{-1} = w_0w$ is a lexicographically maximal subword of $x^{-1} = w_0t$.
This shows that the tilings of $\Le(t,w)$ and $\FPP(w_0w,w_0t)$ are the same.
\end{proof}

\begin{example}
Continuing Example~\ref{eg.uw} for the pair of permutations $t=143256$ and  $w=461235$, Figure~\ref{fig.le2} shows how the flag positroid pipe dream $\FPP(w_0w, w_0t)$ recovers the $\Le$-diagram $\Le(t,w)$ for $t<w$.
We have $w_0w = 316542 < 634521 = w_0t$.
The boxes in the shaded region of $\FPP(w_0w, w_0t)$ correspond  exactly to the boxes in the $\Le$-diagram for $t$ and $w$.
Furthermore, the reduced expression for $x=(w_0w)^{-1}w_0 = w^{-1}$ specified by Lemma~\ref{lem.ux} is $w^{-1} = s_2s_3s_4s_5s_1s_2s_3$ and the reduced expression for $y=(w_0t)^{-1}w_0 = t^{-1}$ specified by Lemma~\ref{lem.vy} is $y=s_2s_3s_2$.
Taking inverses, these expressions coincide with those obtained for $t$ and $w$ from the combinatorics of $\Le$-diagrams.
\end{example}
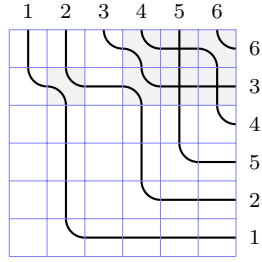
\begin{figure}[ht!]
\begin{tikzpicture}
\begin{scope}[scale=.5, xshift=0]
        \draw [black,thick,domain=180:270] plot ({3+.5*cos(\x)}, {6+.5*sin(\x)});
        \draw [black,thick,domain=180:270] plot ({1+.5*cos(\x)}, {5+.5*sin(\x)});
        \draw [black,thick,domain=180:270] plot ({6+.5*cos(\x)}, {4+.5*sin(\x)});
        \draw [black,thick,domain=180:270] plot ({5+.5*cos(\x)}, {3+.5*sin(\x)});
        \draw [black,thick,domain=180:270] plot ({4+.5*cos(\x)}, {2+.5*sin(\x)});
        \draw [black,thick,domain=180:270] plot ({2+.5*cos(\x)}, {1+.5*sin(\x)});
        \draw [black,thick] (0.5,5) -- (0.5, 6);
        \draw [black,thick] (1.5,5) -- (1.5, 6);
        \draw [black,thick] (1.5,1) -- (1.5, 4);
        \draw [black,thick] (3.5,2) -- (3.5, 4);
        \draw [black,thick] (4.5,3) -- (4.5, 4);
        
        \draw [black,thick] (2,0.5) -- (6,0.5);
        \draw [black,thick] (4,1.5) -- (6,1.5);
        \draw [black,thick] (5,2.5) -- (6,2.5);
        \draw [black,thick] (2,4.5) -- (3,4.5);

        \cross[thick](5,5); 
        \cross[thick](5,6);
        \cross[thick](6,5);    
        
        \elbow[thick](2,5); 
        \elbow[thick](4,5); 
        \elbow[thick](4,6);
        \elbow[thick](6,6);
        
	\draw[very thin, color=blue!50] (0,0) grid (6,6);

	\node at (.5,6.5) {\color{black}{\tiny$1$}};
	\node at (1.5,6.5) {\color{black}{\tiny$2$}};
	\node at (2.5,6.5) {\color{black}{\tiny$3$}};
	\node at (3.5,6.5) {\color{black}{\tiny$4$}};
	\node at (4.5,6.5) {\color{black}{\tiny$5$}};
	\node at (5.5,6.5) {\color{black}{\tiny$6$}};

        \node at (6.5, 5.5) {\color{black}{\tiny$6$}};
        \node at (6.5, 4.5) {\color{black}{\tiny$3$}};
        \node at (6.5, 3.5) {\color{black}{\tiny$4$}};
        \node at (6.5, 2.5) {\color{black}{\tiny$5$}};
        \node at (6.5, 1.5) {\color{black}{\tiny$2$}};
        \node at (6.5, 0.5) {\color{black}{\tiny$1$}};
\end{scope}
\end{tikzpicture}
    \caption{The flag positroid pipe dream $\FPP(w_0w,w_0t)$ where $t=143256$ and $w=461235$.
Compare to $\Le(t,w)$ in the center of Figure~\ref{fig.le1}.
    }
    \label{fig.le2}
\end{figure}

\begin{remark}
It follows from Theorem~\ref{thm.rpd_is_le} that the fact that $\Le$-pipe dreams are $\Le$-free is a special case that flag positroid pipe dreams are $\Gamma$-free, because when $u\in \fS_n$ has at most one ascent, then a configuration of a $\Gamma$ pattern in a flag positroid pipe dream $\FPP(u,v)$ is simplified and coincides with the configuration of a $\Le$ pattern.
\end{remark}

\section{Nonnegatively representable elementary positroid quotients}
\label{sec.positroid_quotients}

In this section we expand on the correspondence between flag positroid pipe dreams and flag positroids.
We begin by defining a map from a two-step nonnegative flag variety of consecutive ranks to a nonnegative Grassmanian, which in turn induces a map from two-step flag positroids of consecutive ranks $(k,k+1)$ on $[n]$ to rank $k+1$ positroids on $[0,n]$.
These maps are bijections onto their images (Theorems~\ref{thm:projection-bijection} and~\ref{thm.positroid-quotient-bijection}), and through this we can use the theory of directed graphs and bases of positroids to obtain the main combinatorial results in this section.
We show in Theorem~\ref{thm.char_Q} that for a given rank $k$ positroid $P$, we can combinatorially characterize all rank $k+1$ positroids $Q$ such that $P \unlhd_q Q$ is a nonnegatively representable quotient via `unblocked columns' of the $\Le$-diagram/$\Le$-pipe dream of $P$.
This yields a partial FPP that encodes the bases of $Q$.
We conclude this section by showing in Theorem~\ref{thm.fpp_richardson} that $\FPP(u,v)$ encodes the constituents of the flag positroid $P_\bullet(u,v)=(P_1(u,v), \ldots, P_{n-1}(u,v))$ that corresponds to any point in the nonnegative Richardson cell $\calR_{u,v}^{>0}$, and explain in Theorem~\ref{thm.sti_preserves_bases} how to compute the $\Le$-diagram of each constituent of $P_\bullet(u,v)$ via a standardization operation.

\subsection{The complete nonnegative representation of a nonnegative flag}\label{subsec.rnr}

In this section we define our preferred choice of a nonnegative matrix representation of a nonnegative flag. 

Let $\Mat_{k,n}$ denote the space of real $k \times n$ matrices.
Given $A\in \Mat_{k,n}$, let $A_{I,J}$ denote the submatrix of $A$ consisting of the rows of $A$ indexed by $I\subseteq [k]$ and columns indexed by $J\subseteq [n]$.
A matrix $A\in \Mat_{k,n}$ is \defn{lower reduced} if all entries below each leading entry of a row are zero.
The leading entry in the $i$-th row of the matrix is the \defn{pivot} of that row, and the column containing it is the \defn{pivot column of the $i$-th row}, which we denote by $u_i$.
For an increasing sequence $\br=(r_1,\ldots, r_t)$ of positive integers with $r_0=0$ and $r_t=k<n$, the matrix $A\in \Mat_{k,n}$ is in \defn{reverse $\br$-echelon form} if for each submatrix $A_{[r_{l-1}+1, r_l], [n]}$, each pivot column is to the left of the previous pivot column, for $l=1,\ldots,t$.

\begin{definition}
Let $n$ be a positive integer and let $\br=(r_1,\ldots, r_t)$ be an increasing sequence of nonnegative integers with $r_0=0$ and $r_t=k<n$.
Suppose $A\in \Mat_{k,n}$ is a representation of the nonnegative flag variety $V_\bullet \in \Fl_{\br;n}^{\geq0}$.
Then $A$ is a \defn{reduced representation} of $V_\bullet$ if
\begin{enumerate}
\item[(i)] $A$ is in reverse $\br$-echelon form,
\item[(ii)] $A$ is lower reduced, and
\item[(iii)] each pivot entry $A_{i,u_i}$ is $\pm1$.
\end{enumerate}
\end{definition}

Given $A\in\Mat_{k,n}$ with pivots at $(i,u_i)$ for $i=1,\ldots, k$, let \defn{$\sgn_A(i)$} denote the sign of $A_{i,u_i}$ and let 
\[\nep_A(i)= |\{1\leq j <i\mid u_j > u_i\}|
\] be the number of pivots of $A$ which are northeast of the pivot in the $i$-th row.

A \defn{generalized permutation matrix} is a matrix with exactly one nonzero entry in each row and in each column.

\begin{lemma}\label{lem.rnr}
Let $A \in \Mat_{k,k}$ be a generalized permutation matrix with pivot columns $u_1,\dots,u_k$. 
Then the minor $\Delta_{\{u_1,\dots,u_i\}}(A)\geq 0$ for all $i=1,\dots, k$ if and only if $\sgn_A(i) = (-1)^{\nep_A(i)}$ for $i=1,\ldots, k$.
\end{lemma}
\begin{proof}
Define $U_i=\{u_j \mid 1\leq j\leq i\}$ for $i=1,\ldots,k$. We induct on $k$. For $k=1$, $\Delta_{U_1}(A)\geq 0$ if and only if $A_{1,u_1} \geq 0$. For $k>1$, note that the only nonzero entry in the last row is $A_{k,u_k}$ in the $(k -\nep_A(k))$-th column, so by cofactor expansion along the last row we have $\Delta_{U_k}(A) = (-1)^{k-(k-\nep_A(k))} A_{k,u_k} \Delta_{U_{k-1}}(A)$. 
Hence, if we assume $\Delta_{U_i}(A)\geq0$ for all $i$, then by the induction hypothesis 
we must have $\sgn_A(i)=(-1)^{\nep_A(i)}$ for all $i$. Conversely, given that $\sgn_A(i)=(-1)^{\nep_A(i)}$ for $i=1,\ldots,k$, the induction hypothesis and the cofactor expansion show that $\Delta_{U_i}(A)\geq0$ for $i=1,\ldots, k$.
The result follows.
\end{proof}

\begin{proposition}\label{prop.rnr}
Let $A\in\Mat_{k,n}$ be a matrix representation of a nonnegative flag $V_\bullet \in \Fl_{\br;n}^{\geq0}$. 
Then
\begin{enumerate}    
\item[(i)] $A$ is a nonnegative representation of $V_\bullet$ if $\sgn_A(i) = (-1)^{\nep_A(i)}$ for $i=1,\ldots, k$. 
\item[(ii)] In the case $\br=(1,2,\dots,k)$, $A$ is a nonnegative representation of $V_\bullet$ if and only if $\sgn_A(i) = (-1)^{\nep_A(i)}$ for $i=1,\ldots, k$.
\end{enumerate}
\end{proposition}
\begin{proof}
Let $r\in\br$, and consider the submatrix $A_{[r],U_r}$. Since each row has a distinct pivot column, this submatrix has the same determinant as the generalized permutation matrix consisting of just the pivot entries with zeros elsewhere. 
Thus by Lemma \ref{lem.rnr}, we can guarantee $\Delta_{U_r}(A)\geq 0$ for all $r\in\br$ by ensuring $\sgn_A(i)=(-1)^{\nep_A(i)}$ for $i=1,\ldots, k$, and if $\br = (1,2,\dots,k)$, then this is the only way to achieve it.
Since $V_\bullet$ is a nonnegative flag, the nonnegativity of all the minors of $A$ for a given rank $r\in\br$ is equivalent to the presence of a single nonnegative minor of that rank, hence statements (i) and (ii) follow.
\end{proof}

\begin{corollary}\label{cor.rnr}
Every nonnegative flag has a reduced nonnegative representation.
\end{corollary}
\begin{proof}
Starting with any matrix representation $A$ of the nonnegative flag $V_\bullet: V_1\subset \cdots \subset V_t$, we can obtain a reduced representation of $V_\bullet$ by adding a scalar multiple of a row to a row beneath it or by scaling a row, since both of these row operations preserve the $l$-th subspace of the flag.  
Lastly, we can then scale the necessary rows to obtain a representation $R$ for $V_\bullet$ with $\sgn_R(i) = (-1)^{\nep_R(i)}$ for $i=1,\ldots, k$.  By Proposition~\ref{prop.rnr}, it follows that $R$ is a reduced nonnegative representation of $V_\bullet$.
\end{proof}

Proposition~\ref{prop.rnr} and its corollary showed that if $V_\bullet$ is a nonnegative flag whose ranks $\br$ are consecutive and begin with $r_1=1$, then there is a {\em unique} reduced nonnegative representation of $V_\bullet$. 
In general, a reduced nonnegative representation of a partial nonnegative flag is not unique, but there exists one with leading entries $A_{i,u_i}=(-1)^{\nep_A(i)}$ for all $i$. 
In this way, if $V_\bullet$ is a projection of a complete flag, then this particular reduced nonnegative representative also represents that complete flag nonnegatively.  
This will be our canonical choice of a representation for a nonnegative flag.

\begin{definition}\label{defn.completerepn}
Given a nonnegative flag $V_\bullet \in \Fl_{\br;n}^{\geq0}$ with $\br=(r_1,\ldots, r_t)$ such that $r_t=k$, the \defn{complete nonnegative representation of $V_\bullet$} denoted by $R(V_\bullet)$, is the reduced matrix representation whose leading entries are $(-1)^{\nep_{R(V_\bullet)}(i)}$ for $i=1,\ldots, k$.

Likewise, for $W\in \Gr_{k,n}^{\geq0}$, the \defn{complete nonnegative representation of $W$} is denoted by $R(W)$ and is the reduced representation of $W$ whose leading entries are $(-1)^{\nep_{R(W)}(i)}$ for $i=1,\ldots, k$.
\end{definition}

For example, the matrix
\[R(V_\bullet)=\begin{bmatrix}
0&0&0&0&1&1&0\\
0&0&-1&-1&0&1&1\\
1&1&0&-1&0&0&0\\ \hline
0&0&0&0&0&1&2\\
\end{bmatrix}
\]
is the complete nonnegative representation of a nonnegative flag $V_\bullet \in\Fl_{(3,4);7}^{\geq0}$.

\begin{remark}
As we have seen in Remark \ref{rem.completable_flag_counterexample}, there are partial nonnegative flags that cannot be extended to a complete nonnegative flag. 
That is, for such a flag, its complete nonnegative representative is {\em not} a representation for a complete nonnegative flag. The term ``complete" here is only meant to refer to the fact that we have {\em at least one} nonnegative minor in particular for each rank, and thus if $V_\bullet$ is a projection of a nonnegative flag, this representation is also a complete nonnegative representation of that complete flag as well.
\end{remark}

\subsection{Two key maps}
Using the complete nonnegative representation of a nonnegative flag, we define a map $\Phi$ from the two-step nonnegative flag variety $\Fl_{(k,k+1);n}^{\geq0}$ to the nonnegative Grassmannian $\Gr_{k+1,n+1}^{\geq0}$.
This induces a map $\phi$ from the two-step flag positroids of consecutive ranks $(k,k+1)$ on $[n]$ to rank $k+1$ positroids on $[0,n]$, and we show in Theorem~\ref{thm.positroid-quotient-bijection} that $\phi$ is a bijection onto its image $\calR_{k+1}([0,n])$ (see Definition~\ref{defn.calR}).

In the following, given a matrix $B\in \Mat_{k+1,n+1}$, we will use the convention that the rows of $B$ are indexed by $[k+1]$ and the columns of $B$ are indexed by $[0,n]=\{0,1,\ldots, n\}$.
Define the column vector
\[\bc_{k+1} = [\, 0\, \cdots\, 0 \, (-1)^{k}\,]^T \in \Mat_{k+1,1}.\]

\begin{lemma}\label{lem:matrix-embedding}
Given $A\in \Mat_{k+1, n}$, let $B = [\bc_{k+1} \ A] \in \Mat_{k+1,n+1}$ be the matrix obtained by appending the column $\bc_{k+1}$ to the left of $A$, and let $A'\in \Mat_{k,n}$ be the restriction of $A$ to its first $k$ rows. 
Then for $S\in\binom{[0,n]}{k+1}$, the minor
\[
\Delta_S(B) = 
\begin{cases}
\Delta_S(A), &\hbox{if } 0 \notin S,\\
\Delta_{S\backslash \{0\}}(A'), &\hbox{if } 0 \in S.
\end{cases}
\]
\end{lemma}
\begin{proof}
If $0\notin S$, then the result is trivial.
Otherwise if $0\in S$, then using cofactor expansion along the zeroth column of $B$ gives $\Delta_S(B) = (-1)^{(k+1)+1}(-1)^k \Delta_{S\backslash\{0\}} (A') = \Delta_{S\backslash\{0\}}(A')$.
\end{proof}

It follows from the lemma that a matrix $A$ has nonnegative maximal minors in its first $k$ rows and first $k+1$ rows if and only if $B=[\bc_{k+1} \ A]$ has nonnegative $(k+1)\times (k+1)$ maximal minors.

Given the two-step flag $V_\bullet: V_k \subset V_{k+1} \in \Fl_{(k,k+1);n}^{\geq0}$ with $A=R(V_\bullet)$, then $B=[\bc_{k+1}\, A]$ is a complete nonnegative representation of a rank $k+1$ positroid on $[0,n]$.
Conversely, given any rank $k+1$ positroid on $[0,n]$ that contains $0$ in at least one of its bases and some vector space $W$ in that positroid cell, then $R(W)$ necessarily has $\bc_{k+1}$ as its first column and removing $\bc_{k+1}$ yields a (not necessarily reduced) nonnegative representation of a flag positroid.
This construction allows us to define a map that will be useful for combinatorially characterizing nonnegatively representable positroid quotients in terms of FPPs.

\begin{definition}
Define a map $\Phi:\Fl_{(k,k+1);n}^{\geq0} \rightarrow \Gr_{k+1,n+1}^{\geq0}$ as follows.
Given a two-step flag $V_\bullet\in \Fl_{(k,k+1);n}^{\geq0}$ with the complete nonnegative representation $A=R(V_\bullet)\in \Mat_{k+1,n}$, let $\Phi(V_\bullet) = W \in \Gr_{k+1,n+1}^{\geq0}$ where $W$ is represented by the matrix $B=[\bc_{k+1}\ A]$.
\end{definition}

By Lemma~\ref{lem:matrix-embedding}, the map $\Phi$ is injective, and so we examine the image of $\Phi$ to establish a bijection. 
Since the first column of $B$ is $\bc_{k+1}$, we know that if $S\in \binom{[0,n]}{k+1}$ is the lexicographically minimal index among the nonzero Pl\"{u}cker coordinates of $W$, then $0\in S$. 
That is, $0$ is in the lexicographically minimal basis of the positroid $M(B)$.
Also note that if $A$ is a complete nonnegative representation, then so is $B$. 
Furthermore, not only does $B$ have $(-1)^k$ as its leading nonzero entry in the last row, its second nonzero entry in that row is $(-1)^{\nep_A(k)}$ because it is a pivot of $A$. 
This is equivalent to forcing the equality of the two Pl\"{u}cker coordinates indexed by the column sets $S$ and $T$, where $T$ is the lexicographically minimal basis of $M(B)$ not containing $0$. 
These restrictions are enough to define the image of $\Phi$.

Recall that the nonnegative Grassmannian 
$\Gr_{k+1,n+1}^{\geq0} = \bigsqcup_P \calS_P$ has a decomposition into positroid cells $\calS_P$.

\begin{definition}\label{defn.calR}
For a positroid $R$ on the set $[0,n]$, let $S(R)$ denote the lexicographically minimal basis of $R$ and let $T(R)$ denote the lexicographically minimal basis of $R$ that does not contain $0$ (if every basis of $R$ contains $0$, define $T(R)=\emptyset$). 
Let 
\[\calR_{k+1}([0,n]) = \left\{ R \in \calP_{k+1}[0,n] \mid 0\in S(R) \hbox{ and } T(R)\neq \emptyset\right\}
\]
denote the set of positroids $R$ on $[0,n]$ of rank $k+1$ such that $0\in S(R)$ and $T(R)\neq \emptyset$.  
For such positroids, define $\widetilde\calS_R = \{V\in \calS_R \mid \Delta_{S(R)}(V) = \Delta_{T(R)}(V)\}$
and 
\[\Cov^{\geq0}_{k+1, n+1} 
= \bigsqcup_{R\in \calR_{k+1}([0,n])} \widetilde\calS_R,
\]
which is a subset of $\Gr^{\geq0}_{k+1,n+1}$.
\end{definition}

Let $\chi_0:\Mat_{k+1,n+1} \rightarrow \Mat_{k+1,n}$ denote the map that removes the zeroth column of the argument.

\begin{definition}
Define a map $\Psi: \Cov_{k+1,n+1}^{\geq0} \rightarrow \Fl_{(k,k+1);n}^{\geq0}$ as follows.  
Given $W\in \Cov_{k+1,n+1}^{\geq0}$ with the complete nonnegative representation $B=R(W)$, let $\Psi(W) = V_\bullet\in \Fl_{(k,k+1);n}^{\geq0}$ with $R(V_\bullet) = \chi_0(R(W))$.
\end{definition}

\begin{theorem}\label{thm:projection-bijection}
	The maps $\Phi$ and $\Psi$ are inverse homeomorphisms between $\Fl^{\geq0}_{(k,k+1);n}$ and $\Cov^{\geq0}_{k+1, n+1}$.
\end{theorem}
\begin{proof}
It is routine to check that $\Psi\circ \Phi$ is the identity on  $\Fl^{\geq0}_{(k,k+1);n}$.
Conversely, note that if $B=R(W)$ is the complete nonnegative representation of $W\in \Cov^{\geq0}_{k+1,n+1}$, then the pivot entry $B_{k+1,u_{k+1}} = (-1)^{\nep_B(k+1)} = (-1)^k$, so the zeroth column of $B$ is $\bc_{k+1}$, and $\Phi\circ \Psi$ is the identity on $\Cov^{\geq0}_{k+1,n+1}$.

To see that $\Psi$ is continuous, let $R:\Cov^{\geq0}_{k+1,n+1} \rightarrow \Mat_{k+1,n+1}$ denote the map that sends $W$ to its complete nonnegative representation, and let 
$\RowS_{(k,k+1)}: \Mat_{k+1,n} \rightarrow \Fl_{(k,k+1);n}^{\geq0}$ denote the map that sends a matrix to the two-step flag consisting of the span of its first $k$ rows and its first $k+1$ rows.
We have $\Psi = \RowS_{(k,k+1)} \circ \chi_0 \circ R$. 
Since $\RowS$ and $\chi_0$ are continuous, our only concern is the continuity of the map $R$. 
The only discontinuities in $R$ occur when a pivot entry of the representative matrix approaches 0 and two rows become exchanged to keep the pivots in descending order. However, by the definition of $\Cov^{\geq0}_{k+1,n+1}$, the pivot entry in the bottom row is always in the zeroth column, and hence the sequence of $k$-th and $(k+1)$-th constituent row spaces is not affected by this exchange. Therefore, once we apply $\RowS$, the discontinuities disappear.
	
The argument to see the continuity of $\Phi = R \circ \alpha_0 \circ \RowS_{k+1}$ is similar, where $\alpha_0$ denotes the operation of appending $\bc_{k+1}$ as the zeroth column of a matrix.
\end{proof}

We next show that the maps $\Phi$ and $\Psi$ induce bijections between their corresponding flag positroids and positroids. 
Let $\calP_{(k,k+1)}([n])$ denote the set of two-step flag positroids $P_\bullet=(P, Q)$ on $[n]$ such that $\rank\,P=k$ and $\rank\,Q=k+1$.

\begin{definition}
Let $\calB(P)$ denote the set of bases of a positroid $P$.
Define the map  $\phi: \calP_{(k,k+1)}([n]) \rightarrow \calR_{k+1}([0,n]): (P,Q) \mapsto R$ where 
\[\calB(R) = \{B\cup\{0\} \mid B\in \calB(P)\}\ \cup\  \calB(Q).\]
Conversely, define $\psi: \calR_{k+1}([0,n]) \rightarrow \calP_{(k,k+1)}([n]): R \mapsto (P,Q)$ where
\[\calB(P)=\{B\backslash \{0\} \mid  B\in \calB(R) \hbox{ and } 0\in B\}\ \hbox{ and }\ \calB(Q)= \{B \mid B\in \calB(R) \hbox{ and } 0\notin B\}. 
\]
\end{definition}
Given $P_\bullet=(P,Q)\in \calP_{(k,k+1)}([n])$ with a nonnegative representation $A$, then by Lemma~\ref{lem:matrix-embedding}, the matrix $B=[\bc_{k+1}\ A]$ represents the positroid $R\in \calR_{k+1}([0,n])$, so $\phi$ is well-defined.
Conversely, given a positroid $R\in \calR_{k+1}([0,n])$ with a nonnegative representation $M$ whose row space is some $W\in \widetilde{S}_R \subseteq \Cov_{k+1,n+1}^{\geq0}$, let $V_\bullet = \Psi(W) \in \calF_{(k,k+1);n}^{\geq0}$. 
Let $B=R(V_\bullet)$ be the complete nonnegative representation of the flag $V_\bullet$ so that by definition, its zeroth column is $\bc_{k+1}$.
Again by Lemma~\ref{lem:matrix-embedding}, the submatrices $A'=B_{[k+1],[n]}$ and $A=B_{[k],[n]}$ respectively represent positroids $Q$ and $P$ of ranks $k+1$ and $k$ such that $P\unlhd_q Q$.
Therefore, $\psi$ is also well-defined.

We can now refine Theorem~\ref{thm:projection-bijection}.
Let $\calS_{(P,Q)}\subseteq \Fl_{(k,k+1);n}^{\geq 0}$ denote the subset of flags whose flag positroid is $(P,Q)$. 

\begin{corollary}
Let $V_\bullet\in \Fl_{(k,k+1);n}^{\geq0}$ be the flag with associated flag positroid $(P,Q)$ and let $W\in\Gr_{k+1,n+1}^{\geq0}$ be the subspace with associated positroid $R$. 
Then $\phi(P,Q)$ is the positroid associated to $\Phi(V_\bullet)$ and $\psi(R)$ is the flag positroid associated to $\Psi(W)$.
Moreover, if $\phi(P,Q) = R$, then the restrictions $\Phi|_{\calS_{(P,Q)}}$ and $\Psi|_{\widetilde\calS_R}$ are inverse homeomorphisms.
\qed
\end{corollary}
\begin{proof}
The first claim follows directly from Lemma~\ref{lem:matrix-embedding}.
As for the second claim, $\Phi(\calS_{(P,Q)}) \subseteq \widetilde\calS_R$ and $\Psi(\widetilde\calS_R) \subseteq \calS_{(P,Q)}$, where we can be sure that $\Phi(\calS_{(P,Q)})$ is in $\widetilde\calS_R$ because it is in $\Cov_{k+1,n+1}^{\geq 0}$. 
By Theorem~\ref{thm:projection-bijection}, these restrictions are also inverse homeomorphisms.
\end{proof}

The key result in this section now follows.
\begin{theorem}\label{thm.positroid-quotient-bijection}
The maps $\phi$ and $\psi$ are inverse bijections between the set of flag positroids $\calP_{(k,k+1)}([n])$ and the set of positroids $\calR_{k+1}([0,n])$. \qed
\end{theorem}

\subsection{Positroid quotients and flag positroid pipe dreams}\label{subsec:le-diagrams-for-quotients}
We will use Theorem~\ref{thm.positroid-quotient-bijection} to obtain a combinatorial characterization in Theorem~\ref{thm.char_Q} of nonnegatively representable positroid quotients of consecutive ranks via partial FPPs.

\begin{definition}
For $1\leq k\leq n$, a \defn{partial flag positroid pipe dream} of rank $k$ is the restriction of a flag positroid pipe dream in $\setFPP(n)$ to its first $k$ rows.
We denote the set of rank $k$ partial flag positroid pipe dreams by $\setFPP_k(n)$.
Alternatively, given a permutation $u=u_1\cdots u_n\in \fS_n$, a partial FPP of $u_1\cdots u_k$ is a $\Gamma$-free filling of a $k\times n$ array 
with pivots in the positions $(i,u_i)$ for $i=1,\ldots, k$.

Let $\setFPP_k^{\sLe}(n)$ denote the subset of rank $k$ partial FPPs with pivot columns $u_1 > \cdots > u_k$ in decreasing order. 
See Example~\ref{eg.pfpp_repns}.
\end{definition}

We now reformulate a result from Section~\ref{subsec.le_parallel}.
Theorem~\ref{thm.rpd_is_le} states that there is a map from flag positroid pipe dreams $\FPP(u,v)$ where $u$ has at most one ascent in the $k$-th position to $\Le$-pipe dreams/$\Le$-diagrams with $k$ rows and is contained in a $k \times (n-k)$ rectangle.
This map is actually a bijection; given a $\Le$-pipe dream, traversing its southwest boundary to label the steps $1,\ldots, n$, the labels on the vertical portions of the boundary are its pivot columns.
Rotating the $\Le$-pipe dream $180^\circ$ and pulling apart the columns according to the pivots columns yields an FPP because being $\Le$-free as a $\Le$-pipe dream is equivalent to being $\Gamma$-free as an FPP.

Since such FPPs are in bijection with $\setFPP_k^{\sLe}(n)$ and such $\Le$-pipe dreams are in bijection with rank $k$ positroids on $[n]$, then we have the following equivalent statement.
\begin{proposition}\label{prop.PFPP}
The set of partial flag positroid pipe dreams $\setFPP_k^{\sLe}(n)$ with pivot columns $u_1> \cdots > u_k$ in decreasing order is in bijection with the set $\calP_k([n])$ of rank $k$ positroids on $[n]$.
\qed    
\end{proposition}

\begin{definition}\label{defn.DLP}
Given a positroid $P\in \calP_k([n])$, let $\defn{D^{\sLe}(P)} \in \setFPP_k^{\sLe}(n)$ denote the partial FPP with pivot columns in decreasing order which corresponds to $P$ in the sense of Proposition~\ref{prop.PFPP}.
\end{definition}

\begin{example}\label{eg.pfpp_repns}
Figure~\ref{fig.PFPP} shows two representations of a rank $3$ partial flag positroid pipe dream $D^{\sLe}(P)\in \setFPP_3^{\sLe}(9)$ with pivot columns $6>4>1$ in decreasing order.
$D^{\sLe}(P)$ encodes a positroid $P$ whose $\Le$-diagram is shown on the right of the figure.
Also see Figure~\ref{fig.PFPP_decperm}.
\end{example}

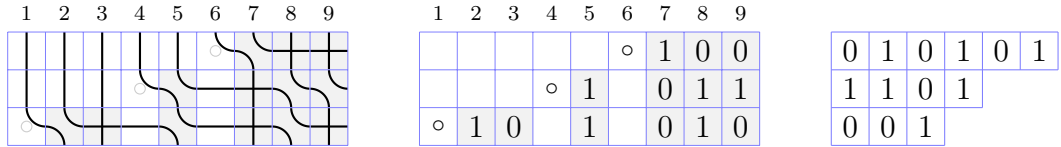
\begin{figure}[ht!]
\begin{tikzpicture}
\begin{scope}[scale=.5, xshift=0]
	\vertex[color=gray!40](w1) at (5.5, 2.5) {};
	\vertex[color=gray!40](w2) at (3.5, 1.5) {};
	\vertex[color=gray!40](w3) at (0.5, 0.5) {};
\draw [black,thick,domain=180:270] plot ({6+.5*cos(\x)}, {3+.5*sin(\x)});
\draw [black,thick,domain=180:270] plot ({4+.5*cos(\x)}, {2+.5*sin(\x)});
\draw [black,thick,domain=180:270] plot ({1+.5*cos(\x)}, {1+.5*sin(\x)});
\draw [black,thick] (0.5,1) -- (0.5,3);
\draw [black,thick] (1.5,1) -- (1.5,3);
\draw [black,thick] (2.5,1) -- (2.5,3);
\draw [black,thick] (3.5,2) -- (3.5,3);
\draw [black,thick] (4.5,2) -- (4.5,3);
\draw [black,thick] (3,0.5) -- (4,0.5);
\draw [black,thick] (5,0.5) -- (6,0.5);
\draw [black,thick] (5,1.5) -- (6,1.5);
\cross[thick](3,1);
\cross[thick](7,1);
\cross[thick](9,1);
\cross[thick](7,2);
\cross[thick](8,3);
\cross[thick](9,3);
\elbow[thick](2,1); 
\elbow[thick](5,1); 
\elbow[thick](8,1);
\elbow[thick](5,2);
\elbow[thick](8,2);
\elbow[thick](9,2);
\elbow[thick](7,3);
        
\draw[very thin, color=blue!50] (0,0) grid (9,3);

	\node at (.5,3.5) {\color{black}{\tiny$1$}};
	\node at (1.5,3.5) {\color{black}{\tiny$2$}};
	\node at (2.5,3.5) {\color{black}{\tiny$3$}};
	\node at (3.5,3.5) {\color{black}{\tiny$4$}};
	\node at (4.5,3.5) {\color{black}{\tiny$5$}};
	\node at (5.5,3.5) {\color{black}{\tiny$6$}};
	\node at (6.5,3.5) {\color{black}{\tiny$7$}};
	\node at (7.5,3.5) {\color{black}{\tiny$8$}};
	\node at (8.5,3.5) {\color{black}{\tiny$9$}};

\end{scope}
\begin{scope}[scale=.5, xshift=310]
\draw[fill, color=gray!10] (6,0) rectangle (9,3);
\draw[fill, color=gray!10] (4,0) rectangle (5,2);
\draw[fill, color=gray!10] (1,0) rectangle (3,1);
\node[] at (5.5,2.5){\footnotesize$\circ$};
\node[] at (3.5,1.5){\footnotesize$\circ$};
\node[] at (0.5,0.5){\footnotesize$\circ$};

\node[] at (2.5,0.5){$0$};
\node[] at (6.5,0.5){$0$};
\node[] at (8.5,0.5){$0$};
\node[] at (6.5,1.5){$0$};
\node[] at (7.5,2.5){$0$};
\node[] at (8.5,2.5){$0$};

\node[] at (1.5,0.5){$1$};
\node[] at (4.5,0.5){$1$};
\node[] at (7.5,0.5){$1$};
\node[] at (4.5,1.5){$1$};
\node[] at (7.5,1.5){$1$};
\node[] at (8.5,1.5){$1$};
\node[] at (6.5,2.5){$1$};
            
\draw[very thin, color=blue!50] (0,0) grid (9,3);
	\node at (0.5,3.5) {\color{black}{\tiny$1$}};
	\node at (1.5,3.5) {\color{black}{\tiny$2$}};
	\node at (2.5,3.5) {\color{black}{\tiny$3$}};
	\node at (3.5,3.5) {\color{black}{\tiny$4$}};
	\node at (4.5,3.5) {\color{black}{\tiny$5$}};
	\node at (5.5,3.5) {\color{black}{\tiny$6$}};
	\node at (6.5,3.5) {\color{black}{\tiny$7$}};
	\node at (7.5,3.5) {\color{black}{\tiny$8$}};
	\node at (8.5,3.5) {\color{black}{\tiny$9$}};
\end{scope}
\begin{scope}[scale=.5, xshift=620]
\draw[blue!50, thin] (0,0)--(3,0);
\draw[blue!50, thin] (0,1)--(4,1);
\draw[blue!50, thin] (0,2)--(6,2);
\draw[blue!50, thin] (0,3)--(6,3);
\draw[blue!50, thin] (0,0)--(0,3);
\draw[blue!50, thin] (1,0)--(1,3);
\draw[blue!50, thin] (2,0)--(2,3);
\draw[blue!50, thin] (3,0)--(3,3);
\draw[blue!50, thin] (4,1)--(4,3);
\draw[blue!50, thin] (5,2)--(5,3);
\draw[blue!50, thin] (6,2)--(6,3);
\node[] at (.5,.5){$0$};
\node[] at (1.5,.5){$0$};
\node[] at (2.5,.5){$1$};
\node[] at (.5,1.5){$1$};
\node[] at (1.5,1.5){$1$};
\node[] at (2.5,1.5){$0$};
\node[] at (3.5,1.5){$1$};
\node[] at (.5,2.5){$0$};
\node[] at (1.5,2.5){$1$};
\node[] at (2.5,2.5){$0$};
\node[] at (3.5,2.5){$1$};
\node[] at (4.5,2.5){$0$};
\node[] at (5.5,2.5){$1$};
\end{scope}
\end{tikzpicture}
\caption{The figures on the left and in the center show two representations of a partial flag positroid pipe dream $D^{\sLe}(P)\in \setFPP_3(9)$; one as a filling with pipes, the other as a $01$-filling.
The figure on the right shows the corresponding $\Le$-diagram.
}
\label{fig.PFPP}
\end{figure}

A consequence of Postnikov's results~\cite[Proposition 5.2, Theorem 6.5]{Postnikov06} shows that given a rank $k$ positroid $P$, we can interpret its bases as certain $k$-tuples of non-intersecting paths on an associated directed graph constructed from its $\Le$-diagram/$D^{\sLe}(P)$.
We generalize this construction to all partial FPPs.
\begin{definition}\label{defn.GofD}
Let $D \in \setFPP_k(n)$. 
Define an acyclic directed graph $G=G(D)$ as follows, where it is convenient for us to visualize this graph on a grid.
At each pivot elbow of $D$ there is a source vertex $s_{u_i}$ for $i=1,\ldots, k$, and at the top of each column of $D$ there is a sink vertex $t_j$ for $j=1,\ldots, n$. 
Additionally, there is a vertex at each elbow of $D$. 
For each non-sink vertex $x\in G$, there is an edge directed upwards to the next vertex in its column, and if $x$ is not a source, then there is an edge directed to $x$ from the vertex that is immediately left of $x$ in its row.
\end{definition}
In the case that $D$ has decreasing pivot columns $u_1> \cdots > u_k$, then the graph $G(D)$ is completely analogous to the $\Gamma$-graph in~\cite[Definition 6.3]{Postnikov06}, but in general, $G(D)$ may not be planar (see Example~\ref{eg.crossingGD}).
Also see Figures~\ref{fig.digraph} and~\ref{fig.std_example} for examples of these acyclic directed graphs.

\begin{figure}[ht!]
\centering
\begin{tikzpicture}
\begin{scope}[scale=.75, xshift=0]
\draw[very thin, color=gray!28] (0,0) grid (9,3);
\nnode[color=black, fill](s1) at (5.5, 2.5) {};
\node() at (5.25,2.25){\scriptsize$s_{u_1}$};
\nnode[color=black, fill](s2) at (3.5, 1.5) {};
\node() at (3.25,1.25){\scriptsize$s_{u_2}$};
\nnode[color=black, fill](s3) at (0.5, 0.5) {};
\node() at (0.25,0.25){\scriptsize$s_{u_3}$};
\nnode[color=black, fill=white, label=above:{\scriptsize$t_1$}](t1) at (0.5, 3) {};
\nnode[color=black, fill=white, label=above:{\scriptsize$t_2$}](t2) at (1.5, 3) {};
\nnode[color=black, fill=white, label=above:{\scriptsize$t_3$}](t3) at (2.5, 3) {};
\nnode[color=black, fill=white, label=above:{\scriptsize$t_4$}](t4) at (3.5, 3) {};
\nnode[color=black, fill=white, label=above:{\scriptsize$t_5$}](t5) at (4.5, 3) {};
\nnode[color=black, fill=white, label=above:{\scriptsize$t_6$}](t6) at (5.5, 3) {};
\nnode[color=black, fill=white, label=above:{\scriptsize$t_7$}](t7) at (6.5, 3) {};
\nnode[color=black, fill=white, label=above:{\scriptsize$t_8$}](t8) at (7.5, 3) {};
\nnode[color=black, fill=white, label=above:{\scriptsize$t_9$}](t9) at (8.5, 3) {};

\nnode[color=mediumseagreen, fill]() at (6.5, 2.5) {};

\nnode[color=mediumseagreen, fill]() at (4.5, 1.5) {};
\nnode[color=mediumseagreen, fill]() at (7.5, 1.5) {};
\nnode[color=mediumseagreen, fill]() at (8.5, 1.5) {};
\nnode[color=mediumseagreen, fill]() at (1.5, 0.5) {};
\nnode[color=mediumseagreen, fill]() at (4.5, 0.5) {};
\nnode[color=mediumseagreen, fill]() at (7.5, 0.5) {};

\draw [color=mediumseagreen, semithick] (s1) -- (6.5,2.5){};
\draw [color=mediumseagreen, semithick] (s2) -- (8.5,1.5){};
\draw [color=mediumseagreen, semithick] (s3) -- (7.5,0.5){};
\draw [-stealth, color=mediumseagreen, semithick] (s1) -- (t6){};
\draw [-stealth, color=mediumseagreen, semithick] (s2) -- (t4){};
\draw [-stealth, color=mediumseagreen, semithick] (s3) -- (t1){};
\draw [-stealth, color=mediumseagreen, semithick] (1.5,0.5) -- (t2){};
\draw [-stealth, color=mediumseagreen, semithick] (4.5,0.5) --  (t5){};
\draw [-stealth, color=mediumseagreen, semithick] (6.5,2.5) --  (t7){};
\draw [-stealth, color=mediumseagreen, semithick] (7.5,0.5) --  (t8){};
\draw [-stealth, color=mediumseagreen, semithick] (8.5,1.5) --  (t9){};

\draw[-stealth, color=mediumseagreen, semithick] (s1) edge (6.25,2.5);
\draw[-stealth,  color=mediumseagreen, semithick] (s2) edge (4.25,1.5) edge (6.25,1.5) edge (8.25,1.5);
\draw[-stealth,  color=mediumseagreen, semithick] (s3) edge (1.25,0.5) edge (3.25,0.5) edge (6.25,0.5);
\draw[-stealth, color=mediumseagreen, semithick] (4.5,0.5) -- (4.5,1.25);
\draw[-stealth,  color=mediumseagreen, semithick] (7.5,0.5) -- (7.5,1.25);
\end{scope}
\end{tikzpicture}
    \caption{The acyclic directed graph $G(D^{\sLe}(P))$ of the partial flag positroid pipe dream $D^{\sLe}(P)$ in Figure~\ref{fig.PFPP}.
    }
    \label{fig.digraph}
\end{figure}
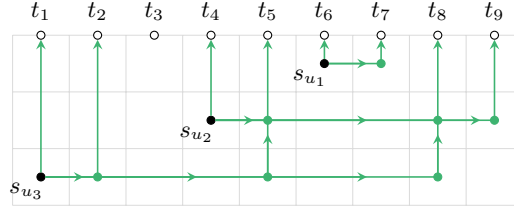

\begin{definition}
An \defn{admissible $k$-collection} in $G$ is a set of $k$ pairwise non-intersecting directed paths in $G$ such that each directed path begins at a source vertex and ends at a sink vertex of $G$.
Define
\[
\calB(D) = \big\{ J \in \hbox{$\binom{[n]}{k}$} \mid \exists \hbox{ an admissible $k$-collection in $G(D)$ with sink set $J$} \big\}.
\]
\end{definition}

A consequence of~\cite[Proposition 5.2 and Theorem 6.5]{Postnikov06} is the following.
\begin{proposition}\label{prop.nonintersecting}
Let $D^{\sLe}(P)\in \setFPP_k^{\sLe}(n)$ be a partial flag positroid pipe dream with pivot columns in decreasing order whose corresponding positroid is $P\in \calP_k([n])$.
Then $\calB(D^{\sLe}(P))$ is the set of bases of $P$.
\qed
\end{proposition}

\begin{remark}\label{rem.lexminmax_bases}
The lexicographically minimal basis of $P$ corresponds with the $k$-tuple of vertical paths $\{ p_i: s_{u_i} \rightarrow t_{u_i} \mid i=1,\ldots, k\}$ in $G(D^{\sLe}(P))$.
That is, the lexicographically minimal basis of $P$ is indexed by the pivot columns in $D^{\sLe}(P)$.
On the other hand, one can construct a $k$-admissible collection of paths in $G(D|_k)$ whose sink set is the lexicographically maximal basis of $P$ by taking paths $p_i$ beginning at $s_{u_i}$ and traveling as far as possible to the right before traveling north, for $i=k,\ldots, 1$. 
This can be shown via induction on $k$, and we leave the details to the reader.
\end{remark}

\begin{example}\label{eg.minmax_bases}
Consider the partial FPP $D^{\sLe}(P)$ from Figure~\ref{fig.PFPP} whose acyclic directed graph is shown in Figure~\ref{fig.digraph}.
Two admissible $3$-collections of paths in $G(D^{\sLe}(P))$ is shown in Figure~\ref{fig.nonint_paths}.
The one on the left shows that the lexicographically minimal basis of $P$ is $\{1,4,6\}$, while the one on the right shows that the lexicographically maximal basis of $P$ is $\{5,7,9\}$.
\end{example}

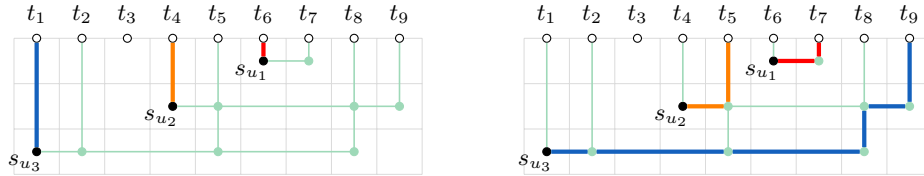
\begin{figure}[ht!]
\centering
\begin{tikzpicture}
\begin{scope}[scale=.6, xshift=0]
\draw[very thin, color=gray!28] (0,0) grid (9,3);
\nnode[color=black, fill](s1) at (5.5, 2.5) {};
\node() at (5.25,2.25){\scriptsize$s_{u_1}$};
\nnode[color=black, fill](s2) at (3.5, 1.5) {};
\node() at (3.25,1.25){\scriptsize$s_{u_2}$};
\nnode[color=black, fill](s3) at (0.5, 0.5) {};
\node() at (0.25,0.25){\scriptsize$s_{u_3}$};
\nnode[color=black, fill=white, label=above:{\scriptsize$t_1$}](t1) at (0.5, 3) {};
\nnode[color=black, fill=white, label=above:{\scriptsize$t_2$}](t2) at (1.5, 3) {};
\nnode[color=black, fill=white, label=above:{\scriptsize$t_3$}](t3) at (2.5, 3) {};
\nnode[color=black, fill=white, label=above:{\scriptsize$t_4$}](t4) at (3.5, 3) {};
\nnode[color=black, fill=white, label=above:{\scriptsize$t_5$}](t5) at (4.5, 3) {};
\nnode[color=black, fill=white, label=above:{\scriptsize$t_6$}](t6) at (5.5, 3) {};
\nnode[color=black, fill=white, label=above:{\scriptsize$t_7$}](t7) at (6.5, 3) {};
\nnode[color=black, fill=white, label=above:{\scriptsize$t_8$}](t8) at (7.5, 3) {};
\nnode[color=black, fill=white, label=above:{\scriptsize$t_9$}](t9) at (8.5, 3) {};

\nnode[color=mediumseagreen!50, fill]() at (6.5, 2.5) {};

\nnode[color=mediumseagreen!50, fill]() at (4.5, 1.5) {};
\nnode[color=mediumseagreen!50, fill]() at (7.5, 1.5) {};
\nnode[color=mediumseagreen!50, fill]() at (8.5, 1.5) {};
\nnode[color=mediumseagreen!50, fill]() at (1.5, 0.5) {};
\nnode[color=mediumseagreen!50, fill]() at (4.5, 0.5) {};
\nnode[color=mediumseagreen!50, fill]() at (7.5, 0.5) {};

\draw [color=mediumseagreen!50, semithick] (s1) -- (6.5,2.5){};
\draw [color=mediumseagreen!50, semithick] (s2) -- (8.5,1.5){};
\draw [color=mediumseagreen!50, semithick] (s3) -- (7.5,0.5){};
\draw [color=mediumseagreen!50, semithick] (s1) -- (t6){};
\draw [color=mediumseagreen!50, semithick] (s2) -- (t4){};
\draw [color=mediumseagreen!50, semithick] (s3) -- (t1){};
\draw [color=mediumseagreen!50, semithick] (1.5,0.5) -- (t2){};
\draw [color=mediumseagreen!50, semithick] (4.5,0.5) --  (t5){};
\draw [color=mediumseagreen!50, semithick] (6.5,2.5) --  (t7){};
\draw [color=mediumseagreen!50, semithick] (7.5,0.5) --  (t8){};
\draw [color=mediumseagreen!50, semithick] (8.5,1.5) --  (t9){};

\draw [color=red, ultra thick] (s1)--(t6);
\draw [color=orange, ultra thick] (s2)--(t4);
\draw [color=denim, ultra thick] (s3)--(t1);
\end{scope}

\begin{scope}[scale=.6, xshift=320]
\draw[very thin, color=gray!28] (0,0) grid (9,3);
\nnode[color=black, fill](s1) at (5.5, 2.5) {};
\node() at (5.25,2.25){\scriptsize$s_{u_1}$};
\nnode[color=black, fill](s2) at (3.5, 1.5) {};
\node() at (3.25,1.25){\scriptsize$s_{u_2}$};
\nnode[color=black, fill](s3) at (0.5, 0.5) {};
\node() at (0.25,0.25){\scriptsize$s_{u_3}$};
\nnode[color=black, fill=white, label=above:{\scriptsize$t_1$}](t1) at (0.5, 3) {};
\nnode[color=black, fill=white, label=above:{\scriptsize$t_2$}](t2) at (1.5, 3) {};
\nnode[color=black, fill=white, label=above:{\scriptsize$t_3$}](t3) at (2.5, 3) {};
\nnode[color=black, fill=white, label=above:{\scriptsize$t_4$}](t4) at (3.5, 3) {};
\nnode[color=black, fill=white, label=above:{\scriptsize$t_5$}](t5) at (4.5, 3) {};
\nnode[color=black, fill=white, label=above:{\scriptsize$t_6$}](t6) at (5.5, 3) {};
\nnode[color=black, fill=white, label=above:{\scriptsize$t_7$}](t7) at (6.5, 3) {};
\nnode[color=black, fill=white, label=above:{\scriptsize$t_8$}](t8) at (7.5, 3) {};
\nnode[color=black, fill=white, label=above:{\scriptsize$t_9$}](t9) at (8.5, 3) {};

\nnode[color=mediumseagreen!50, fill](c7) at (6.5, 2.5) {};

\nnode[color=mediumseagreen!50, fill](b5) at (4.5, 1.5) {};
\nnode[color=mediumseagreen!50, fill](b8) at (7.5, 1.5) {};
\nnode[color=mediumseagreen!50, fill](b9) at (8.5, 1.5) {};
\nnode[color=mediumseagreen!50, fill](a2) at (1.5, 0.5) {};
\nnode[color=mediumseagreen!50, fill](a5) at (4.5, 0.5) {};
\nnode[color=mediumseagreen!50, fill](a8) at (7.5, 0.5) {};

\draw [color=mediumseagreen!50, semithick] (s1) -- (6.5,2.5){};
\draw [color=mediumseagreen!50, semithick] (s2) -- (8.5,1.5){};
\draw [color=mediumseagreen!50, semithick] (s3) -- (7.5,0.5){};
\draw [color=mediumseagreen!50, semithick] (s1) -- (t6){};
\draw [color=mediumseagreen!50, semithick] (s2) -- (t4){};
\draw [color=mediumseagreen!50, semithick] (s3) -- (t1){};
\draw [color=mediumseagreen!50, semithick] (1.5,0.5) -- (t2){};
\draw [color=mediumseagreen!50, semithick] (4.5,0.5) --  (t5){};
\draw [color=mediumseagreen!50, semithick] (6.5,2.5) --  (t7){};
\draw [color=mediumseagreen!50, semithick] (7.5,0.5) --  (t8){};
\draw [color=mediumseagreen!50, semithick] (8.5,1.5) --  (t9){};

\draw [color=red, ultra thick] (s1)--(c7)--(t7);
\draw [color=orange, ultra thick] (s2)--(b5)--(t5);
\draw [color=denim, ultra thick] (s3)--(a2)--(a5)--(a8)--(b8)--(b9)--(t9);
\end{scope}
\end{tikzpicture}
    \caption{Two admissible $3$-collections of paths on $G(D^{\sLe}(P))$.  See Example~\ref{eg.minmax_bases}.
    }
    \label{fig.nonint_paths}
\end{figure}

Given a two-step flag positroid $(P,Q)\in \Fl_{(k,k+1);n}^{\geq0}$, each constituent of the flag has a corresponding partial FPP $D^{\sLe}(P)$ and $D^{\sLe}(Q)$.
It is then natural to ask, how are these related?
An intermediate solution to this problem comes from the combinatorial counterpart to Theorem~\ref{thm.positroid-quotient-bijection}, which would essentially state that for $R=\phi(P,Q)$, the restriction of the pipe dream $D^{\sLe}(R)$ to its last $n$ columns yields a partial FPP (whose pivot columns are not necessarily in decreasing order) that corresponds to the flag positroid $(P,Q)$ in the sense that the set of sinks of its admissible $k$-collections of paths is the set of bases of $P$, and the sinks of its admissible $(k+1)$-collections of paths is the set of bases of $Q$.
Therefore, the question of characterizing positroids $Q$ such that $P\unlhd_q Q$ becomes a question of what row can be appended below $D^{\sLe}(P)$ to form a partial FPP.
This leads to the next definition.

\begin{definition}\label{defn.Leblocked}
Let $D^{\sLe}\in \setFPP_k^{\sLe}(n)$ be a partial flag positroid pipe dream whose pivot columns are in decreasing order.  
A column in $D^{\sLe}$ is \defn{blocked} if it is a pivot column, or there is a cross tile in the column with an elbow tile to its right. 
Otherwise, the column is \defn{unblocked}.
See Example~\ref{eg.PFPP}.
\end{definition}
This definition will later be extended to all $\setFPP_k(n)$, see Definition~\ref{defn.blocked}.

\begin{remark}\label{rem.noncrossing_pipes_U}
We observe that if a pair of pipes in $D^{\sLe}$ both exit at the bottom of the partial FPP and they cross, then necessarily at least one of them exits in a blocked column because $D^{\sLe}$ is $\Gamma$-free.
Therefore, the collection of pipes in $D^{\sLe}$ that exit at the bottom of the diagram in the columns indexed by the unblocked columns are pairwise non-intersecting.
\end{remark}

Given $D\in \setFPP(n)$, let $D_{i,j}$ denote the tile in the $i$-th row and $j$-th column of the pipe dream. 

\begin{theorem}\label{thm.char_Q}
Let $P\in \calP_k([n])$ be a rank $k$ positroid and let $D^{\sLe}(P)\in \setFPP_k^{\sLe}(n)$ be its rank $k$ partial flag positroid pipe dream with pivot columns in decreasing order.
Let $U$ be the set of unblocked columns of $D^{\sLe}(P)$.
The positroids $Q\in \calP_{k+1}([n])$ such that $P\unlhd_q Q$ are in bijection with the nonempty subsets of $U$.
\end{theorem}
\begin{proof}
Consider the map $\phi: \calP_{(k,k+1)}([n]) \rightarrow \calR_{k+1}([0,n])$ from Theorem~\ref{thm.positroid-quotient-bijection} and suppose $\phi(P,Q)=R$.
Let $D^{\sLe}(R)$ be the partial FPP of $R$, and let $E$ denote the restriction of $D^{\sLe}(R)$ to its first $k$ rows and last $n$ columns.
We note that $E= D^{\sLe}(P)$.

Define a map 
\[\theta:
\left\{ Q\in \calP_{k+1}([n]) \mid P \unlhd_q Q \right\} \rightarrow 2^U\backslash \{\emptyset \}
\]
where 
\[\theta(Q) = \left\{ j \mid 1\leq j \leq n \hbox{ such that $D^{\sLe}(R)_{k+1, j}$ is an elbow tile}\right\}.
\]

We first discuss why $\theta(Q)$ is a nonempty subset of $2^U$.
Recall from Definition~\ref{defn.calR} that $R$ has a basis that contains $0$, and Proposition~\ref{prop.nonintersecting} states that the bases of $R$ are in bijection with the sinks of $(k+1)$-tuples of non-intersecting paths in $G(D^{\sLe}(R))$.
Since the pivot columns of $D^{\sLe}(R)$ are decreasing, then we must have $u_1 > \cdots > u_k > u_{k+1}=0$; that is, the source vertex $s_{u_{k+1}}$ is in the zeroth column of $G(D^{\sLe}(R))$ because the only directed path in $G(D^{\sLe}(R))$ that can end at the sink $t_0$ is the vertical path $s_0 \rightarrow t_0$.
Furthermore, by definition, $R$ has a basis that does not contain $0$.
This basis corresponds to a $(k+1)$-tuple of non-intersecting paths in $G(D^{\sLe}(R))$ and therefore, the path that starts at $s_{u_{k+1}}=s_0$ must have a sink that is not $t_0$.
This implies that $G(D^{\sLe}(R))$ necessarily has another (non-source) vertex in its $(k+1)$-st row, and hence $D^{\sLe}(R)$ has at least one elbow tile in its $(k+1)$-st row.
Lastly, $D^{\sLe}(R)$ is a partial FPP, so it is $\Gamma$-free.  
In particular, this means that the elbow tiles in the $(k+1)$-st row of $D^{\sLe}(R)$ cannot occur in the blocked columns of $E$.
Therefore, $\theta(Q) \in 2^U \backslash \{ \emptyset\}$.

The inverse map is defined as follows.
Given a nonempty subset $C \subseteq U$, let $F$ be the $(k+1) \times (n+1)$ partial FPP whose restriction to its first $k$ rows and last $n$ columns is the partial FPP $D^{\sLe}(P)$, and its $(k+1)$-st row consists of a pivot elbow in the zeroth column, elbow tiles in the columns specified by $C$, and all remaining tiles in its Rothe diagram are cross tiles.
Since the pivot columns of $F$ are decreasing and $C \subseteq U$, then $F$ is $\Gamma$-free and hence $F=D^{\sLe}(R)$ for some rank $k+1$ positroid $R$ on $[0,n]$. 
Note that the lexicographically minimal basis of $R$ is the indices of its pivot columns $\{u_1,\ldots, u_k, u_{k+1}=0\}$.
Moreover, $C$ is nonempty implies $F$ has at least one other elbow tile in its $(k+1)$-st row, so let $x$ denote the vertex in $G(F)$ that corresponds to the elbow tile $F_{k+1,c}$ for the minimum $c\in C$.
Then the $k+1$ directed paths $\{p_i: s_{u_i} \rightarrow t_{u_i} \mid i=1,\ldots, k\} \cup \{ p_{k+1}: s_{u_{k+1}} =s_0 \rightarrow x \rightarrow t_c\}$ in $G(F)$ is non-intersecting and its sinks do not contain $t_0$, and therefore $R \in \calR_{k+1}([0,n])$.
Let $Q\in \calP_{k+1}([n])$ be defined by $\psi(R) = (P,Q)$.
By Theorem~\ref{thm.positroid-quotient-bijection}, $P \unlhd_q Q$.
Therefore, $\theta$ is a bijection.
\end{proof}
Figure~\ref{fig.Q_PFPP} gives an illustration of the inverse map $\theta^{-1}$.

\begin{corollary}\label{cor.char_Q}
Let $P\in\calP_k([n])$ be a positroid. 
The number of positroids $Q$ such that $(P,Q)$ is a nonnegatively representable elementary quotient is $2^{|U|}-1$, where $U$ is the set of unblocked columns in the partial FPP $D^{\sLe}(P)$.
\end{corollary}

Consider the partial FPP $F=D^{\sLe}(R)$ in the proof of Theorem~\ref{thm.char_Q}.
Let $D$ denote the partial FPP that is induced by deleting the zeroth column of $F$.
We conclude this section by discussing the role of $D$ in terms of the flag positroid $(P,Q)=\psi(R)$.

\begin{corollary}\label{cor.PFPP_Q}
Given a flag positroid $(P,Q)\in \calP_{(k,k+1)}([n])$, let $F=D^{\sLe}(R)$ be the partial FPP of $R = \phi(P,Q)$ with pivot columns in decreasing order.
Let $D$ be the partial FPP induced by deleting the zeroth column of $F$.
Then $\calB(D)$ is the set of bases of the positroid $Q$, and if $D|_k$ is the restriction of $D$ to its first $k$ rows, then $\calB(D|_k)$ is the set of bases of the positroid $P$.
\end{corollary}
\begin{proof}
Since $R=\phi(P,Q)$, then by construction, the bases of $Q$ are the bases of $R$ which do not contain $0$.
As discussed in the proof of Theorem~\ref{thm.char_Q}, let $x$ denote the vertex in $G(F)$ that corresponds to the elbow tile $F_{k+1,c}$ for the minimum $c\in C \subseteq U$.
Then the source vertex in the $(k+1)$-st row of $G(D)$ is $x$.

Let $\calN = \{p_i: s_{u_i} \rightarrow t_{j_i} \mid i=1,\ldots, k+1 \}$ be a $(k+1)$-tuple of non-intersecting paths in $G(D)$ whose set of sinks is $B$. 
Let $y$ denote the source vertex in the $(k+1)$-st row of $G(F)$.
Then adding the edge $(y, x)$ to the beginning of the path $p_{k+1}$ yields a $(k+1)$-tuple of non-intersecting paths in $G(F)$ whose set of sinks is $B$.
By Proposition~\ref{prop.nonintersecting}, $B$ is a basis of $R$, and by construction $B$ does not contain $0$, so it is also a basis of $Q$.

Conversely, every basis $B$ of $Q$ is a basis of $R$, so there is a $(k+1)$-tuple of non-intersecting paths in $G(F)$ whose set of sink vertices is $B$.  
Deleting the edge $(y,x)$ from the path $p_{k+1}$ yields a $(k+1)$-tuple of non-intersecting paths in $G(D)$ whose set of sinks is $B$. 

Lastly, we note that the restriction of $D$ to its first $k$ rows is $D^{\sLe}(P)$, so $\calB(D|_k)$ is the set of bases of $P$.
\end{proof}

\begin{example}\label{eg.PFPP}
Consider the rank $3$ partial FPP $D^{\sLe}(P)$ from Figure~\ref{fig.PFPP}. 
Its set of unblocked columns is $U=\{2,5,8,9\}$.
Figure~\ref{fig.Q_PFPP} illustrates the construction in Theorem~\ref{thm.char_Q} that characterizes the rank $4$ positroids $Q$ such that $P \unlhd_q Q$ is a nonnegatively representable quotient.
Any filling of the four unblocked boxes in the fourth row that is not all crosses yields a partial FPP $F$, and deleting the zeroth column of $F$ induces a partial FPP $D(P,Q)$ that corresponds with a flag positroid $(P,Q)$.
Corollary~\ref{cor.char_Q} states that there are $15$ rank $4$ positroids $Q$ such that $P \unlhd_q Q$ is a nonnegatively representable elementary quotient.

For example if we choose $C=\{5,9\} \subseteq U$, then $D(P,Q)$ is constructed from $D^{\sLe}(P)$ by placing a pivot elbow in the fifth column, an elbow in the ninth column, and crosses in the rest of the fourth row of the Rothe diagram.
We denote this operation by $D(P,Q)=D^{\sLe}(P)\cdot C$, and this is shown on the right side of Figure~\ref{fig.Q_PFPP}.
This pipe dream corresponds with the flag positroid $(P,Q)$ in the sense of Corollary~\ref{cor.PFPP_Q}.
\end{example}

\begin{figure}[ht!]
\begin{tikzpicture}
\begin{scope}[scale=0.5, xshift=0]
\vertex[color=gray!40]() at (5.5, 2.5) {};
\vertex[color=gray!40]() at (3.5, 1.5) {};
\vertex[color=gray!40]() at (0.5, 0.5) {};
\vertex[color=gray!40]() at (-.5, -.5) {};

\draw[fill, color=mediumseagreen!20] (1,-1) rectangle (2,0);
\draw[fill, color=mediumseagreen!20] (4,-1) rectangle (5,0);
\draw[fill, color=mediumseagreen!20] (7,-1) rectangle (9,0);

\draw [black,thick,domain=180:270] plot ({6+.5*cos(\x)}, {3+.5*sin(\x)});
\draw [black,thick,domain=180:270] plot ({4+.5*cos(\x)}, {2+.5*sin(\x)});
\draw [black,thick,domain=180:270] plot ({1+.5*cos(\x)}, {1+.5*sin(\x)});
\draw [black,thick,domain=180:270] plot ({0+.5*cos(\x)}, {0+.5*sin(\x)});
\draw [black,thick] (-.5,0) -- (-.5,3);
\draw [black,thick] (0.5,1) -- (0.5,3);
\draw [black,thick] (1.5,1) -- (1.5,3);
\draw [black,thick] (2.5,1) -- (2.5,3);
\draw [black,thick] (3.5,2) -- (3.5,3);
\draw [black,thick] (4.5,2) -- (4.5,3);
\draw [black,thick] (0,-.5) -- (1,-.5);
\draw [black,thick] (3,-.5) -- (4,-.5);
\draw [black,thick] (3,0.5) -- (4,0.5);
\draw [black,thick] (5,-.5) -- (6,-.5);
\draw [black,thick] (5,0.5) -- (6,0.5);
\draw [black,thick] (5,1.5) -- (6,1.5);
\draw [black, thick] (2,-.5)--(3,-.5);
\draw [black, thick] (2.5,-1)--(2.5,0);
\draw [black, thick] (6,-.5)--(7,-.5);
\draw [black, thick] (6.5,-1)--(6.5,0);
\cross[thick](3,1);
\cross[thick](7,1);
\cross[thick](9,1);
\cross[thick](7,2);
\cross[thick](8,3);
\cross[thick](9,3);
\elbow[thick](2,1); 
\elbow[thick](5,1); 
\elbow[thick](8,1);
\elbow[thick](5,2);
\elbow[thick](8,2);
\elbow[thick](9,2);
\elbow[thick](7,3);

\draw[very thin, color=blue!50] (-1,-1) grid (9,3);
\draw[very thick, color=blue!50] (0,0)--(9,0)--(9,3)--(0,3)--(0,0);
\node at (-1.8, -.5) {\footnotesize$F$};
\node[] at (10.25, 2.5) {\footnotesize\textcolor{blue!70}{$D^{\sLe}(P)$}};
	\node at (-.5,3.5) {\color{black}{\tiny$0$}};
	\node at (0.5,3.5) {\color{black}{\tiny$1$}};
	\node at (1.5,3.5) {\color{black}{\tiny$2$}};
	\node at (2.5,3.5) {\color{black}{\tiny$3$}};
	\node at (3.5,3.5) {\color{black}{\tiny$4$}};
	\node at (4.5,3.5) {\color{black}{\tiny$5$}};
	\node at (5.5,3.5) {\color{black}{\tiny$6$}};
	\node at (6.5,3.5) {\color{black}{\tiny$7$}};
	\node at (7.5,3.5) {\color{black}{\tiny$8$}};
	\node at (8.5,3.5) {\color{black}{\tiny$9$}};
\end{scope}

\begin{scope}[scale=0.5, xshift=400]
\vertex[color=gray!40]() at (5.5, 2.5) {};
\vertex[color=gray!40]() at (3.5, 1.5) {};
\vertex[color=gray!40]() at (0.5, 0.5) {};
\vertex[color=gray!40]() at (4.5, -.5) {};

\draw [black,thick,domain=180:270] plot ({6+.5*cos(\x)}, {3+.5*sin(\x)});
\draw [black,thick,domain=180:270] plot ({4+.5*cos(\x)}, {2+.5*sin(\x)});
\draw [black,thick,domain=180:270] plot ({1+.5*cos(\x)}, {1+.5*sin(\x)});
\draw [black,thick,domain=180:270] plot ({5+.5*cos(\x)}, {0+.5*sin(\x)});
\draw [black,thick] (0.5,1) -- (0.5,3);
\draw [black,thick] (1.5,-1) -- (1.5,0);
\draw [black,thick] (1.5,1) -- (1.5,3);
\draw [black,thick] (2.5,-1) -- (2.5,0);
\draw [black,thick] (2.5,1) -- (2.5,3);
\draw [black,thick] (3.5,2) -- (3.5,3);
\draw [black,thick] (4.5,2) -- (4.5,3);
\draw [black,thick] (3,0.5) -- (4,0.5);
\draw [black,thick] (5,-.5) -- (6,-.5);
\draw [black,thick] (5,0.5) -- (6,0.5);
\draw [black,thick] (5,1.5) -- (6,1.5);

\cross[thick](3,1);
\cross[thick](7,1);
\cross[thick](9,1);
\cross[thick](7,0); 
\cross[thick](7,2);
\cross[thick](8,0); 
\cross[thick](8,3);
\cross[thick](9,3);
\elbow[thick](2,1); 
\elbow[thick](9,0); 
\elbow[thick](5,1); 
\elbow[thick](8,1);
\elbow[thick](5,2);
\elbow[thick](8,2);
\elbow[thick](9,2);
\elbow[thick](7,3);

\draw[very thin, color=blue!50] (0,-1) grid (9,3);
\node at (-1.6, -.5) {\footnotesize$D(P,Q)$};

	\node at (0.5,3.5) {\color{black}{\tiny$1$}};
	\node at (1.5,3.5) {\color{black}{\tiny$2$}};
	\node at (2.5,3.5) {\color{black}{\tiny$3$}};
	\node at (3.5,3.5) {\color{black}{\tiny$4$}};
	\node at (4.5,3.5) {\color{black}{\tiny$5$}};
	\node at (5.5,3.5) {\color{black}{\tiny$6$}};
	\node at (6.5,3.5) {\color{black}{\tiny$7$}};
	\node at (7.5,3.5) {\color{black}{\tiny$8$}};
	\node at (8.5,3.5) {\color{black}{\tiny$9$}};
\end{scope}
\end{tikzpicture}
\caption{The rank $3$ partial FPP $D^{\sLe}(P)$ from Figure~\ref{fig.PFPP} has unblocked columns $U=\{2,5,8,9\}$. 
Choosing the nonempty subset $C=\{5,9\} \subseteq U$ to fill with elbow tiles in the fourth row induces a partial FPP $D(P,Q)$ that corresponds with a flag positroid $(P,Q)$.
}
\label{fig.Q_PFPP}
\end{figure}
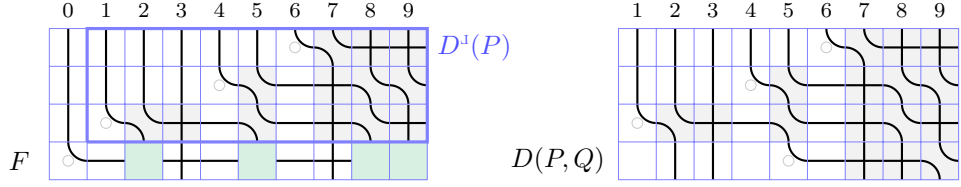

\subsection{Standardizing the partial flag positroid pipe dream}
\label{subsec.stdizing}
Given a rank $k$ positroid $P\in \calP_k([n])$, Theorem~\ref{thm.char_Q} and Corollary~\ref{cor.PFPP_Q} give a combinatorial characterization of all rank $k+1$ positroids $Q$ such that $P\unlhd_q Q$ in terms of how one can add a row below $D^{\sLe}(P)$ to obtain a partial FPP $D(P,Q)$ such that $\calB(D(P,Q))$ is the set of bases of $Q$.
However, we do not yet have a direct description of $D^{\sLe}(Q)$.

In this section, we show how $D(P,Q)$ can be converted to $D^{\sLe}(Q)$ via a standardization process, and hence from this we can go on to recover the $\Le$-diagram and other combinatorial objects associated to $Q$.

Given $D\in \setFPP_k(n)$, let $D_{i,j}$ denote the tile in the $i$-th row and $j$-th column of the pipe dream. 

\begin{definition}\label{defn.rowop}
Fix $i\in [n-1]$ and $k\in [n]$.
The \defn{standardization operation $\st_i$} on $\setFPP_k(n)$ is defined as follows.
Let $D\in \setFPP_k(n)$.
\begin{enumerate}
\item[(I)] If its pivot columns $u_i > u_{i+1}$ are in decreasing order, then $\st_i(D) = D$.
\item[(II)] Otherwise its pivot columns $u_i < u_{i+1}$ are in increasing order.

	\begin{enumerate}
	\item[(a)] If $D_{i,u_{i+1}}$ is a cross, let $D':=\st_i(D)$ be obtained by switching the $i$-th and $(i+1)$-st rows of $D$ and then replacing $D'_{i,u_i}$ with a vertical pipe and $D'_{i+1,u_{i+1}}$ with a horizontal pipe  (these two replacements are forced by virtue of the fact that two pivot elbows have switched rows).
	\item[(b)] Otherwise, $D_{i,u_{i+1}}$ is an elbow.
Let $j^*> u_{i+1}$ be the first column in which $D_{i,j^*}$ is a cross and $D_{i+1,j^*}$ is an elbow or pivot elbow (if such a column does not exist, then set $j^*=\infty$).
Let $D':=\st_i(D)$ be the pipe dream obtained as follows:
		\begin{enumerate}
		\item[(i)] In the columns $u_i$ to $u_{i+1}$ inclusive, switch the $i$-th and $(i+1)$-st rows of $D$ and then replace $D'_{i,u_i}$ with a vertical pipe and $D'_{i+1,u_{i+1}}$ with a horizontal pipe, and
		\item[(ii)] in the columns $j^*$ to $n$ inclusive, switch the $i$-th and $(i+1)$-st rows of $D$ and then replace $D'_{i+1,j^*}$ with an elbow.
		\end{enumerate}
	This case is illustrated in Figure~\ref{fig.new_move}.
	\end{enumerate}
\end{enumerate}
\end{definition}

\begin{figure}[ht!]
\begin{tikzpicture}
\begin{scope}[scale=.65, xshift=0]
    \vertex[color=gray!40] at (5.5, 0.5) {};
	\vertex[color=gray!40] at (1.5, 1.5) {};	
    \draw [black,thick,domain=180:270] 
        plot ({6+.5*cos(\x)}, {1+.5*sin(\x)});
    \draw [black,thick,domain=180:270] 
        plot ({2+.5*cos(\x)}, {2+.5*sin(\x)});
        
    \draw [black,thick] (2.5,0)--(2.5,1){};
    \draw [black,thick] (4.5,0)--(4.5,1){};
    \elbow[thick](7,1);
    \draw [black,thick] (7,.5)--(8,.5){};
    \cross[thick](9,1);
    \elbow[thick](10,1);
    \elbow[thick](11,1);
    \cross[thick](12,1);

    \elbow[thick](3,2);
    \draw [black,thick] (3,1.5)--(4,1.5){};
    \cross[thick](5,2);
    \elbow[thick](6,2);
    \elbow[thick](7,2);
    \draw [black,thick] (7,1.5)--(8,1.5){};
    \elbow[thick](9,2);
    \cross[thick](10,2);
    \cross[thick](11,2);
    \cross[thick](12,2);
    
    \node[label={\textcolor{blue}{\tiny{$j^*$}}}] at (9.5,1.7){};
    \node[label={\textcolor{blue}{\tiny{$u_{i+1}$}}}] at (5.5,1.7){};
    \node[label={\textcolor{blue}{\tiny{$u_i$}}}] at (1.5,1.7){};

	\node[label=left:{\textcolor{blue}{\tiny$i$}}] at (1.2,1.5) {};
	\node[label=left:{\textcolor{blue}{\tiny$i+1$}}] at (1.2,0.5) {};

    \node[label={\footnotesize$D$}] at (6.5,-2.5){};
    
    \node[] at (13,1.7){\scriptsize$\st_i$};
    \node[] at (13,1){$\longrightarrow$};
                
	\draw[very thin, color=blue!50] (1,0) grid (12,2);
	
	\draw[very thick, color=lavender(floral)] (6,-1.1)--(6,2.7);
	\draw[very thick, color=lavender(floral)] (9,-1.1)--(9,2.7);
	\draw[ultra thick, color=lavender(floral)] (1,0)--(2,0)--(2,1)--(1,1)--(1,0);
	\draw[ultra thick, color=lavender(floral)] (5,1)--(6,1)--(6,2)--(5,2)--(5,1);
	\draw[ultra thick, color=lavender(floral)] (9,1)--(10,1)--(10,2)--(9,2)--(9,1);
	\node[] at (3.5,-.5){\tiny\textcolor{lavender(floral)!125}{switch rows}};
	\node[] at (3.5,-1){\tiny\textcolor{lavender(floral)!125}{and replace tiles}};
	\node[] at (7.5,-.5){\tiny\textcolor{lavender(floral)!125}{freeze}};
	\node[] at (10.7,-.5){\tiny\textcolor{lavender(floral)!125}{switch rows}};
	\node[] at (10.7,-1){\tiny\textcolor{lavender(floral)!125}{and replace tile}};
\end{scope}

\begin{scope}[scale=.65, xshift=370]
    \vertex[color=gray!40] at (5.5, 1.5) {};
	\vertex[color=gray!40] at (1.5, 0.5) {};	
    \draw [black,thick,domain=180:270] 
        plot ({6+.5*cos(\x)}, {2+.5*sin(\x)});
    \draw [black,thick,domain=180:270] 
        plot ({2+.5*cos(\x)}, {1+.5*sin(\x)});
        
    \draw [black,thick] (1.5,1)--(1.5,2){};
    \draw [black,thick] (2.5,1)--(2.5,2){};
    \draw [black,thick] (4.5,1)--(4.5,2){};
    \elbow[thick](7,2);
    \draw [black,thick] (7,1.5)--(8,1.5){};
    \elbow[thick](9,2);
    \elbow[thick](10,2);
    \elbow[thick](11,2);
    \cross[thick](12,2);

    \elbow[thick](3,1);
    \draw [black,thick] (3,.5)--(4,.5){};
    \cross[thick](5,1);
    \draw [black,thick] (5,.5)--(6,.5){};
    \elbow[thick](7,1);
    \draw [black,thick] (7,.5)--(8,.5){};
    \cross[thick](9,1);
    \elbow[thick](10,1);
    \cross[thick](11,1);
    \cross[thick](12,1);
    
    \node[label={\textcolor{blue}{\tiny{$j^*$}}}] at (9.5,1.7){};
    \node[label={\textcolor{blue}{\tiny{$u_{i+1}$}}}] at (5.5,1.7){};
    \node[label={\textcolor{blue}{\tiny{$u_i$}}}] at (1.5,1.7){};
    \node[label={\footnotesize$D'=\st_i(D)$}] at (6.5,-2.5){};

	\draw[very thin, color=blue!50] (1,0) grid (12,2);	
\end{scope}
\end{tikzpicture} 
\caption{An illustration of case (IIb) in the standardization operation $\st_i$ in Definition~\ref{defn.rowop}.}
    \label{fig.new_move}
\end{figure}
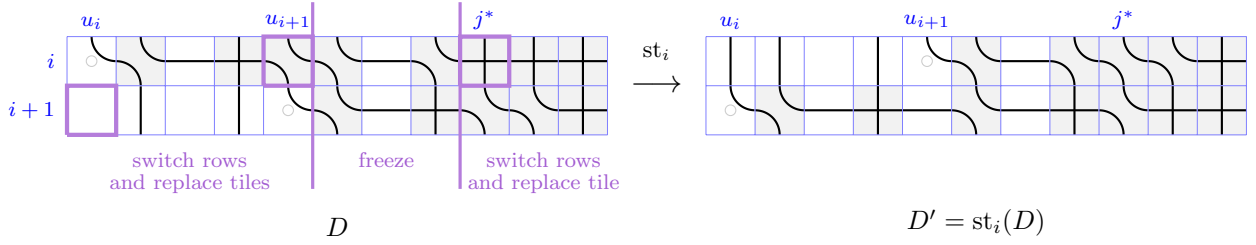

We show that applying $\st_i$ results in another partial FPP.
\begin{proposition}\label{prop.sti_preserves_FPP}
Let $D \in \setFPP_k(n)$ be a partial flag positroid pipe dream with pivot columns $u_i < u_{i+1}$ in increasing order for some $1\leq i \leq k-1$.
Then $\st_i(D) \in \setFPP_k(n)$.
\end{proposition}
\begin{proof}
Suppose for contradiction that $D'=\st_i(D)$ contains a $\Gamma$ pattern.  
Then the cross tile $X$ in the northwest corner of this $\Gamma$ pattern must be one of the following:
\begin{enumerate}
\item[(i)] $D'_{i+1,j}$ for some $1\leq j<u_{i+1}$,
\item[(ii)] $D'_{i,j}$ or $D'_{i+1,j}$ for some $u_{i+1}< j< j^*$, or
\item[(iii)] $D'_{i,j}$ for some $j^* < j \leq n$.
\end{enumerate}

In Case (i), let $D'_{i+1, l}$ be the elbow tile to the right of $X$ in the $\Gamma$ pattern.
We must have $l> u_{i+1}$ because $D$ is $\Gamma$-free.
By definition, if $D'_{i+1, l}$ is an elbow for some $u_{i+1}<l < j^*$, then so is $D_{i,l}$.
So $D$ is $\Gamma$-free implies $l>j^*$.
But then no such elbow tile can exist because $D'_{i+1, j^*}$ is the last elbow in its row, so Case (i) is not possible.

In Case (ii), an analysis similar to that in Case (i) rules out the possibility that $X=D'_{i+1,j}$.
So suppose $X=D'_{i,j}$ and let $D'_{i, l}$ be the elbow tile to the right of $X$ in the $\Gamma$ pattern.
We must have $l > j^*$ because $D$ is $\Gamma$-free.
By definition, $D'_{i+1,j}=D_{i+1,j}$ is also a cross tile, so $D$ is $\Gamma$-free implies no such elbow tile $D'_{i,l}=D_{i+1,l}$ exists, and Case (ii) is not possible.

Similarly, we see that Case (iii) also is not possible, because for $j > j^*$, the $i$-th and $(i+1)$-st rows of $D$ have simply been switched, and $D$ is $\Gamma$-free implies no new $\Gamma$ patterns in $D'$ can be introduced for these values of $j$.

Therefore, we may conclude that $\st_i(D)$ is $\Gamma$-free, and hence is a partial FPP.
\end{proof}

Our next goal is to show that applying $\st_i$ to $D$ yields a partial FPP that encodes the same bases for a positroid as $D$.
To this end, we must describe the action of $\st_i$ on the associated directed graph $G(D)$.
In the technical lemma that follows, it is convenient for us to visualize $G(D)$ on a $k\times n$ grid. 
Given a vertex $w\in G(D)$, let $N(w)$, $S(w)$, $E(w)$, and $W(w)$ respectively denote the vertex in $G(D)$ that is immediately north, south, east, and west of $w$.

\begin{lemma}\label{lem.ggprime}
Let $D \in \setFPP_k(n)$ be a partial flag positroid pipe dream whose pivot columns $u_i< u_{i+1}$ are in increasing order for a fixed $i$.
Let $D'=\st_i(D)$, and let $G=G(D)$ and $G'=G(D')$ be their associated directed graphs.
Let $j^*$ be the first column in which $D_{i,j^*}$ is a cross and $D_{i+1,j^*}$ is an elbow or pivot elbow.
We have the following cases.
\begin{enumerate}
\item[(i)] Case $j^*= u_{i+1}$. Then $G=G'$.
\item[(ii)] Case $u_{i+1}< j^*\leq n$.  
Let $w=N(s_{u_{i+1}})$, $v=W(w)$, $x$ be the vertex in the $(i+1)$-st row and $j^*$-th column of $G$, and $y$ be the rightmost vertex in the $i$-th row of $G$.
Then $G'$ is obtained from $G$ by
	\begin{itemize}
	\item deleting $w$ and all edges incident to it, and adding the edge $(s_{u_{i+1}}, N(w))$,
	\item adding a vertex $z$ and the edge $(z,x)$, 
	\item 
	if $E(w)$ exists, then delete $(s_{u_{i+1}}, E(s_{u_{i+1}}))$ and add $(s_{u_{i+1}}, E(w))$ and $(y,x)$, 
	\item if $x=E(s_{u_{i+1}})$ then add $(v,z)$, 
	otherwise delete $(W(x), x)$ and add $(W(x), z)$ and $(v, E(s_{u_{i+1}}))$,
	\end{itemize}

\item[(iii)] Case $j^*=\infty$. 
Let $w=N(s_{u_{i+1}})$ and $v=W(w)$.
Then $G'$ is obtained from $G$ by
	\begin{itemize}
	\item deleting $w$ and all edges incident to it, and adding $(s_{u_{i+1}}, N(w))$,	

	\item if $E(s_{u_{i+1}})$ exists, then delete $(s_{u_{i+1}}, E(s_{u_{i+1}}))$ and add $(v, E(s_{u_{i+1}}))$,
	\item if $E(w)$ exists, then add $(s_{u_{i+1}}, E(w))$.
	\end{itemize}
\end{enumerate}
\end{lemma}
\begin{proof}
In all cases, to obtain $G'$ from $G$, the horizontal edges in the $i$-th and $(i+1)$-st rows are deleted and reattached according to how $D'$ is obtained from $D$.
We note that in Case (iii), $G'$ contains one less vertex than $G$.

In the case $j^*=u_{i+1}$, $D_{i,u_{i+1}}$ is a cross.
Then $D_{i,j}$ is a cross for all $j\geq u_{i+1}$ because $D$ does not contain a $\Gamma$ pattern, and so $G$ does not have a vertex in the $i$-th row and $j$-th column for $j\geq u_{i+1}$.
$G$ also has not have a vertex in the $(i+1)$-st row and $j$-th column for $1\leq j < u_{i+1}$ by construction.
Applying $\st_i$ to $G$ then simply transposes its $i$-th and $(i+1)$-st rows, thus $G'= G$ (although $G$ and $G'$ have a different embedding on the grid as induced from $D$ and $D'$). 

In the case $u_{i+1}< j^* \leq n$, $D_{i,u_{i+1}}$ is an elbow.
The existence of the vertices $v, w, x, y$ are by assumption.
Applying $\st_i$ to $G$ deletes the vertex $w$, adds the vertex $z$, and all new edges are induced from the action of $\st_i$ on $D$.

The case $j^*=\infty$ can be considered as a special case of $(ii)$ where $x$ does not exist and hence $z$ is not added.
\end{proof}

\begin{example}
Figure~\ref{fig.nnprime} illustrates the most generic construction from Lemma~\ref{lem.ggprime}, when $u_{i+1} < j^* \leq n$ with $x \neq E(s_{u_{i+1}})$ and $y\neq w$.
Figure~\ref{fig.std_example} gives specific instances of each of the cases in Lemma~\ref{lem.ggprime}.
Case (i) is illustrated by $\st_3$, case (ii) is illustrated by $\st_1$, and case (iii) is illustrated by $\st_2$.
\end{example}

\begin{figure}[ht!]
\begin{tikzpicture}

\begin{scope}[scale=.7, xshift=0]
\draw[very thin, color=gray!28] (0,0) grid (10,2);
\nnode[color=black, fill](s1) at (0.5, 1.5) {};
	\node() at (0.25,1.25){\tiny$s_{u_i}$};
\nnode[color=black, fill](s2) at (3.5, 0.5) {};
	\node() at (3.25,0.25){\tiny$s_{u_{i+1}}$};
\node[](t1) at (0.5, 2.5) {};
\node[](t2) at (1.5, 2.5) {};
\node[](t3) at (2.5, 2.5) {};
\node[](t4) at (3.5, 2.5) {};
\node[](t5) at (4.5, 2.5) {};
\node[](t6) at (5.5, 2.5) {};
\node[](t7) at (6.5, 2.5) {};
\node[](t8) at (7.5, 2.5) {};
\node[](t9) at (8.5, 2.5) {};
\node[](tX) at (9.5, 2.5) {};
\nnode[color=mediumseagreen, fill]() at (5.5, 0.5){};
\nnode[color=mediumseagreen, fill]() at (6.5, 0.5){};
\nnode[color=mediumseagreen, fill]() at (8.5, 0.5){};
	\node() at (8.75,0.75){\tiny$x$};
\nnode[color=mediumseagreen, fill]() at (9.5, 0.5){};
\nnode[color=mediumseagreen, fill]() at (6.5, 1.5){};
\nnode[color=mediumseagreen, fill]() at (1.5, 1.5){};
	\node() at (1.75,1.75){\tiny$v$};
\nnode[color=mediumseagreen, fill]() at (3.5, 1.5){};
	\node() at (3.25,1.75){\tiny$w$};
\nnode[color=mediumseagreen, fill]() at (4.5, 1.5){};
\nnode[color=mediumseagreen, fill]() at (5.5, 1.5){};
\nnode[color=mediumseagreen, fill]() at (6.5, 1.5){};
\nnode[color=mediumseagreen, fill]() at (7.5, 1.5){};
	\node() at (7.75,1.75){\tiny$y$};
\draw [color=mediumseagreen, thick] (s1) -- (t1){};
\draw [color=mediumseagreen, thick] (s2) -- (t4){};
\draw [color=mediumseagreen, thick] (s1) -- (7.5,1.5){};
\draw [color=mediumseagreen, thick] (s2) -- (9.5,0.5){};
\draw [color=mediumseagreen, thick] (1.5,1)--(t2){};
\draw [color=mediumseagreen, thick] (4.5,1)--(t5){};
\draw [color=mediumseagreen, thick] (5.5,0)--(t6){};
\draw [color=mediumseagreen, thick] (6.5,0)--(t7){};
\draw [color=mediumseagreen, thick] (7.5,1)--(t8){};
\draw [color=mediumseagreen, thick] (8.5,0)--(t9){};
\draw [color=mediumseagreen, thick] (9.5,0)--(tX){};
    \node[] at (11.2,1.7){\scriptsize$\st_i$};
    \node[] at (11.2,1){$\longrightarrow$};
    \node at (-.5, 0) {\footnotesize$G$};
	\draw[very thick, color=lavender(floral)] (4,-1.1)--(4,2.7);
	\draw[very thick, color=lavender(floral)] (8,-1.1)--(8,2.7);
	\node[] at (2,-.5){\tiny\textcolor{lavender(floral)!125}{delete w}};
    \node[] at (2,-1){\tiny\textcolor{lavender(floral)!125}{and reconnect}};
	\node[] at (6,-.5){\tiny\textcolor{lavender(floral)!125}{freeze}};
	\node[] at (9.4,-.5){\tiny\textcolor{lavender(floral)!125}{add z}};
	\node[] at (9.4,-1){\tiny\textcolor{lavender(floral)!125}{and reconnect}};
\end{scope}

\begin{scope}[scale=.7, xshift=350]
\draw[very thin, color=gray!28] (0,0) grid (10,2);
\nnode[color=black, fill](s1) at (0.5, 1.5) {};
	\node() at (0.25,1.25){\tiny$s_{u_i}$};
\nnode[color=black, fill](s2) at (3.5, 0.5) {};
	\node() at (3.25,0.25){\tiny$s_{u_{i+1}}$};
\node[](t1) at (0.5, 2.5) {};
\node[](t2) at (1.5, 2.5) {};
\node[](t3) at (2.5, 2.5) {};
\node[](t4) at (3.5, 2.5) {};
\node[](t5) at (4.5, 2.5) {};
\node[](t6) at (5.5, 2.5) {};
\node[](t7) at (6.5, 2.5) {};
\node[](t8) at (7.5, 2.5) {};
\node[](t9) at (8.5, 2.5) {};
\node[](tX) at (9.5, 2.5) {};
\nnode[color=mediumseagreen, fill](Es2) at (5.5, 0.5){};
\nnode[color=mediumseagreen, fill](Wx) at (6.5, 0.5){};
\nnode[color=mediumseagreen, fill](x) at (8.5, 0.5){};
	\node() at (8.75,0.75){\tiny$x$};
\nnode[color=mediumseagreen, fill]() at (9.5, 0.5){};
\nnode[color=mediumseagreen, fill]() at (6.5, 1.5){};
\nnode[color=mediumseagreen, fill](v) at (1.5, 1.5){};
	\node() at (1.75,1.75){\tiny$v$};
\nnode[color=gray!30, fill](w) at (3.5, 1.5){};
	\node() at (3.25,1.75){\tiny\textcolor{gray!30}{$w$}};
\nnode[color=mediumseagreen, fill](Ew) at (4.5, 1.5){};
\nnode[color=mediumseagreen, fill]() at (5.5, 1.5){};
\nnode[color=mediumseagreen, fill]() at (6.5, 1.5){};
\nnode[color=mediumseagreen, fill](y) at (7.5, 1.5){};
	\node() at (7.75,1.75){\tiny$y$};
\nnode[color=mediumseagreen, fill](z) at (8.5, -.5){};
	\node() at (8.75,-.75){\tiny$z$};	
\draw [color=mediumseagreen, thick] (s1) -- (t1){};
\draw [color=mediumseagreen, thick] (s1) -- (1.5,1.5){};
\draw [color=gray!30, thick] (v)--(Ew){};
\draw [color=mediumseagreen, thick] (4.5,1.5) -- (7.5,1.5){};
\draw [color=gray!30, thick] (s2) -- (Es2){};
\draw [color=mediumseagreen, thick] (5.5,0.5) -- (6.5,0.5){};
\draw [color=gray!30, thick] (Wx) -- (x){};
\draw [color=mediumseagreen, thick] (8.5,0.5) -- (9.5,0.5){};
\draw [color=mediumseagreen, thick] (1.5,1)--(t2){};
\draw [color=mediumseagreen, thick] (4.5,1)--(t5){};
\draw [color=mediumseagreen, thick] (5.5,0)--(t6){};
\draw [color=mediumseagreen, thick] (6.5,0)--(t7){};
\draw [color=mediumseagreen, thick] (7.5,1)--(t8){};
\draw [color=mediumseagreen, thick] (8.5,0)--(t9){};
\draw [color=mediumseagreen, thick] (9.5,0)--(tX){};
\draw [color=mediumseagreen, thick] (s2) -- (t4){};

\draw [color=mediumseagreen, thick] (x)--(8.5,-1);
\draw [color=mediumseagreen, thick] (y)--(x);
\draw [color=mediumseagreen, thick] (Wx)--(z);
\draw [color=mediumseagreen, thick] (v)--(Es2);
\draw [color=mediumseagreen, thick] (s2)--(Ew);
\node at (-.5, 0) {\footnotesize$G'$};
\end{scope}
\end{tikzpicture}    
\caption{The effect of $\st_i$ induced on the directed graphs associated to partial flag positroid pipe dreams. 
Case (ii) of Lemma~\ref{lem.ggprime} is illustrated here. 
Note that $G'$ is not depicted in grid form.}
    \label{fig.nnprime}
\end{figure}
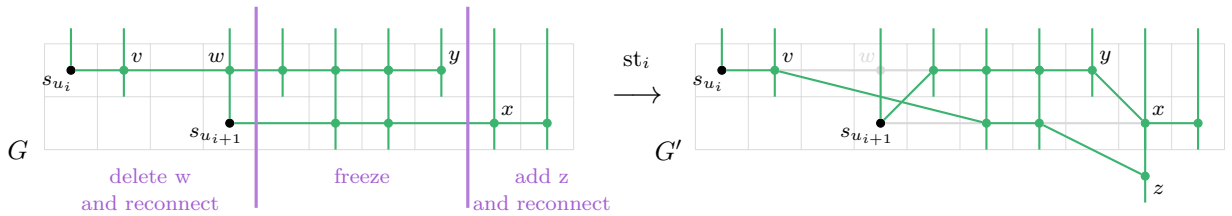

Recall from Proposition~\ref{prop.nonintersecting} that the bases of a positroid $P$ is $\calB(D^{\sLe}(P))$, which is the sink sets of non-intersecting paths in $G(D^{\sLe}(P))$.

\begin{theorem}\label{thm.sti_preserves_bases}
Let $D \in \setFPP_k(n)$ be a partial flag positroid pipe dream with increasing pivot columns $u_i < u_{i+1}$ for some $1\leq i \leq k-1$.
For $D' =\st_i(D)$, we have $\calB(D) = \calB(D')$, so $D'$ encodes the same positroid as $D$.
\end{theorem}
\begin{proof}
By Proposition~\ref{prop.sti_preserves_FPP}, we know that $D'$ is a partial FPP.
We need to show that the bases $\calB(D)$ and $\calB(D')$ are equal.
Let $j^*$ be the first column in which $D_{i,j^*}$ is a cross and $D_{i+1,j^*}$ is an elbow or a pivot elbow.

In the case $j^* =u_{i+1}$, the graphs $G(D)$ and $G(D')$ are equal by Lemma~\ref{lem.ggprime}, and therefore $\calB(D) = \calB(D')$ in this case.

It remains to show that $\calB(D) = \calB(D')$ when $j^*>u_{i+1}$.
As noted in Lemma~\ref{lem.ggprime}, $G'=G(D')$ is obtained from $G(D)$ by deleting a vertex $w$, adding a vertex $z$ (when $j^*\neq \infty$), then exchanging the horizontal edges in the $i$-th and $(i+1)$-st rows between columns $u_{i+1}-1$ and $u_{i+1}+1$ and between columns $j^*-1$ and $j^*$, where applicable.  
Given a $k$-tuple of non-intersecting paths $\calN=\{p_i: s_{u_i} \rightarrow t_{j_i} \mid i=1,\ldots, k \}$ in $G$, we will construct a $k$-tuple of non-intersecting paths $\calN'=\{p_i': s_{u_i'} \rightarrow t_{j_i'} \mid i=1,\ldots, k \}$ in $G'$ having the same set of sink vertices $T = \{t_{j_1}, \ldots, t_{j_k}\} = \{t_{j_1'}, \ldots, t_{j_k'} \}$, and vice versa.
We note that $s_{u_i} = s_{u_{i+1}'}$ and $s_{u_{i+1}} = s_{u_i'}$.

Following the notation from Lemma~\ref{lem.ggprime}, we explain the most generic case where the vertex $x$ exists in $G$ and is not equal to $E(s_{u_{i+1}})$, and the vertex $E(w)$ exists so that $y\neq w$. 
All other cases are degenerate versions of this one.
Please refer to Figure~\ref{fig.nnprime} for a visual aid.

Note that the set of edges which are present in $G$ but not in $G'$ is 
\[\calF = \left\{ (v,w), (w,E(w)), (w, N(w)), (s_{u_{i+1}}, w), (s_{u_i+1}, E(s_{u_{i+1}})), (W(x), x) \right\},
\]
so if $\calN$ does not contain any of the edges in $\calF$, then $\calN'=\calN$ is a $k$-tuple of non-intersecting paths in $G'$.

We proceed to check the cases when $\calN$ contains some of the edges in $\calF$, listing them in order of precedence.
(i) Suppose $\calN$ has a path $p_h$ that contains the edge $(v,w)$.
Note that $h\neq i+1$.
Since $\calN$ is non-intersecting, then the path $p_{i+1}$ necessarily contains the edge $(s_{u_{i+1}}, E(s_{u_{i+1}}))$.
Let $p_h'$ be the path in $G'$ consisting of the portion of $p_h$ from its source $s_{u_h}$ to $v$, concatenated with $(v, E(s_{u_{i+1}}))$, and followed by the portion of $p_{i+1}$ from $E(s_{u_{i+1}})$ to its sink $t_{j_{i+1}}$.
Let $p_i'$ be the path in $G'$ consisting of the edge from $s_{u_i'}=s_{u_{i+1}}$ to $E(w)$, concatenated with the portion of $p_h$ from $w$ to its sink $t_{j_h}$. 
Let $p_{i+1}'= p_i$ if $h\neq i$.

(ii) Suppose the path $p_h\in \calN$ contains the edge $(w, E(w))$ but not $(v,w)$; then $h=i+1$.
$\calN$ is non-intersecting so it does not contain any path that contains the edge $(v,w)$.
Let $p_i'$ be the path in $G'$ consisting of the edge from $s_{u_i'} = s_{u_{i+1}}$ to $E(w)$, concatenated with the portion of $p_{i+1}$ from $E(w)$ to its sink $t_{j_{i+1}}$.
Let $p_{i+1}'= p_i$.

(iii) Suppose the path $p_{i+1}\in \calN$ contains the edge $(s_{u_{i+1}}, w)$ but not $(w,E(w))$; then it must contain $(w, N(w))$.
Let $p_i'$ be the path in $G'$ consisting of the edge from $s_{u_i'}=s_{u_{i+1}}$ to $N(w)$, concatenated with the portion of $p_{i+1}$ from $N(w)$ to its sink $t_{j_{i+1}}$.
Let $p_{i+1}' = p_i$.

(iv) Suppose the path $p_{i+1}\in \calN$ contains the edge $(s,E(s))$ but not $(v,w)$.
If $p_{i+1}$ does not contain a vertex from the $i$-th row of $G$, then by assumption it must contain the vertex $x$.
Let $p_i'$ be the path in $G'$ going through the vertices $s_{u_i'}= s_{u_{i+1}}, E(w), y, x$, concatenated with the portion of $p_{i+1}$ from $x$ to its sink $t_{j_{i+1}}$.
Otherwise let $a$ be the leftmost vertex of $p_{i+1}$ that is in the $i$-th row of $G$.
Let $p_i'$ be the path in $G'$ going through the vertices $s_{u_i'}= s_{u_{i+1}}, E(w), S(a), a$, concatenated with the portion of $p_{i+1}$ from $a$ to its sink $t_{j_{i+1}}$.
Let $p_{i+1}' = p_i$.

(v) Suppose the path $p_h\in \calN$ contains the edge $(W(x), x)$ but not $(s,E(s))$.
Let $p_h'$ be the path consisting of the portion of $p_h$ from its source $s_{u_h}$ to $W(x)$, concatenated with the edges $(W(x), z)$ and $(z,x)$, and followed by the portion of $p_h$ from $x$ to its source $t_{j_h}$.

Finally, we note that if the path $p_h\in \calN$ contains the edge $(w,N(w))$, then it also contains the edge $(v,w)$ or $(s_{u_{i+1}}, w)$.  
These cases have already been considered.

Combining some of the cases (i) -- (v) if necessary, following the order of precedence, the set $\calN'$ of paths obtained by replacing $p_h$, $p_i$, and $p_{i+1}$ in $\calN$ with their primed counterparts is a set of $k$ non-intersecting paths in $G'$ with the same sinks as $\calN$.
Therefore $\calB(D) \subseteq \calB(D')$.

The reverse construction is similar; the set of edges which are present in $G'$ but not in $G$ is (identifying $v,x,y,z$ with their counterparts in $G$)
\[\calF' = \{ (v,E(v)), (s_{u_i'}, E(s_{u_i'})), (s_{u_i'}, N(s_{u_i'})), (W(z), z), (z,x), (W(x), x)\}.
\]
A careful case analysis similar to the one above shows that $\calB(D') \subseteq \calB(D)$, and the result follows.
\end{proof}

Following Theorem~\ref{thm.sti_preserves_bases}, we can make the next definition.
\begin{definition}
For any partial FPP $D \in \setFPP_k(n)$, we define its \defn{standardization}, denoted $\st(D)\in \setFPP_k^{\sLe}(n)$, to be the partial FPP obtained by performing a sequence of standardization operations $\st_i$ on $D$ until its pivot columns are in decreasing order.
\end{definition}

\begin{example}
Figure~\ref{fig.std_example} shows an example of the standardization of a rank $4$ partial FPP.
If $D$ encodes the positroid $Q$, then the standardization $\st(D)=D^{\sLe}(Q)$ is the $\Le$-pipe dream of $Q$, after rotating $180^\circ$ and pushing the columns of the Rothe diagram together.
By Theorem~\ref{thm.sti_preserves_bases}, the sink sets of the admissible collections of paths in the directed graphs $G(D)$ and $G(\st(D))$ are equal, so $D$ and $\st(D)$ encode the same rank $4$ positroid.
\end{example}

\begin{figure}[ht!]
\begin{tikzpicture}
\begin{scope}[scale=.48, yshift=0]
    \vertex[color=gray!40] at (4.5, 3.5) {};
	\vertex[color=gray!40] at (2.5, 2.5) {};
	\vertex[color=gray!40] at (0.5, 1.5) {};
	\vertex[color=gray!40] at (5.5, 0.5) {};	
	
	\draw [black,thick,domain=180:270] 
        plot ({5+.5*cos(\x)}, {4+.5*sin(\x)});
    \draw [black,thick,domain=180:270] 
        plot ({3+.5*cos(\x)}, {3+.5*sin(\x)});
    \draw [black,thick,domain=180:270] 
        plot ({1+.5*cos(\x)}, {2+.5*sin(\x)});
    \draw [black,thick,domain=180:270] 
        plot ({6+.5*cos(\x)}, {1+.5*sin(\x)});    
        
    \elbow[thick](6,4);
    \cross[thick](7,4);
    \elbow[thick](4,3);
    \elbow[thick](6,3);
    \elbow[thick](7,3);
    \elbow[thick](2,2);
    \elbow[thick](4,2);
    \cross[thick](6,2);
    \cross[thick](7,2);
    \elbow[thick](7,1);
    \draw [black,thick] (2,1.5)--(3,1.5){};
    \draw [black,thick] (4,1.5)--(5,1.5){};
    \draw [black,thick] (4,2.5)--(5,2.5){};
    \draw [black,thick] (0.5,2)--(0.5,4){};
    \draw [black,thick] (1.5,0)--(1.5,1){};
    \draw [black,thick] (1.5,2)--(1.5,4){};
    \draw [black,thick] (2.5,3)--(2.5,4){};
    \draw [black,thick] (3.5,0)--(3.5,1){};
    \draw [black,thick] (3.5,3)--(3.5,4){};

    \node[label={\textcolor{blue}{\tiny{$j^*$}}}] at (5.5,3.7){};
    \node[label={\footnotesize$D$}] at (3.5,-1.5){};
    
    \node[] at (8,1.7){\scriptsize$\st_3$};
    \node[] at (8,1){$\longrightarrow$};
                
	\draw[very thin, color=blue!50] (0,0) grid (7,4);
\end{scope}

\begin{scope}[scale=.48, xshift=255]
    \vertex[color=gray!40] at (4.5, 3.5) {};
	\vertex[color=gray!40] at (2.5, 2.5) {};
	\vertex[color=gray!40] at (5.5, 1.5) {};
	\vertex[color=gray!40] at (0.5, 0.5) {};
	
	\draw [black,thick,domain=180:270] 
        plot ({5+.5*cos(\x)}, {4+.5*sin(\x)});
    \draw [black,thick,domain=180:270] 
        plot ({3+.5*cos(\x)}, {3+.5*sin(\x)});
    \draw [black,thick,domain=180:270] 
        plot ({6+.5*cos(\x)}, {2+.5*sin(\x)});
    \draw [black,thick,domain=180:270] 
        plot ({1+.5*cos(\x)}, {1+.5*sin(\x)});    
        
    \elbow[thick](6,4);
    \cross[thick](7,4);
    \elbow[thick](4,3);
    \elbow[thick](6,3);
    \elbow[thick](7,3);
    \elbow[thick](7,2);
    \elbow[thick](2,1);
    \elbow[thick](4,1);
    \cross[thick](7,1);
    \draw [black,thick] (2,0.5)--(3,0.5){};
    \draw [black,thick] (4,0.5)--(6,0.5){};
    \draw [black,thick] (4,2.5)--(5,2.5){};
    \draw [black,thick] (0.5,1)--(0.5,4){};
    \draw [black,thick] (1.5,1)--(1.5,2){};
    \draw [black,thick] (1.5,2)--(1.5,4){};
    \draw [black,thick] (2.5,3)--(2.5,4){};
    \draw [black,thick] (3.5,1)--(3.5,2){};
    \draw [black,thick] (3.5,3)--(3.5,4){};

    
    \node[] at (8,2.7){\scriptsize$\st_2$};
    \node[] at (8,2){$\longrightarrow$};
                
	\draw[very thin, color=blue!50] (0,0) grid (7,4);
\end{scope}

\begin{scope}[scale=.48, xshift=510]
    \vertex[color=gray!40] at (4.5, 3.5) {};
	\vertex[color=gray!40] at (5.5, 2.5) {};
	\vertex[color=gray!40] at (2.5, 1.5) {};
	\vertex[color=gray!40] at (0.5, 0.5) {};
	
	\draw [black,thick,domain=180:270] 
        plot ({5+.5*cos(\x)}, {4+.5*sin(\x)});
    \draw [black,thick,domain=180:270] 
        plot ({6+.5*cos(\x)}, {3+.5*sin(\x)});
    \draw [black,thick,domain=180:270] 
        plot ({3+.5*cos(\x)}, {2+.5*sin(\x)});
    \draw [black,thick,domain=180:270] 
        plot ({1+.5*cos(\x)}, {1+.5*sin(\x)});    
        
    \elbow[thick](6,4);
    \cross[thick](7,4);
    \elbow[thick](7,3);
    \elbow[thick](4,2);
    \elbow[thick](7,2);
    \elbow[thick](2,1);
    \elbow[thick](4,1);
    \cross[thick](7,1);
    \draw [black,thick] (2,0.5)--(3,0.5){};
    \draw [black,thick] (4,0.5)--(6,0.5){};
    \draw [black,thick] (4,1.5)--(6,1.5){};
    \draw [black,thick] (0.5,1)--(0.5,4){};
    \draw [black,thick] (1.5,1)--(1.5,4){};
    \draw [black,thick] (2.5,2)--(2.5,4){};
    \draw [black,thick] (3.5,2)--(3.5,4){};

    \node[label={\textcolor{blue}{\tiny{$j^*$}}}] at (6.5,3.7){};
    
    \node[] at (8,3.7){\scriptsize$\st_1$};
    \node[] at (8,3){$\longrightarrow$};
                
	\draw[very thin, color=blue!50] (0,0) grid (7,4);
\end{scope}

\begin{scope}[scale=.48, xshift=765]
    \vertex[color=gray!40] at (5.5, 3.5) {};
	\vertex[color=gray!40] at (4.5, 2.5) {};
	\vertex[color=gray!40] at (2.5, 1.5) {};
	\vertex[color=gray!40] at (0.5, 0.5) {};
	
	\draw [black,thick,domain=180:270] 
        plot ({6+.5*cos(\x)}, {4+.5*sin(\x)});
    \draw [black,thick,domain=180:270] 
        plot ({5+.5*cos(\x)}, {3+.5*sin(\x)});
    \draw [black,thick,domain=180:270] 
        plot ({3+.5*cos(\x)}, {2+.5*sin(\x)});
    \draw [black,thick,domain=180:270] 
        plot ({1+.5*cos(\x)}, {1+.5*sin(\x)});    
        
    \elbow[thick](7,3);
    \elbow[thick](7,4);
    \elbow[thick](4,2);
    \elbow[thick](7,2);
    \elbow[thick](2,1);
    \elbow[thick](4,1);
    \cross[thick](7,1);
    \draw [black,thick] (2,0.5)--(3,0.5){};
    \draw [black,thick] (4,0.5)--(6,0.5){};
    \draw [black,thick] (4,1.5)--(6,1.5){};
    \draw [black,thick] (5,2.5)--(6,2.5){};
    \draw [black,thick] (0.5,1)--(0.5,4){};
    \draw [black,thick] (1.5,1)--(1.5,4){};
    \draw [black,thick] (2.5,2)--(2.5,4){};
    \draw [black,thick] (3.5,2)--(3.5,4){};
    \draw [black,thick] (4.5,3)--(4.5,4){};

    \node[label={\footnotesize$\st(D)$}] at (3.5,-1.6){};
                  
	\draw[very thin, color=blue!50] (0,0) grid (7,4);
\end{scope}

\begin{scope}[scale=.48, yshift=-180]
\node[label={\footnotesize$G(D)$}] at (3.5,-1.6){};
\draw[very thin, color=gray!28] (0,0) grid (7,4);
\nnode[color=black, fill](s1) at (4.5, 3.5) {};
\nnode[color=black, fill](s2) at (2.5, 2.5) {};
\nnode[color=black, fill](s3) at (0.5, 1.5) {};
\nnode[color=black, fill](s4) at (5.5, 0.5) {};
\nnode[color=black, fill=white](t1) at (0.5, 4) {};
\nnode[color=black, fill=white](t2) at (1.5, 4) {};
\nnode[color=black, fill=white](t3) at (2.5, 4) {};
\nnode[color=black, fill=white](t4) at (3.5, 4) {};
\nnode[color=black, fill=white](t5) at (4.5, 4) {};
\nnode[color=black, fill=white](t6) at (5.5, 4) {};
\nnode[color=black, fill=white](t7) at (6.5, 4) {};
\nnode[color=mediumseagreen, fill]() at (6.5, 0.5){};
\nnode[color=mediumseagreen, fill]() at (1.5, 1.5){};
\nnode[color=mediumseagreen, fill]() at (3.5, 1.5){};
\nnode[color=mediumseagreen, fill]() at (3.5, 2.5){};
\nnode[color=mediumseagreen, fill]() at (5.5, 2.5){};
\nnode[color=mediumseagreen, fill]() at (6.5, 2.5){};
\nnode[color=mediumseagreen, fill]() at (5.5, 3.5){};
\draw [color=mediumseagreen, thick] (s1) -- (t5){};
\draw [color=mediumseagreen, thick] (s2) -- (t3){};
\draw [color=mediumseagreen, thick] (s3) -- (t1){};
\draw [color=mediumseagreen, thick] (s4) -- (t6){};
\draw [color=mediumseagreen, thick] (s1) -- (5.5,3.5){};
\draw [color=mediumseagreen, thick] (s2) -- (6.5,2.5){};
\draw [color=mediumseagreen, thick] (s3) -- (3.5,1.5){};
\draw [color=mediumseagreen, thick] (s4) -- (6.5,0.5){};
\draw [color=mediumseagreen, thick] (1.5,1.5) -- (t2){};
\draw [color=mediumseagreen, thick] (3.5,1.5) -- (t4){};
\draw [color=mediumseagreen, thick] (6.5,0.5) -- (t7){};	
    \node[] at (8,1.7){\scriptsize$\st_3$};
    \node[] at (8,1){$\longrightarrow$};	
\end{scope}

\begin{scope}[scale=.48, xshift=255, yshift=-180]
\draw[very thin, color=gray!28] (0,0) grid (7,4);
\nnode[color=black, fill](s1) at (4.5, 3.5) {};
\nnode[color=black, fill](s2) at (2.5, 2.5) {};
\nnode[color=black, fill](s3) at (5.5, 1.5) {};
\nnode[color=black, fill](s4) at (0.5, 0.5) {};
\nnode[color=black, fill=white](t1) at (0.5, 4) {};
\nnode[color=black, fill=white](t2) at (1.5, 4) {};
\nnode[color=black, fill=white](t3) at (2.5, 4) {};
\nnode[color=black, fill=white](t4) at (3.5, 4) {};
\nnode[color=black, fill=white](t5) at (4.5, 4) {};
\nnode[color=black, fill=white](t6) at (5.5, 4) {};
\nnode[color=black, fill=white](t7) at (6.5, 4) {};
\nnode[color=mediumseagreen, fill]() at (1.5, 0.5){};
\nnode[color=mediumseagreen, fill]() at (3.5, 0.5){};
\nnode[color=mediumseagreen, fill]() at (3.5, 2.5){};
\nnode[color=mediumseagreen, fill]() at (5.5, 2.5){};
\nnode[color=mediumseagreen, fill]() at (6.5, 1.5){};
\nnode[color=mediumseagreen, fill]() at (6.5, 2.5){};
\nnode[color=mediumseagreen, fill]() at (5.5, 3.5){};
\draw [color=mediumseagreen, thick] (s1) -- (t5){};
\draw [color=mediumseagreen, thick] (s2) -- (t3){};
\draw [color=mediumseagreen, thick] (s3) -- (t6){};
\draw [color=mediumseagreen, thick] (s4) -- (t1){};
\draw [color=mediumseagreen, thick] (s1) -- (5.5,3.5){};
\draw [color=mediumseagreen, thick] (s2) -- (6.5,2.5){};
\draw [color=mediumseagreen, thick] (s3) -- (6.5,1.5){};
\draw [color=mediumseagreen, thick] (s4) -- (3.5,0.5){};
\draw [color=mediumseagreen, thick] (1.5,0.5) -- (t2){};
\draw [color=mediumseagreen, thick] (3.5,0.5) -- (t4){};
\draw [color=mediumseagreen, thick] (6.5,1.5) -- (t7){};
    \node[] at (8,2.7){\scriptsize$\st_2$};
    \node[] at (8,2){$\longrightarrow$};
\end{scope}

\begin{scope}[scale=.48, xshift=510, yshift=-180]
\draw[very thin, color=gray!28] (0,0) grid (7,4);
\nnode[color=black, fill](s1) at (4.5, 3.5) {};
\nnode[color=black, fill](s2) at (5.5, 2.5) {};
\nnode[color=black, fill](s3) at (2.5, 1.5) {};
\nnode[color=black, fill](s4) at (0.5, 0.5) {};
\nnode[color=black, fill=white](t1) at (0.5, 4) {};
\nnode[color=black, fill=white](t2) at (1.5, 4) {};
\nnode[color=black, fill=white](t3) at (2.5, 4) {};
\nnode[color=black, fill=white](t4) at (3.5, 4) {};
\nnode[color=black, fill=white](t5) at (4.5, 4) {};
\nnode[color=black, fill=white](t6) at (5.5, 4) {};
\nnode[color=black, fill=white](t7) at (6.5, 4) {};
\nnode[color=mediumseagreen, fill]() at (1.5, 0.5){};
\nnode[color=mediumseagreen, fill]() at (3.5, 0.5){};
\nnode[color=mediumseagreen, fill]() at (3.5, 1.5){};
\nnode[color=mediumseagreen, fill]() at (6.5, 1.5){};
\nnode[color=mediumseagreen, fill]() at (6.5, 2.5){};
\nnode[color=mediumseagreen, fill]() at (5.5, 3.5){};
\draw [color=mediumseagreen, thick] (s1) -- (t5){};
\draw [color=mediumseagreen, thick] (s2) -- (t6){};
\draw [color=mediumseagreen, thick] (s3) -- (t3){};
\draw [color=mediumseagreen, thick] (s4) -- (t1){};
\draw [color=mediumseagreen, thick] (s1) -- (5.5,3.5){};
\draw [color=mediumseagreen, thick] (s2) -- (6.5,2.5){};
\draw [color=mediumseagreen, thick] (s3) -- (6.5,1.5){};
\draw [color=mediumseagreen, thick] (s4) -- (3.5,0.5){};
\draw [color=mediumseagreen, thick] (1.5,0.5) -- (t2){};
\draw [color=mediumseagreen, thick] (3.5,0.5) -- (t4){};
\draw [color=mediumseagreen, thick] (6.5,1.5) -- (t7){};
    \node[] at (8,3.7){\scriptsize$\st_1$};
    \node[] at (8,3){$\longrightarrow$};
\end{scope}

\begin{scope}[scale=.48, xshift=765, yshift=-180]
\node[label={\footnotesize$G(\st(D))$}] at (3.5,-1.6){};
\draw[very thin, color=gray!28] (0,0) grid (7,4);
\nnode[color=black, fill](s1) at (5.5, 3.5) {};
\nnode[color=black, fill](s2) at (4.5, 2.5) {};
\nnode[color=black, fill](s3) at (2.5, 1.5) {};
\nnode[color=black, fill](s4) at (0.5, 0.5) {};
\nnode[color=black, fill=white](t1) at (0.5, 4) {};
\nnode[color=black, fill=white](t2) at (1.5, 4) {};
\nnode[color=black, fill=white](t3) at (2.5, 4) {};
\nnode[color=black, fill=white](t4) at (3.5, 4) {};
\nnode[color=black, fill=white](t5) at (4.5, 4) {};
\nnode[color=black, fill=white](t6) at (5.5, 4) {};
\nnode[color=black, fill=white](t7) at (6.5, 4) {};
\nnode[color=mediumseagreen, fill]() at (1.5, 0.5){};
\nnode[color=mediumseagreen, fill]() at (3.5, 0.5){};
\nnode[color=mediumseagreen, fill]() at (3.5, 1.5){};
\nnode[color=mediumseagreen, fill]() at (6.5, 1.5){};
\nnode[color=mediumseagreen, fill]() at (6.5, 2.5){};
\nnode[color=mediumseagreen, fill]() at (6.5, 3.5){};
\draw [color=mediumseagreen, thick] (s1) -- (t6){};
\draw [color=mediumseagreen, thick] (s2) -- (t5){};
\draw [color=mediumseagreen, thick] (s3) -- (t3){};
\draw [color=mediumseagreen, thick] (s4) -- (t1){};
\draw [color=mediumseagreen, thick] (s1) -- (6.5,3.5){};
\draw [color=mediumseagreen, thick] (s2) -- (6.5,2.5){};
\draw [color=mediumseagreen, thick] (s3) -- (6.5,1.5){};
\draw [color=mediumseagreen, thick] (s4) -- (3.5,0.5){};
\draw [color=mediumseagreen, thick] (1.5,0.5) -- (t2){};
\draw [color=mediumseagreen, thick] (3.5,0.5) -- (t4){};
\draw [color=mediumseagreen, thick] (6.5,1.5) -- (t7){};\end{scope}

\end{tikzpicture}    
\caption{The top line is an example of the standardization of a partial flag positroid pipe dream.
The bottom line shows the induced action of the standardization operators on the corresponding directed graphs.}
    \label{fig.std_example}
\end{figure}
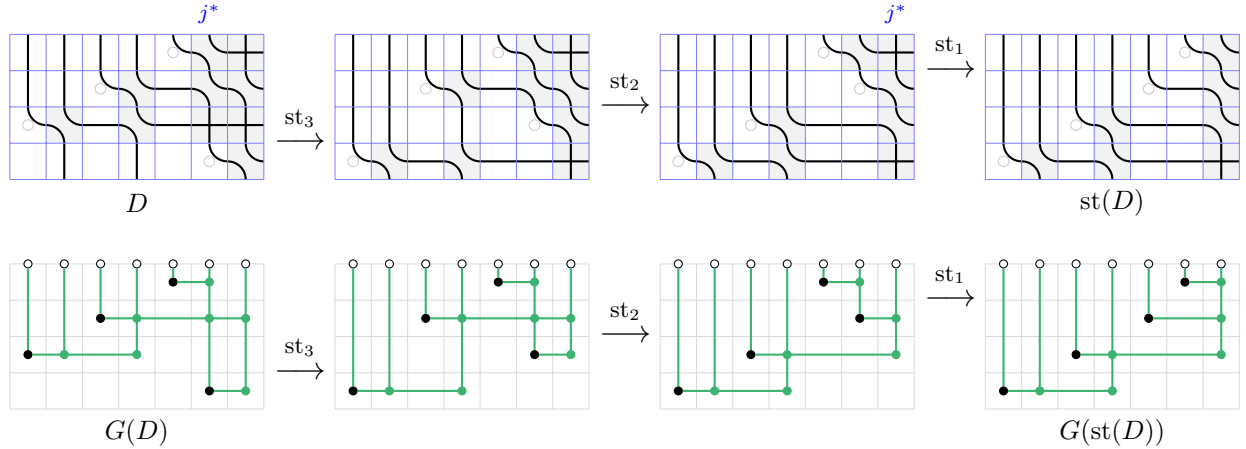

\subsection{Flag positroids and flag positroid pipe dreams}\label{subsec.fp_fpp}
In this section so far, we have shown how a standardized partial FPP $D^{\sLe}(P)$ representing a positroid $P$ can be extended by a single row to obtain a pipe dream representation $D(P,Q) = D^{\sLe}(P)\cdot C$ of a nonnegatively representable elementary positroid quotient $P\unlhd_q Q$.
Applying the standardization operation to $D(P,Q)$ allows us to then identify the positroid $Q$ via its $\Le$-pipe dream.
But as we have seen in Theorem~\ref{thm.sti_preserves_bases}, the standardization operation on partial FPPs preserves the associated bases of the positroids they encode.
In the final part of this section, we show that we can extend the concept of unblocked columns to all partial FPPs, which eliminates the need to work strictly with standardized partial FPPs, and subsequently allows us to construct a complete flag positroid pipe dream by successively appending rows according to the unblocked columns at each rank.
We conclude by showing in Theorem~\ref{thm.fpp_richardson} that the complete flag positroid pipe dream $\FPP(u,v)$ encodes the flag bases of the flag positroid $P_\bullet(u,v)$ that is associated to any point in the Richardson cell $\calR_{u,v}^{>0}$.

Thus we begin by extending the definition of an unblocked column (Definition~\ref{defn.Leblocked}) to all partial FPPs.
This definition allows us to build partial FPPs row by row so that the pipe dream is $\Gamma$-free in the sense of Definition~\ref{defn.gammapatt}.

\begin{definition}\label{defn.blocked}
Let $D$ be a partial flag positroid pipe dream. 
A column of $D$ is \defn{blocked} if it contains a cross tile whose horizontal pipe exits at the bottom boundary of $D$. 
Otherwise the column is \defn{unblocked}.
\end{definition}
For example, referring to Figure~\ref{fig.OKgamma}, let $D|_2$ be the restriction of $\FPP(2413,4231)$ to its first two rows.  
Then the set of unblocked columns of $D|_2$ is $U =\{1,3\}$.
We note that Definition~\ref{defn.blocked} specializes to Definition~\ref{defn.Leblocked} when $D|_k$ has pivot columns in decreasing order.

Recall from Definition~\ref{defn.FPP} that for permutations $u\leq v$ in Bruhat order, $D=\FPP(u,v)$ is the flag positroid pipe dream where the pivot elbows are at $(i,u_i)$ for all $i=1,\ldots, n$ and $v$ is the exit permutation.
Also recall from Proposition~\ref{prop.sti_preserves_FPP} that $\st_i(D)$ is an FPP.

\begin{proposition}\label{prop.pipe_exchange}
Given $u,v\in\fS_n$ with $u\leq v$ in Bruhar order, suppose $D=\FPP(u,v)$ has $u_i < u_{i+1}$ for some $1\leq i \leq n$.
Let $D'= \FPP(u',v') = \st_i(D)$. 
Then $u'=us_i$ and 
\[v' = \begin{cases}
vs_i, & \hbox{if $\exists\, j^*$ such that $D_{i,j^*}$ is a cross and $D_{i+1,j^*}$ is an elbow or pivot elbow},\\
v, & \hbox{otherwise}.
\end{cases}
\]
That is, if the pipes that exit the $i$-th and $(i+1)$-st rows of $D$ cross, then $\st_i$ switches their exit rows. 
Otherwise $\st_i$ fixes the exit row or column of each pipe.
\end{proposition}
\begin{proof}
From the definition of $\st_i$, the pivot elbows in the $i$-th and $(i+1)$-st rows of $D$ are switched, therefore $u'= us_i$.

Next, suppose there exists $j^*\geq u_{i+1}$ such that $D_{i,j^*}$ is a cross and $D_{i+1,j^*}$ is an elbow or pivot elbow.
Since $D$ is $\Gamma$-free, then $D_{i,j^*}$ is the last elbow tile in its row.
If $j^*=u_{i+1}$, then the effect of $\st_i$ is to simply switch the $i$-th and $(i+1)$-st rows of $D$, necessarily removing a cross tile and replacing $D'_{i+1, u_{i+1}}$ with a horizontal pipe.
Thus the exit rows of pipes $v_i$ and $v_{i+1}$ are switched in $D'$.
Moreover, the order of the pipes which exit at the bottom boundary of $D$ are preserved by $\st_i$, and therefore $v'=vs_i$.
Otherwise $j^* > u_{i+1}$, and the effect of $\st_i$ is to switch the $i$-th and $(i+1)$-st rows of $D$ from the $j^*$-th column until the end, removing a cross tile and replacing $D'_{i+1, j^*}$ with an elbow tile.
Thus the exit rows of pipes $v_i$ and $v_{i+1}$ are switched in $D'$ while the order of the pipes which exit at the bottom boundary of $D$ are preserved by $\st_i$, and therefore $v'=vs_i$ as well.

Lastly, if $j^*$ does not exist (i.e. $j^*=\infty$), then the effect of $\st_i$ is to switch the $i$-th and $(i+1)$-st rows of $D$ up to and including the $u_{i+1}$-st column only (and again, necessarily replace $D'_{i+1, u_{i+1}}$ with a horizontal pipe). 
This does not alter the exit rows or columns of any pipe, therefore $v'=v$ in this case.
\end{proof}
See Figure~\ref{fig:exchange_pipes} for an illustration of these cases.

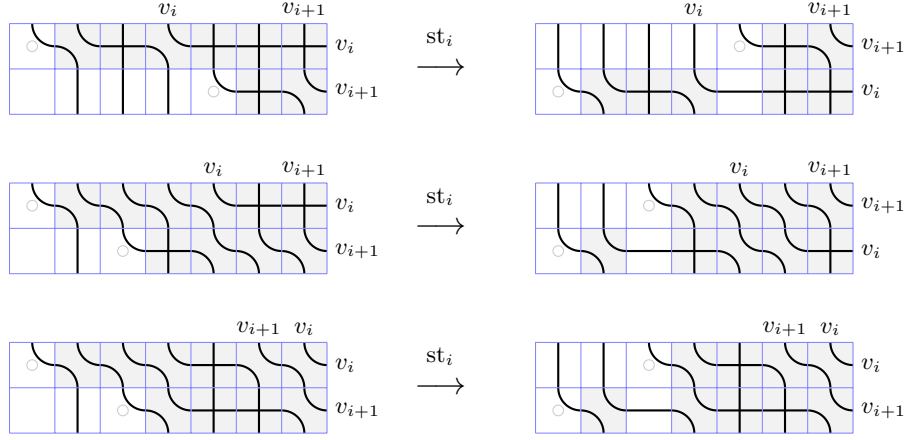
\begin{figure}[ht!]
\begin{tikzpicture}

\def\vpipe[#1](#2,#3){
    \draw [black, #1] (#2-0.5, #3-1) -- (#2-0.5, #3);
}
\def\hpipe[#1](#2,#3){
    \draw [black, #1] (#2-1, #3-0.5) -- (#2, #3-0.5);
}

\begin{scope}[scale=.6]
\vertex[color=gray!40] at (0.5, 1.5) {};
\vertex[color=gray!40] at (4.5, 0.5) {};

\node[label=right:{\scriptsize$v_i$}] at (6.7, 1.5){};
\node[label=right:{\scriptsize$v_{i+1}$}] at (6.7, 0.5){};

\node[] at (3.5, 2.3) {\scriptsize$v_i$};
\node[] at (6.5, 2.3) {\scriptsize$v_{i+1}$};

\pivotelbow[thick](1,2);
\pivotelbow[thick](5,1);
\elbow[thick](2,2);
\elbow[thick](4,2);
\elbow[thick](7,1);

\cross[thick](3,2);
\cross[thick](5,2);
\cross[thick](6,2);
\cross[thick](7,2);
\cross[thick](6,1);

\vpipe[thick](2,1);
\vpipe[thick](3,1);
\vpipe[thick](4,1);

\draw[very thin, color=blue!50] (0,0) grid (7,2);

\node[] at (9.5,1.7){\scriptsize$\st_i$};
\node[] at (9.5,1){$\longrightarrow$};
\end{scope}
\begin{scope}[scale=.6, xshift=330]
\vertex[color=gray!40] at (0.5, 0.5) {};
\vertex[color=gray!40] at (4.5, 1.5) {};

\node[label=right:{\scriptsize$v_i$}] at (6.7, 0.5){};
\node[label=right:{\scriptsize$v_{i+1}$}] at (6.7, 1.5){};

\node[] at (3.5, 2.3) {\scriptsize$v_i$};
\node[] at (6.5, 2.3) {\scriptsize$v_{i+1}$};

\pivotelbow[thick](1,1);
\pivotelbow[thick](5,2);
\elbow[thick](2,1);
\elbow[thick](4,1);
\elbow[thick](7,2);

\cross[thick](3,1);
\cross[thick](6,1);
\cross[thick](7,1);
\cross[thick](6,2);

\hpipe[thick](5,1);

\vpipe[thick](1,2);
\vpipe[thick](2,2);
\vpipe[thick](3,2);
\vpipe[thick](4,2);

\draw[very thin, color=blue!50] (0,0) grid (7,2);
\end{scope}

\begin{scope}[scale=.6, yshift=-100]
\vertex[color=gray!40] at (0.5, 1.5) {};
\vertex[color=gray!40] at (2.5, 0.5) {};

\node[label=right:{\scriptsize$v_i$}] at (6.7, 1.5){};
\node[label=right:{\scriptsize$v_{i+1}$}] at (6.7, 0.5){};

\node[] at (4.5, 2.3) {\scriptsize$v_i$};
\node[] at (6.5, 2.3) {\scriptsize$v_{i+1}$};

\pivotelbow[thick](1,2);
\pivotelbow[thick](3,1);

\elbow[thick](2,2);
\elbow[thick](4,2);
\elbow[thick](7,1);
\elbow[thick](3,2);
\elbow[thick](5,2);
\elbow[thick](6,1);
\elbow[thick](5,1)

\cross[thick](4,1);
\cross[thick](6,2);
\cross[thick](7,2);

\vpipe[thick](2,1);

\draw[very thin, color=blue!50] (0,0) grid (7,2);

\node[] at (9.5,1.7){\scriptsize$\st_i$};
\node[] at (9.5,1){$\longrightarrow$};
\end{scope}
\begin{scope}[scale=.6, xshift=330, yshift=-100]
\vertex[color=gray!40] at (0.5, 0.5) {};
\vertex[color=gray!40] at (2.5, 1.5) {};

\node[label=right:{\scriptsize$v_i$}] at (6.7, 0.5){};
\node[label=right:{\scriptsize$v_{i+1}$}] at (6.7, 1.5){};

\node[] at (4.5, 2.3) {\scriptsize$v_i$};
\node[] at (6.5, 2.3) {\scriptsize$v_{i+1}$};

\pivotelbow[thick](1,1);
\pivotelbow[thick](3,2);

\elbow[thick](4,2);
\elbow[thick](7,2);
\elbow[thick](5,2);
\elbow[thick](6,1);
\elbow[thick](6,2);
\elbow[thick](5,1)
\elbow[thick](2,1);

\cross[thick](4,1);
\cross[thick](7,1);

\vpipe[thick](1,2);
\vpipe[thick](2,2);

\hpipe[thick](3,1);

\draw[very thin, color=blue!50] (0,0) grid (7,2);
\end{scope}

\begin{scope}[scale=.6, yshift=-200]
\vertex[color=gray!40] at (0.5, 1.5) {};
\vertex[color=gray!40] at (2.5, 0.5) {};

\node[label=right:{\scriptsize$v_i$}] at (6.7, 1.5){};
\node[label=right:{\scriptsize$v_{i+1}$}] at (6.7, 0.5){};

\node[] at (6.5, 2.3) {\scriptsize$v_i$};
\node[] at (5.5, 2.3) {\scriptsize$v_{i+1}$};

\pivotelbow[thick](1,2);
\pivotelbow[thick](3,1);

\elbow[thick](2,2);
\elbow[thick](4,2);
\elbow[thick](7,1);
\elbow[thick](3,2);
\cross[thick](5,2);
\cross[thick](6,1);
\cross[thick](5,1)

\elbow[thick](4,1);
\elbow[thick](6,2);
\elbow[thick](7,2);

\vpipe[thick](2,1);

\draw[very thin, color=blue!50] (0,0) grid (7,2);

\node[] at (9.5,1.7){\scriptsize$\st_i$};
\node[] at (9.5,1){$\longrightarrow$};
\end{scope}
\begin{scope}[scale=.6, xshift=330, yshift=-200]
\vertex[color=gray!40] at (0.5, 0.5) {};
\vertex[color=gray!40] at (2.5, 1.5) {};

\node[label=right:{\scriptsize$v_i$}] at (6.7, 1.5){};
\node[label=right:{\scriptsize$v_{i+1}$}] at (6.7, 0.5){};
\node[] at (6.5, 2.3) {\scriptsize$v_i$};
\node[] at (5.5, 2.3) {\scriptsize$v_{i+1}$};

\pivotelbow[thick](1,1);
\pivotelbow[thick](3,2);
\elbow[thick](4,2);
\elbow[thick](7,1);
\cross[thick](5,2);
\cross[thick](6,1);
\cross[thick](5,1)
\elbow[thick](4,1);
\elbow[thick](6,2);
\elbow[thick](7,2);
\elbow[thick](2,1);
\vpipe[thick](1,2);
\vpipe[thick](2,2);
\hpipe[thick](3,1);

\draw[very thin, color=blue!50] (0,0) grid (7,2);
\end{scope}
\end{tikzpicture}
    \caption{An illustration of the standardization operation $\st_i$ on the exit permutation $v$ of a flag positroid pipe dream.
The first two rows are the cases $j^*= u_{i+1}$ and $j^*> u_{i+1}$, and the last row is the case when $j^*$ does not exist.
}
    \label{fig:exchange_pipes}
\end{figure}

An immediate consequence of Proposition~\ref{prop.pipe_exchange} is that the standardization operation does not change the boundary where a pipe exits the pipe dream.
\begin{corollary} \label{cor.right}
Let $D\in \setFPP_k(n)$. 
A pipe exits at the right boundary of $D$ if and only if it exits at the right boundary of $\st_i(D)$.
\qed
\end{corollary}

\begin{corollary}\label{cor.unblocked-unchanged}
The standardization operation on a partial flag positroid pipe dream preserves its set of unblocked columns.
\end{corollary}
\begin{proof}
Let $D\in\setFPP_k(n)$ and $i=1,\ldots, k-1$. 
From the definition of $\st_i$ (and with the help of Figure~\ref{fig:exchange_pipes}), one can see that for each pipe $p$ that exits at the bottom boundary of the $(i+1)$-st row of $D$, that pipe enters the top of the $i$-th row in the same column before and after $\st_i$ is applied, and the columns in the $i$-th and $(i+1)$-st rows containing a cross that pipe $p$ passes through horizontally do not change.
This implies that for any pipe exiting the bottom boundary of $D$, the columns containing a cross tile through which this pipe passes horizontally are unchanged by $\st_i$.
Therefore, the set of blocked columns of $D$, and hence the set of unblocked columns, is invariant under $\st_i$.
This holds for all $i=1,\ldots,k-1$, and therefore, $\st$ preserves the set of unblocked columns of $D$.
\end{proof}

\begin{definition}
Let $D \in \setFPP(n)$. 
For $k=1,\ldots, n-1$, let $D|_k$ denote the restriction of $D$ to its first $k$ rows.
Define the sequence of positroids
\[P_\bullet(D) = (P_1, \ldots, P_{n-1})
\]  
where $P_k$ is the positroid corresponding to the partial FPP $\st(D|_k)=D^{\sLe}(P_k)$.
(We may let $P_0$ be the unique rank $0$ matroid and $P_n$ be the unique rank $n$ matroid on $[n]$.)
\end{definition}

Since $D\in \setFPP(n)$ is $\Gamma$-free, then each restriction $D|_k$ is $\Gamma$-free and hence is a partial FPP that encodes a rank $k$ positroid.
By Theorem~\ref{thm.sti_preserves_bases} and Corollary~\ref{cor.unblocked-unchanged}, $D|_k$ and $\st(D|_k)$ encode the same positroid and have the same set of unblocked columns, and the unblocked columns of $D_k$ are precisely the columns in which adding an elbow or a pivot elbow in the next row would maintain the $\Gamma$-free property of the pipe dream.
Therefore, it follows from Theorem~\ref{thm.char_Q} that 
\[P_1 \unlhd_q P_2 \unlhd_q \cdots \unlhd_q P_{n-1},\]
and so $P_\bullet(D)$ is a positroid flag matroid.
But in fact, $P_\bullet(D)$ is a flag positroid; 
each nonnegatively representable quotient $P_k \unlhd_q P_{k+1}$ 
forms an oriented matroid quotient, so by Proposition 7.10 of Boretsky, Eur, and Williams~\cite{BEW24}, $P_\bullet(D)$ is a complete flag positroid.

This brings us to the final point in this section, which is to show that the flag positroid $P_\bullet(D)$ associated to $D=\FPP(u,v)$ is indeed the flag positroid associated to any point in the positive Richardson cell $\calR_{u,v}^{>0}$, for $u\leq v$ in Bruhat order.

Given a permutation $z=z_1\cdots z_n \in \fS_n$, denote the restriction of $z$ to its first $k$ values by $z[k]=z_1\cdots z_k$.
It was shown by various authors (see Kodama and Williams~\cite[Lemma 3.11]{KW15}, Billey and Weaver~\cite[Theorem 1.4]{BW25}, and Boretsky, Eur, and Williams~\cite[Lemma 7.7]{BEW24}) that the bases of the $k$-th constituent of the flag positroid associated to any point in $\calR_{u,v}^{>0}$ are obtained by restricting the permutations in the Bruhat interval $[u,v]$ to their first $k$ values, hence it suffices to show that for each $k$, the lexicographically minimal and maximal bases of the $k$-th constituent of $P_\bullet(D)$ are $u[k]$ and $v[k]$.

\begin{lemma}\label{lem:lexminmax-from-diagram}
Let $P$ be a rank $k$ positroid with partial FPP $D^{\sLe}(P)$.
The lexicographically minimal basis of $P$ is the column set of the pivot elbows of $D^{\sLe}(P)$, and the lexicographically maximal basis of $P$ is the set of pipes exiting at the right boundary of $D^{\sLe}(P)$.
\end{lemma}
\begin{proof}
Recall from Section~\ref{subsec.le_parallel} that $P$ is associated with an interval $[w_0w,w_0t]$ in the Bruhat order, where $w_0w$ is a permutation with a unique ascent in position $k$.
By~\cite[Theorem 3.8]{Postnikov06}, since the cell decompositions of $\Gr_{k,n}^{\geq0}$ given by Postnikov and Rietsch coincide, and the Richardson cell $\calR_{w_0w,w_0t}^{>0}$ of the complete flag variety $\Fl_n^{\geq0}$ projects to the positroid cell of $\Gr_{k,n}^{\geq0}$ indexed by $[w_0w,w_0t]$, then $P$ is the $k$-th constituent of the flag positroid associated to any point in $\calR_{w_0w,w_0t}^{>0}$.

By~\cite[Lemma 7.7]{BEW24} of Boretsky et al., then $w_0w[k]$ is the lexicographically minimal basis of $P$, and $w_0t[k]$ is the lexicographically maximal basis of $P$.
In terms of flag positroid pipe dreams, it follows from Theorem~\ref{thm.rpd_is_le} that the trivial completion of $D^{\sLe}(P)$ is $\FPP(w_0w, w_0t)$, and so $w_0w[k]$ represents the column set of the pivot elbows of $D^{\sLe}(P)$, and $w_0t[k]$ represents the set of pipes exiting at the right boundary of $D^{\sLe}(P)$.
\end{proof}
We remark that one can also give a direct proof of the previous lemma.
If the trivial completion of $D^{\sLe}(P)$ is $D=\FPP(u,v)$, then by considering the associated directed graph $G(D)$ it follows directly that the lexicographically minimal basis of $P$ is the column set of pivot elbows because they constitute the source vertices of $G(D|_k)$, and following Remark~\ref{rem.lexminmax_bases}, one can show that the lexicographically maximal basis of $P$ is the sink set of paths which exit at the right boundary of $D|_k$.

\begin{theorem}\label{thm.fpp_richardson}
For $u\leq v$ in Bruhat order, let $D=\FPP(u,v)$. 
Then $P_\bullet(D)$ is the flag positroid $P_\bullet(u,v)$ that is associated to any point in the positive Richardson cell $\calR_{u,v}^{>0}$.
\end{theorem}
\begin{proof}
Consider the $k$-th constituent $P_k$ of $P_\bullet(D)$. Applying the standardization operation to the partial FPP $D|_k$ does not change the column set of the first $k$ pivot elbows, and it does not change the set of the first $k$ pipes which exit at the right boundary of $D$ (Corollary~\ref{cor.right}).
Therefore, since $\st(D|_k) = D^{\sLe}(P_k)$, it follows from Lemma~\ref{lem:lexminmax-from-diagram} that the lexicographically minimal basis of the $k$-th constituent of $P_\bullet(D)$ is $u[k]$ and the lexicographically maximal basis is $v[k]$, for $k=1,\ldots, n-1$.
Since $P_\bullet(D)$ is a flag positroid, then Lemma 7.7 of Boretsky, Eur, and Williams~\cite{BEW24} identifies $P_\bullet(D)$ as the flag positroid associated to any point in the positive Richardson cell $\calR_{u,v}^{>0}$.
\end{proof}

\begin{example}\label{eg.crossingGD}
The pipe dream $D=\FPP(2413,4231)$ in Figure~\ref{fig.OKgamma} encodes the complete flag positroid $P_\bullet=(P_1,P_2,P_3)$ whose respective bases are 
\[\calB(P_1) = \{2,4\}, \quad 
\calB(P_2)=\{24\}, \quad 
\calB(P_3)= \{124, 234\}.
\]
These bases can be computed by considering admissible $k$-collections of paths in the associated directed graphs $G(D|_k)$ for $k=1,2,3$.

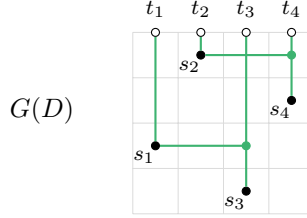
\begin{center}
\begin{tikzpicture}
\begin{scope}[scale=.6, yshift=0]
\node[label={\footnotesize$G(D)$}] at (-2,1.5){};
\draw[very thin, color=gray!28] (0,0) grid (4,4);
\nnode[color=black, fill](s1) at (1.5, 3.5) {};
	\node() at (1.25,3.25){\tiny$s_2$};
\nnode[color=black, fill](s2) at (3.5, 2.5) {};
	\node() at (3.25,2.25){\tiny$s_4$};
\nnode[color=black, fill](s3) at (0.5, 1.5) {};
	\node() at (0.25,1.25){\tiny$s_1$};
\nnode[color=black, fill](s4) at (2.5, 0.5) {};
	\node() at (2.25,0.25){\tiny$s_3$};
\nnode[color=black, fill=white, label=above:{\tiny{$t_1$}}](t1) at (0.5, 4) {};
\nnode[color=black, fill=white, label=above:{\tiny{$t_2$}}](t2) at (1.5, 4) {};
\nnode[color=black, fill=white, label=above:{\tiny{$t_3$}}](t3) at (2.5, 4) {};
\nnode[color=black, fill=white, label=above:{\tiny{$t_4$}}](t4) at (3.5, 4) {};
\nnode[color=mediumseagreen, fill]() at (2.5, 1.5){};
\nnode[color=mediumseagreen, fill]() at (3.5, 3.5){};
\draw [color=mediumseagreen, thick] (s1) -- (t2){};
\draw [color=mediumseagreen, thick] (s2) -- (t4){};
\draw [color=mediumseagreen, thick] (s3) -- (t1){};
\draw [color=mediumseagreen, thick] (s4) -- (t3){};
\draw [color=mediumseagreen, thick] (s1) -- (3.5,3.5){};
\draw [color=mediumseagreen, thick] (s3) -- (2.5,1.5){};
\end{scope}
\end{tikzpicture}
\end{center}
\end{example}

\section{Elementary quotients via decorated permutations}\label{sec.decperm}

Using flag positroid pipe dreams, we give a new proof in Theorem~\ref{thm.covering_decperms} of the theorem of Chen et al.~\cite[Theorem 1.5]{CFGSZ26} in the case of nonnegatively representable elementary quotients.
A strength of our approach is that we can explicitly characterize the ``frozen'' subsets $A$ which appear in the theorem. 
This partially addresses~\cite[Remark 3.6]{CFGSZ26}, which asked for ``a concise characterization - based solely on $\pi$ - of the sets $A$''.

We conclude this article by introducing the poset $(\Pi_n,\unlhd_q)$ of nonnegatively representable elementary positroid quotients and showing some of the properties it possesses.

\subsection{Decorated permutations}\label{subsec.decperms}

The following definition of a decorated permutation is equivalent to the standard ones in the literature (see~\cite[Definition 4.5]{ARW16} or~\cite[Definition 13.3]{Postnikov06} for example), where for our purposes it is convenient to extend the colouring/decorating notation to all of $[n]$ and not just defined on the fixed points of $[n]$.
\begin{definition}\label{defn.decperm}
A \defn{decorated permutation} on $[n]$ is a bijection $\pi:[n] \rightarrow [n]$ together with a colouring function $c:[n]\rightarrow\{1,2\}$ satisfying
\begin{itemize}
\item[(i)] $c(j)=1$ if $\pi(j) < j$, and
\item[(ii)] $c(j)=2$ if $\pi(j) > j$.
\end{itemize}
In other words, if we regard fixed points $j$ such that $c(j)=2$ as weak excedances of $\pi$, and fixed points $j$ such that $c(j)=1$ as non-weak excedances of $\pi$, then the colouring function is given by  $c(j)=2$ if and only if $j$ is a weak excedance of $\pi$.
If $c(j)=i$, we say that $j$ is an \defn{$i$-coloured position} and $\pi(j)$ is an \defn{$i$-coloured value} of $\pi$.
In one-line notation, we notate the colouring function by decorating $\pi(j)$ with an underline if $j$ is a $1$-coloured position, and $\pi(j)$ with an overline if $j$ is a $2$-coloured position.
Let $\fD_n$ denote the set of all decorated permutations on $[n]$.
\end{definition}

See Example~\ref{eg.Ddecperm} where the decorated permutation $\pi$ has (weak) excedances $1,4,6$.  
The only fixed point of $\pi$ is $3$ and it is $1$-coloured.

The bijection between decorated permutations and $\Le$-diagrams is due to Postnikov~\cite[Theorem 20.3]{Postnikov06}.
We will rephrase it in terms of the partial flag positroid pipe dreams $D^{\sLe}$.
Given $D^{\sLe}\in \setFPP_k^{\sLe}(n)$, write the pivot columns in decreasing order down the right side of the diagram and the remaining elements in $[n]$ in increasing order along the bottom of the diagram, and decorate the column indices along the top of the diagram with an underline if the pipe originating in that column ends at the bottom edge, and with an overline if the pipe ends at the right boundary.
Following the pipes starting from the southeast boundary of the diagram to the top yields the decorated permutation $\pi$.

Recall from Section~\ref{subsec.le_parallel} that each $D^{\sLe}
$/$\Le$-pipe dream corresponds with a pair of permutations $t\leq w$ where $w$ is a Grassmannian permutation.
Postnikov showed~\cite[Theorem 20.3]{Postnikov06} that the decorated permutation is given by $\pi =w_0tw^{-1}w_0$.
This is rephrased in terms of FPPs as follows.
Given $D^{\sLe}\in \setFPP_k(n)$, define its \defn{trivial completion} to be $\FPP(u,v)$ whose first $k$ rows is $D^{\sLe}$ and whose pivot elbows in the last $n-k$ rows are in decreasing order.
Following the notation from Section~\ref{subsec.le_parallel}, then $u=w_0w$ and $v=w_0t$, so that the decorated permutation associated to $D^{\sLe}$ is given by $\pi = w_0tw^{-1} w_0 = vu^{-1}$.
The decoration is given by $\overline{\pi(j)}$ if $j\in \{u_1,\ldots, u_k\}$ and $\underline{\pi(j)}$ otherwise.
Thus, the pipe $\pi(j)$ exits at the right boundary of $D^{\sLe}$ if it is decorated with an overline, and exits at the bottom boundary if it is decorated with an underline.

\begin{example}\label{eg.Ddecperm}
Consider the partial FPP $D^{\sLe}$ in Figure~\ref{fig.PFPP_decperm}, whose decorated permutation $\pi =\overline{5}\underline{13}\overline{9}\underline{2}\overline{7}\underline{648}$ can be read by following the pipes from the southeast boundary to the top.
Following the notation from Section~\ref{subsec.le_parallel}, we have $w_0w=\textcolor{red}{641}\textcolor{blue}{987532}$, $w_0t=795846231$ (obtained by following the pipes up in the order of $w_0w$), and one can verify that $\pi=vu^{-1}$ with the appropriate decorations.
In other words, pipes $\pi(1)=5, \pi(4)=9$ and $\pi(6)=7$ exit at the right boundary of $D^{\sLe}$ while the rest exit at the bottom boundary.
\end{example}

\begin{figure}[ht!]
\begin{tikzpicture}
\begin{scope}[scale=.5, xshift=0]
	\vertex[color=gray!40](w1) at (5.5, 2.5) {};
	\vertex[color=gray!40](w2) at (3.5, 1.5) {};
	\vertex[color=gray!40](w3) at (0.5, 0.5) {};
\draw [black,thick,domain=180:270] plot ({6+.5*cos(\x)}, {3+.5*sin(\x)});
\draw [black,thick,domain=180:270] plot ({4+.5*cos(\x)}, {2+.5*sin(\x)});
\draw [black,thick,domain=180:270] plot ({1+.5*cos(\x)}, {1+.5*sin(\x)});
\draw [black,thick] (0.5,1) -- (0.5,3);
\draw [black,thick] (1.5,1) -- (1.5,3);
\draw [black,thick] (2.5,1) -- (2.5,3);
\draw [black,thick] (3.5,2) -- (3.5,3);
\draw [black,thick] (4.5,2) -- (4.5,3);
\draw [black,thick] (3,0.5) -- (4,0.5);
\draw [black,thick] (5,0.5) -- (6,0.5);
\draw [black,thick] (5,1.5) -- (6,1.5);
\cross[thick](3,1);
\cross[thick](7,1);
\cross[thick](9,1);
\cross[thick](7,2);
\cross[thick](8,3);
\cross[thick](9,3);
\elbow[thick](2,1); 
\elbow[thick](5,1); 
\elbow[thick](8,1);
\elbow[thick](5,2);
\elbow[thick](8,2);
\elbow[thick](9,2);
\elbow[thick](7,3);
        
\draw[very thin, color=blue!50] (0,0) grid (9,3);

\node at (0.5,3.5) {\color{black}{\footnotesize$\underline{1}$}};
\node at (1.5,3.5) {\color{black}{\footnotesize$\underline{2}$}};
\node at (2.5,3.5) {\color{black}{\footnotesize$\underline{3}$}};
\node at (3.5,3.5) {\color{black}{\footnotesize$\underline{4}$}};
\node at (4.5,3.5) {\color{black}{\footnotesize$\overline{5}$}};
\node at (5.5,3.5) {\color{black}{\footnotesize$\underline{6}$}};
\node at (6.5,3.5) {\color{black}{\footnotesize$\overline{7}$}};
\node at (7.5,3.5) {\color{black}{\footnotesize$\underline{8}$}};
\node at (8.5,3.5) {\color{black}{\footnotesize$\overline{9}$}};

\node at (1.5,-.5) {\color{blue}{\footnotesize$2$}};
\node at (2.5,-.5) {\color{blue}{\footnotesize$3$}};
\node at (4.5,-.5) {\color{blue}{\footnotesize$5$}};
\node at (6.5,-.5) {\color{blue}{\footnotesize$7$}};
\node at (7.5,-.5) {\color{blue}{\footnotesize$8$}};
\node at (8.5,-.5) {\color{blue}{\footnotesize$9$}};	
\node at (9.5, 2.5) {\color{red}{\footnotesize$6$}};
\node at (9.5, 1.5) {\color{red}{\footnotesize$4$}};
\node at (9.5, 0.5) {\color{red}{\footnotesize$1$}};

\end{scope}
\end{tikzpicture}
\caption{Following the pipes in $D^{\sLe}$ from the southeast boundary yields the corresponding decorated permutation $\pi = \overline{5}\underline{13}\overline{9}\underline{2}\overline{7}\underline{648}$.
}
\label{fig.PFPP_decperm}
\end{figure}

\subsection{Elementary positroid quotients via decorated permutations}

In this section we translate the characterization of nonnegatively representable elementary positroid quotients of Theorem~\ref{thm.char_Q} from unblocked columns of a partial flag positroid pipe dream to cyclic shifts of its corresponding decorated permutation (Theorem~\ref{thm.covering_decperms}).
In the context of our work, it is more natural to describe the cyclic shift set, which is the complement of the freeze set, see Definitions~\ref{defn.rcshift} and~\ref{defn.TC}.

\begin{definition}\label{defn.unblockedposn}
Let $\pi$ be a decorated permutation on $[n]$. 
Then $\pi(j)$ is an \defn{unblocked value} of $\pi$ if $\pi(j)$ is underlined and for all $j'>j$ such that $\pi(j')$ is underlined, we have $\pi(j') > \pi(j)$.
We also say that such a $j$ is an \defn{unblocked position} of $\pi$.
\end{definition}

Note that unless $\pi=\overline{12\cdots n}$, then the largest $1$-coloured position is always unblocked.
Also observe that by definition, if the set $U=\{j_1,\ldots, j_r\}$ of unblocked positions of $\pi$ is ordered such that $j_1< \cdots < j_r$, then by definition the unblocked values of $\pi$ also satisfy $\pi(j_1) < \cdots < \pi(j_r)$.
See Example~\ref{eg.blocked_decperm}.

\begin{lemma}\label{lem:unblocked-in-decperm}
Let $D^{\sLe}\in \setFPP_k(n)$ be a partial flag positroid pipe dream with pivot columns in decreasing order, and let $\pi$ be its corresponding decorated permutation.
Then $j$ is an unblocked of column of $D^{\sLe}$ if and only if
$j$ is an unblocked position of $\pi$.
\end{lemma}
\begin{proof}
First, suppose the $j$-th column of $D^{\sLe}$ is blocked and $D^{\sLe}_{i,j}$ is a cross with an elbow to its right. 
Suppose the horizontal and vertical pipes through $D^{\sLe}_{i,j}$ are respectively numbered $h$ and $v$, so that $h < v$ necessarily.
Since the column is blocked at $D_{i,j}^{\sLe}$ then every box below $D_{i,j}^{\sLe}$ is also a cross, therefore $\pi(j) =\underline{v}$ and additionally each horizontal pipe passing below $D_{i,j}^{\sLe}$ is numbered smaller than $v$.
Together with pipe $h$, there are at least $k-i+1$ pipes smaller than $v$ which exit $D^{\sLe}$ to the right of pipe $v$.
But since there are only $k-i$ rows below the $i$-th row, then at least one of these pipes must exit at the bottom of $D^{\sLe}$, say, pipe $h'$ in the $j'$-th column.
Hence $\pi(j') = \underline{h}'$ with $h'< v$, so $j'$ is a blocked position of $\pi$.

Now suppose that $\pi(j)=\underline{v}$ and there is some $j'>j$ such that $\pi(j')=\underline{h}'$ and $h' < v$. 
Since $(j,j')$ is an inversion of $\pi$, the pipes $h'$ and $v$ must cross within $D^{\sLe}$, with $h'$ passing through horizontally.
Since pipe $h'$ exits at the bottom of the $D^{\sLe}$ it must go through an elbow somewhere to the right of this cross, so the column which contains this cross is blocked, and pipe $v$ must exit in the $j$-th column.
Hence the $j$-th column of $D^{\sLe}$ is blocked.
\end{proof}

The following definition uses the convention as set forth by Chen et al. in~\cite[Definition 3.1]{CFGSZ26}, where the cyclic shift is described on the positions of the permutation.
The original definition by Benedetti, Chavez, and Tamayo in~\cite[Definition 22]{BCT22} describes the cyclic shift on the values of the permutation. 

\begin{definition}\label{defn.rcshift}
For $A\subseteq[n]$, let $b_1 < \cdots < b_\ell$ be the elements of $[n]\backslash A$ and define the permutation $\sigma_{[n]\backslash A}= (b_\ell \,\cdots \,b_1)$ in cycle notation.
The \defn{right cyclic shift operator} $\overrightarrow{\rho_A}:\fD_n\rightarrow \fD_n$ is defined as follows.
Given a decorated permutation $\pi$ of $[n]$, let $\overrightarrow{\rho_A}(\pi) = \pi \sigma_{[n]\backslash A}$ be the decorated permutation such that if $j\in [n]\backslash A$ is a fixed point of $\overrightarrow{\rho_A}(\pi)$ then $c(j) = 2$ (otherwise the colouring function is completely determined).

Simply put, $\overrightarrow{\rho_A}(\pi)$ is obtained from $\pi$ by freezing the entry $\pi(j)$ if and only if $j\in A$, cyclically shifting all other entries of $\pi$ one position to the right, and any {\em shifted value} that is a fixed point of $\overrightarrow{\rho_{A}}(\pi)$ is decorated with an overline~\footnote{Consider $\pi=\underline{12}$ for example, and choose $A=\{1\}$.  Then $\overrightarrow{\rho_{A}}(\pi)= \underline{1}\overline{2}$.}.
\end{definition}

\begin{definition}\label{defn.TC}
Let $\pi$ be a decorated permutation on $[n]$ and let $U$ be the set of unblocked positions of $\pi$. 
For nonempty $C\subseteq U$, construct the set $T(C)$ as follows.
\begin{enumerate}
\item[(i)] Set $z=0$, $m=\max(C)$, and $T(C) = \emptyset$.
\item[(ii)] Let $t$ be the smallest $2$-coloured position of $\pi$ such that $z < t < \min(C)$ and $\pi(m)< \pi(t)$. 
If there is no such $t$, then stop.
\item[(iii)] Otherwise, add $t$ to $T(C)$, set $z:= t$, $m:=\pi(t)$ and return to (ii).
\end{enumerate}
In other words, $T(C)=\{t_1,\ldots, t_s\}$ is the maximal subset of $2$-coloured positions of $\pi$ such that $1\leq t_1 < \cdots < t_s < \min(C)$ and $\pi(\max(C)) < \pi(t_1) < \cdots < \pi(t_s)$, greedily chosen.
Note that $C$ is a subset of $1$-coloured positions of $\pi$ and $T(C)$ is a subset of $2$-coloured positions of $\pi$ so $C$ and $T(C)$ are disjoint.
We call $C\sqcup T(C)$ a \defn{right cyclic shift set}.
Define the complement $A(C)= [n]\backslash (C\sqcup T(C))$ to be a \defn{freeze set}.
\end{definition}

\begin{example}\label{eg.blocked_decperm}
Consider the decorated permutation 
\[\setlength\arraycolsep{1.6pt}
\pi=
\begin{matrix}
\textcolor{red}{1}&\textcolor{blue}{2}&\textcolor{blue}{3}&\textcolor{red}{4}&\textcolor{blue}{5}&\textcolor{red}{6}&\textcolor{blue}{7}&\textcolor{blue}{8}&\textcolor{blue}{9}\\
\overline{5}&\underline{1}&\underline{3}&\overline{9}&\underline{2}&\overline{7}&\underline{6}&\underline{4}&\underline{8}
\end{matrix}
\]
from Example~\ref{eg.Ddecperm}, where we write it in two-line notation to emphasize the relationship between unblocked columns of the partial FPP and the unblocked values of $\pi$.
The $1$-coloured values with no smaller $1$-coloured values to their right are $\underline{1}$, $\underline{2}$, $\underline{4}$, and $\underline{8}$ in positions $2,5,8,9$, which are the unblocked columns of the corresponding partial FPP $D^{\sLe}(P)$ in Example~\ref{eg.PFPP}. 

For the decorated permutation $\pi$, we choose some sample nonempty subsets $C \subseteq U$ and compute the decorated permutation resulting from applying the right cyclic shift specified by $C$.
\[{\renewcommand{\arraystretch}{1.25} 
\begin{array}{|c|c|c|c|}
\hline
C & T(C) & A(C) & \overrightarrow{\rho_{A(C)}}(\pi)\\
    \hline
    \{2,5,8,9\} & \emptyset & \{1,3,4,6,7 \} &\overline{5}\overline{8}\underline{3}\overline{9}\underline{1}\overline{7}\underline{6}\underline{2}\underline{4}\\
    \hline 
    \{2,5,8\} & \{1\} & \{3,4,6,7,9 \}& \overline{4}\overline{5}\underline{3}\overline{9}\underline{1}\overline{7}\underline{6}\underline{2}\underline{8}\\
    \hline
    \{5,9\} & \{4\} & \{1,2,3,6,7,8 \} &
\overline{5}\underline{1}\underline{3}\overline{8}\overline{9}\overline{7}\underline{6}\underline{4}\underline{2}\\
	\hline        
	\{8\} & \{1,4\} & \{2,3,5,6,7,9 \} &
\overline{4}\underline{1}\underline{3}\overline{5}\underline{2}\overline{7}\underline{6}\overline{9}\underline{8}\\
	\hline	
\end{array}}
\]
\end{example}

The remainder of this section is devoted to showing that given a positroid $P$ with corresponding decorated permutation $\pi$ whose unblocked positions are $U$, then $Q$ is a positroid such that $P \unlhd_q Q$  is a nonnegatively representable elementary quotient if and only if $Q$ has corresponding decorated permutation $\overrightarrow{\rho_{A(C)}}(\pi)$ for $\emptyset \neq C \subseteq U$.
The approach is summarized in the following diagram:
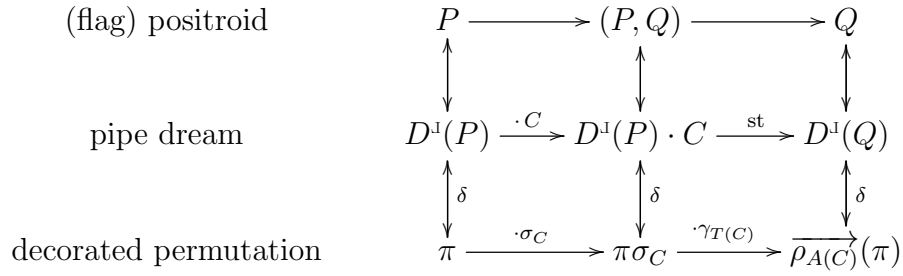
\[
\xymatrix{
\textrm{(flag) positroid}
	&P \ar[r] \ar[d] 
	& (P,Q) \ar[r] \ar[d]
	& Q \ar[d]\\
\textrm{pipe dream}	
	& D^{\sLe}(P) \ar[r]^<<<<{\cdot\,C} \ar[d]^{\delta} \ar[u]
	& D^{\sLe}(P) \cdot C \ar[r]^>>>>>{\st} \ar[d]^{\delta} \ar[u]
	& D^{\sLe}(Q)\ar[d]^{\delta} \ar[u]\\
\textrm{decorated permutation}
	& \pi \ar[r]^<<<<<<<<{\cdot \sigma_C} \ar[u]
	& \pi \sigma_C \ar[r]^<<<<<<{\cdot \gamma_{T(C)}} \ar[u]
	& \overrightarrow{\rho_{A(C)}}(\pi) \ar[u]
}
\]
The top of the diagram was established in Sections~\ref{subsec:le-diagrams-for-quotients} and~\ref{subsec.stdizing}.
Here, $C$ is a nonempty subset of the unblocked columns of $D^{\sLe}(P)$, and $\cdot C$ denotes appending the row below $D^{\sLe}(P)$ such that the pivot elbow and elbows in the row appear only in the columns indexed by $C$.
We now explain the bottom row of the diagram.

We begin by extending the notion of trivial completion to all partial FPPs, so that we can define a permutation for partial FPPs that is similar to a decorated permutation for partial FPPs with pivot columns in decreasing order.
\begin{definition}
Given a partial flag positroid pipe dream $D\in \setFPP_k(n)$, define its \defn{trivial completion} to be the flag positroid pipe dream $\FPP(u,v)$ whose first $k$ rows is $D$, and whose pivot elbows in the last $n-k$ rows are in decreasing order.
Define the permutation $\defn{\delta(D)} = vu^{-1}$.
Note that there is no colouring function or decoration associated with $\delta(D)$.
\end{definition}
In the special case when $D^{\sLe}\in\setFPP_k^{\sLe}(n)$ has pivot columns in decreasing order, it was shown in Section~\ref{subsec.decperms} that the underlying permutation of its decorated permutation is precisely $\delta(D^{\sLe})$, and the colouring function is completely determined by $u$. 
So in a sense, $\delta(D)$ strongly mimics the idea of a decorated permutation for $D^{\sLe}$, even though $\delta(D)$ does not explicitly carry a colouring.

Note that $\delta(D)$ can be obtained from the partial FPP $D$ by following the pipes starting from the southeast boundary of $D$ to the top in exactly the same way that the decorated permutation is read off of $D^{\sLe}$.

\begin{proposition}\label{prop.sigma_C}
Let $P$ be a positroid with corresponding decorated permutation $\pi$, whose set of unblocked positions is $U$.
Given $\emptyset \neq C =\{c_1,\ldots, c_r\} \subseteq U$ such that $c_1 < \cdots < c_r$, define $\sigma_C = (c_r\,\cdots\, c_1)$.
Then $\delta(D^{\sLe}(P) \cdot C) = \pi \sigma_C$.
\end{proposition}
\begin{proof}
By Lemma~\ref{lem:unblocked-in-decperm}, the unblocked positions of $\pi$ correspond with the unblocked columns of $D^{\sLe}(P)$, so by Remark~\ref{rem.noncrossing_pipes_U}, the collection of pipes in $D^{\sLe}(P)$ that exit in the columns indexed by $C\subseteq U$ are non-intersecting.

Let $\pi'=\delta(D^{\sLe}(P) \cdot C)$.
Since $D^{\sLe}(P) \cdot C$ is obtained from $D^{\sLe}(P)$ by appending a row below it with elbow tiles in the columns indexed by $C$ only, then the non-intersecting property of the pipes exiting in the columns $C$ imply that 
\[\pi'(j) =
\begin{cases}
\pi(j), & \hbox{if } j\notin C,\\
\pi(c_{i-1}), & \hbox{if } j=c_i \hbox{ for some } i=2,\ldots, r\\
\pi(c_r), & \hbox{if } j=c_1.
\end{cases}
\]
Therefore, $\pi' = \pi \sigma_C$.
\end{proof}

\begin{proposition}\label{prop.gamma_T}
Let $D\in \setFPP_k(n)$ be a partial flag positroid pipe dream with pivot columns $u_1,\ldots, u_k$, and let $\varpi=\delta(D)$.
For $1\leq i \leq k$, 
\[\delta(\st_i(D))= 
\begin{cases}
\delta(D),
	 &\hbox{if } \varpi(u_i) < \varpi(u_{i+1}),\\
\delta(D)\cdot(u_i\,u_{i+1}),
	 &\hbox{if } \varpi(u_i) > \varpi(u_{i+1}).
\end{cases}
\]
\end{proposition}

\begin{proof}
Suppose the trivial completions of $D$ and $\st_i(D)$ are $\FPP(u,v)$ and $\FPP(u',v')$ respectively, so that $\varpi =\delta(D) = vu^{-1}$, and $\varpi'= \delta(\st_i(D)) = v'(u')^{-1}$.

It follows from Proposition~\ref{prop.pipe_exchange} that if there exists $j^*$ such that $D_{i,j^*}$ is a cross and $D_{i+1,j^*}$ is an elbow or pivot elbow, then
\[\varpi'=v'(u')^{-1} = vs_i(us_i)^{-1} = vu^{-1}=\varpi,\]
and otherwise,
\[\varpi'=v'(u')^{-1} = v(us_i)^{-1} = vs_i u^{-1} = \varpi\cdot (u_i\,u_{i+1}).
\]

Note that $\varpi u = v$ if and only if $v_i = \varpi(u_i)$ for all $i$.
So it remains to show that the situation $j^*$ exists if and only if $v_i < v_{i+1}$.

Suppose $D_{i,j^*}$ is a cross and $D_{i+1,j^*}$ is an elbow or pivot elbow.
Since $D$ is $\Gamma$-free, then there are no elbows to the right of $D_{i,j^*}$, so the horizontal pipe through $D_{i,j^*}$ is pipe $v_i$.  
Because $D_{i+1,j^*}$ is an elbow, then the pipe $v_{i+1}$ crosses pipe $v_i$ in $D_{i,h}$ for some $h>j^*$.  
Since pipes cross at most once, then $v_i < v_{i+1}$.

Now suppose $v_i > v_{i+1}$ and suppose for contradiction that there exists $j^*$ such that $D_{i,j^*}$ is a cross and $D_{i+1,j^*}$ is an elbow or pivot elbow.
Let $D_{i,l}$ be the leftmost tile in which pipe $v_i$ occupies the $i$-th row of $D$.
Then $j^* \ngeq l$, for we would have $v_i < v_{i+1}$ as discussed before.
But clearly $j^* \not<l$ because there are no horizontal pipes to the left of $D_{i,l}$.
Therefore, we conclude that such a $j^*$ cannot exist.
\end{proof}
\begin{remark}
Visually, there exists $j^*$ such that $D_{i,j^*}$ is a cross and $D_{i+1,j^*}$ is an elbow or pivot elbow if and only if the pipes $v_i$ and $v_{i+1}$ cross in $D$.
\end{remark}

The main result of this section is the following characterization of nonnegatively representable elementary positroid quotients via cyclic shifts of decorated permutations, with an explicit description of the frozen set (the complement of the cyclic shift set).

\begin{theorem}\label{thm.covering_decperms}
Let $P$ be a rank $k$ positroid with corresponding decorated permutation $\pi$, whose set of unblocked positions is $U$. 
The rank $k+1$ positroids $Q$ such that $P\unlhd_q Q$ are those whose decorated permutations are $\overrightarrow{\rho_{A(C)}}(\pi)$ for nonempty $C\subseteq U$.
\end{theorem}
\begin{proof}
By Theorems~\ref{thm.char_Q} and~\ref{thm.sti_preserves_bases} there is a correspondence between rank $k+1$ positroids $Q$ such that $P \unlhd_q Q$, and nonempty subsets $C\subseteq U$, where $D^{\sLe}(Q) = \st( D^{\sLe}(P) \cdot C)$.
Therefore, it remains to show that the decorated permutation of $Q$ is given by
\[\delta(D^{\sLe}(Q))=\delta(\st( D^{\sLe}(P) \cdot C))) = \overrightarrow{\rho_{A(C)}} (\pi).
\]

Suppose $C=\{c_1,\ldots, c_r\} \subseteq U$ and let $T(C)=\{t_1,\ldots, t_s\}$ be as defined in Definition~\ref{defn.TC} so that 
\[t_1< \cdots < t_s < c_1 < \cdots < c_r
\quad\hbox{and}\quad
\pi(c_1) < \cdots < \pi(c_r) < \pi(t_1) < \cdots < \pi(t_s).\]
Denote $D' = D^{\sLe}(P) \cdot C$ and $\pi' = \delta(D')$. 
By Proposition~\ref{prop.sigma_C},
\[\pi' = \pi \sigma_C = \pi \cdot (c_r \,\cdots\, c_1),\]
so that in particular, $\pi'(c_1) =\pi(c_r) = \pi(\max(C))$, and $\pi'(j) = \pi(j)$ if $j\notin C$.

Suppose $D^{\sLe}(P)$ has pivot columns $u_1>\cdots > u_k$ so that $D'$ has the same pivot columns together with $u_{k+1}=c_1$.
We note that $\pi'(u_i)$ is a $2$-coloured value for $i=1,\ldots, k$.
To standardize $D'$, we apply the operations $\st_a \circ \cdots \circ \st_{k-1} \circ \st_k$ where $a$ is the index such that
\[u_1>\cdots > u_{a-1} > c_1 > u_a > \cdots > u_k. \]
Beginning with $\st_k$, Proposition~\ref{prop.gamma_T} states that
\[
\delta(\st_k(D')) = 
\begin{cases}
\pi', 
	&\hbox{if } \pi'(u_k) < \pi'(c_1),\\
\pi' \cdot (u_k \ c_1),
	&\hbox{if } \pi'(u_k) > \pi'(c_1).\\
\end{cases}
\]
Since $u_1,\ldots, u_k\notin C$, then $\pi'(u_k)= \pi(u_k)$, and so $\delta(\st_k(D'))= \pi'\cdot (u_k\, c_1)$ if and only if  $\pi(u_k) > \pi(\max(C))$.
That is to say if and only if $u_k = t_1 \in T(C)$.
Otherwise, $\delta(\st_k(D'))=\pi'$.

Successively applying the standardization operators until $D'$ is standardized, we conclude that 
\[\delta(\st(D')) = \pi'(t_1\, c_1)(t_2\, c_1) \cdots (t_s \, c_1),\]
where $t_s < u_a < c_1=\min(C)$.

Defining $\gamma_{T(C)}=(t_1\, c_1)(t_2\, c_1) \cdots (t_s \, c_1) = (c_1\, t_s\, \cdots\, t_1)$, observe that 
\[\sigma_C \gamma_{T(C)} 
= (c_r \,\cdots\,c_1)(c_1\, t_s\, \cdots\, t_1) 
= (c_r \,\cdots\,c_1\, t_s\, \cdots\, t_1)
=\sigma_{C\sqcup T(C)},
\]
and therefore  
\[\delta(D^{\sLe}(Q))
=\delta(\st(D'))
=\pi' \gamma_{T(C)}
=\pi \sigma_C \gamma_{T(C)}
=\pi \sigma_{C\sqcup T(C)}
=\overrightarrow{\rho_{A(C)}}(\pi)
\]
as permutations.

Lastly, we verify that the colouring functions $c^\delta$ on $\delta(D^{\sLe}(Q))$ and $c^\rho$ on $\overrightarrow{\rho_{A(C)}}(\pi)$ are equal.
Since the underlying permutations coincide, we only need to check this on the fixed points of the permutation, which for convenience we notate as $\widehat\pi =\overrightarrow{\rho_{A(C)}}(\pi)=\delta(D^{\sLe}(Q))$.
Let $c$ denote the colouring function on $\pi$.

First suppose $j$ is a fixed point of $\widehat\pi$ and $j\notin C\sqcup T(C)$.
In terms of $\overrightarrow{\rho_{A(C)}}(\pi)$, this means $j$ was not cyclically shifted and is a fixed point of $\pi$, so $c^\rho(j)= c(j)$ trivially. 
In terms of $\delta(D^{\sLe}(Q))$, if $c^\delta(j)=1$, then pipe $j$ was simply extended by a cross tile in the $(k+1)$-st row of $D^{\sLe}(Q)$, so $c(j) = c^\delta(j) = 1$. 
Otherwise $c^\delta(j)=2$.
If $j>\min(C)$, then no standardization operation on $D^{\sLe}(P)\cdot C$ interacted with the pipe $\pi(j)$, so this pipe exits at the right boundary of $D^{\sLe}(P)$ and $c(j)=c^\delta(j)=2$.
If $j<\min(C)$, then $j$ is not in the right cyclic shift set of $\pi$ means that the pipe $\pi'(\min(C))$ intersects pipe $\pi'(j)$ (where $\pi'=\pi \sigma_C$) in $D^{\sLe}(P)\cdot C$.
After the standardization operation is applied to the row containing a pivot elbow in the $j$-th column, a cross tile will be replaced with an elbow tile in that row, so in fact such a $j$ cannot be a fixed point of $\delta(D^{\sLe}(Q))$.
Altogether, we have $c^\delta(j) = c(j) = c^\rho(j)$ in this case.

Now suppose $j$ is a fixed point of $\widehat\pi$ and $j\in C\sqcup T(C)$.
We first observe that $j\neq c_i$ for $i=2,\ldots, r$, for if it was, then by the right cyclic shift operator,
\[c_i= \widehat\pi(c_i)=\pi(c_{i-1}) < \pi(c_i) \leq c_i,\]
which is impossible.
This leaves the possibility of $j=\min(C)$ or $j\in T(C)$. 
In this case, the pipe $\widehat\pi(j)$ exits at the right boundary of $D^{\sLe}(Q)$ so $c^\delta(j)=2$.
On the other hand, since $j$ is in the right cyclic shift set of $\pi$ and is a fixed point of $\overrightarrow{\rho_{A(C)}}(\pi)$, then $c^\pi(j)=2$ by the definition of the right cyclic shift operator.
Therefore, $c^\delta=c^\rho$.
\end{proof}

\begin{corollary}
Given $\pi\in \fD_n$, let $U$ be the set of unblocked positions of $\pi$ and choose $\emptyset \neq C \subseteq U$.
If the $2$-coloured positions of $\pi$ are $\{i_1,\ldots, i_k\}$, then the $2$-coloured positions of $\overrightarrow{\rho_{A(C)}}(\pi)$ are $\{\min(C), i_1,\ldots, i_k\}$.
\qed
\end{corollary}

\begin{example}\label{eg.grand_finale}
We finish our running example for the partial FPP $D^{\sLe}(P)$ from Example~\ref{eg.PFPP}.
We computed in Example~\ref{eg.Ddecperm} that the corresponding decorated permutation is $\pi =\overline{5}\underline{13}\overline{9}\underline{2}\overline{7}\underline{648}$.
Recall that the set of unblocked columns of $D^{\sLe}(P)$/ unblocked positions of $\pi$ is $U=\{2,5,8,9\}$.

Selecting $C=\{5,9\}\subseteq U$, we computed in Example~\ref{eg.blocked_decperm} that $T(C) =\{4\}$, so $\sigma_{C \sqcup T(C)} = (9 \, 5 \, 4)$ and hence the decorated permutation of the partial FPP $D^{\sLe}(Q)$ is
\[\overrightarrow{\rho_{A(C)}}(\pi) 
= \pi\sigma_{C \sqcup T(C)}
= \overline{5}\underline{13}\overline{897}\underline{642}.
\]
See Figure~\ref{fig.grand_finale}.
\end{example}

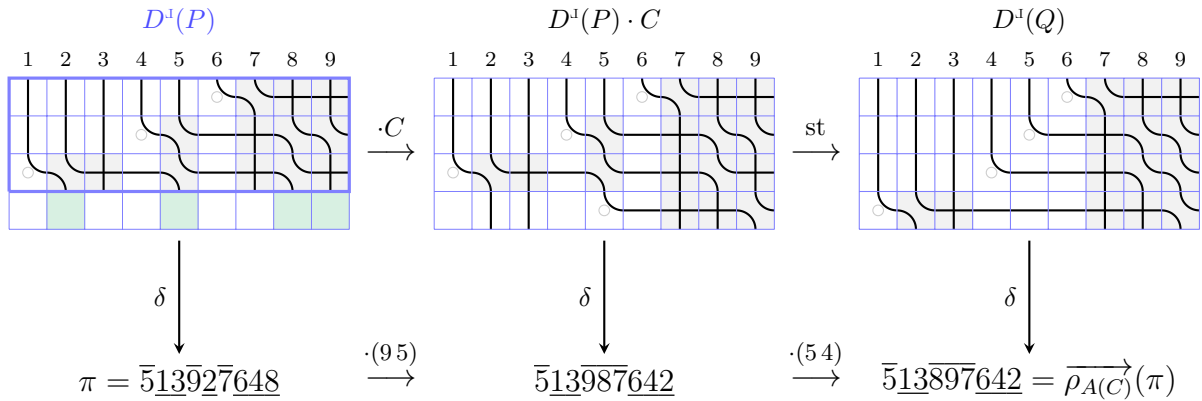
\begin{figure}[ht!]
\begin{center}
\begin{tikzpicture}
\begin{scope}[scale=.5, xshift=0]
\vertex[color=gray!40]() at (5.5, 2.5) {};
\vertex[color=gray!40]() at (3.5, 1.5) {};
\vertex[color=gray!40]() at (0.5, 0.5) {};

\draw[fill, color=mediumseagreen!20] (1,-1) rectangle (2,0);
\draw[fill, color=mediumseagreen!20] (4,-1) rectangle (5,0);
\draw[fill, color=mediumseagreen!20] (7,-1) rectangle (9,0);

\draw [black,thick,domain=180:270] plot ({6+.5*cos(\x)}, {3+.5*sin(\x)});
\draw [black,thick,domain=180:270] plot ({4+.5*cos(\x)}, {2+.5*sin(\x)});
\draw [black,thick,domain=180:270] plot ({1+.5*cos(\x)}, {1+.5*sin(\x)});
\draw [black,thick] (0.5,1) -- (0.5,3);
\draw [black,thick] (1.5,1) -- (1.5,3);
\draw [black,thick] (2.5,1) -- (2.5,3);
\draw [black,thick] (3.5,2) -- (3.5,3);
\draw [black,thick] (4.5,2) -- (4.5,3);
\draw [black,thick] (3,0.5) -- (4,0.5);
\draw [black,thick] (5,0.5) -- (6,0.5);
\draw [black,thick] (5,1.5) -- (6,1.5);
\cross[thick](3,1);
\cross[thick](7,1);
\cross[thick](9,1);
\cross[thick](7,2);
\cross[thick](8,3);
\cross[thick](9,3);
\elbow[thick](2,1); 
\elbow[thick](5,1); 
\elbow[thick](8,1);
\elbow[thick](5,2);
\elbow[thick](8,2);
\elbow[thick](9,2);
\elbow[thick](7,3);

\draw[very thin, color=blue!50] (0,-1) grid (9,3);
\draw[very thick, color=blue!50] (0,0)--(9,0)--(9,3)--(0,3)--(0,0);
\node[] at (4.5, 4.5) {\footnotesize\textcolor{blue!70}{$D^{\sLe}(P)$}};
	\node at (0.5,3.5) {\color{black}{\tiny$1$}};
	\node at (1.5,3.5) {\color{black}{\tiny$2$}};
	\node at (2.5,3.5) {\color{black}{\tiny$3$}};
	\node at (3.5,3.5) {\color{black}{\tiny$4$}};
	\node at (4.5,3.5) {\color{black}{\tiny$5$}};
	\node at (5.5,3.5) {\color{black}{\tiny$6$}};
	\node at (6.5,3.5) {\color{black}{\tiny$7$}};
	\node at (7.5,3.5) {\color{black}{\tiny$8$}};
	\node at (8.5,3.5) {\color{black}{\tiny$9$}};
	
    \node[] at (10.1,1.7){\footnotesize$\cdot C$};
    \node[] at (10.1,1){$\longrightarrow$};

\node at (4.5,-5){$\pi =
\overline{5}
\underline{1}
\underline{3}
\overline{9}
\underline{2}
\overline{7}
\underline{6}
\underline{4}
\underline{8}
$};

\node[] at (10.1,-4.3){\scriptsize$\cdot (9\,5)$};
\node[] at (10.1,-5){$\longrightarrow$};
\draw[-stealth, thick] (4.5,-1.3)--(4.5,-4.3);
\node[label={\footnotesize$\delta$}] at (4,-3.6){};
\end{scope}

\begin{scope}[scale=.5, xshift=320]
\vertex[color=gray!40]() at (5.5, 2.5) {};
\vertex[color=gray!40]() at (3.5, 1.5) {};
\vertex[color=gray!40]() at (0.5, 0.5) {};
\vertex[color=gray!40]() at (4.5, -.5) {};

\draw [black,thick,domain=180:270] plot ({6+.5*cos(\x)}, {3+.5*sin(\x)});
\draw [black,thick,domain=180:270] plot ({4+.5*cos(\x)}, {2+.5*sin(\x)});
\draw [black,thick,domain=180:270] plot ({1+.5*cos(\x)}, {1+.5*sin(\x)});
\draw [black,thick,domain=180:270] plot ({5+.5*cos(\x)}, {0+.5*sin(\x)});
\draw [black,thick] (0.5,1) -- (0.5,3);
\draw [black,thick] (1.5,-1) -- (1.5,0);
\draw [black,thick] (1.5,1) -- (1.5,3);
\draw [black,thick] (2.5,-1) -- (2.5,0);
\draw [black,thick] (2.5,1) -- (2.5,3);
\draw [black,thick] (3.5,2) -- (3.5,3);
\draw [black,thick] (4.5,2) -- (4.5,3);
\draw [black,thick] (3,0.5) -- (4,0.5);
\draw [black,thick] (5,-.5) -- (6,-.5);
\draw [black,thick] (5,0.5) -- (6,0.5);
\draw [black,thick] (5,1.5) -- (6,1.5);

\cross[thick](3,1);
\cross[thick](7,1);
\cross[thick](9,1);
\cross[thick](7,0); 
\cross[thick](7,2);
\cross[thick](8,0); 
\cross[thick](8,3);
\cross[thick](9,3);
\elbow[thick](2,1); 
\elbow[thick](9,0); 
\elbow[thick](5,1); 
\elbow[thick](8,1);
\elbow[thick](5,2);
\elbow[thick](8,2);
\elbow[thick](9,2);
\elbow[thick](7,3);

\draw[very thin, color=blue!50] (0,-1) grid (9,3);
\node at (4.5, 4.5) {\footnotesize$D^{\sLe}(P)\cdot C$};

	\node at (0.5,3.5) {\color{black}{\tiny$1$}};
	\node at (1.5,3.5) {\color{black}{\tiny$2$}};
	\node at (2.5,3.5) {\color{black}{\tiny$3$}};
	\node at (3.5,3.5) {\color{black}{\tiny$4$}};
	\node at (4.5,3.5) {\color{black}{\tiny$5$}};
	\node at (5.5,3.5) {\color{black}{\tiny$6$}};
	\node at (6.5,3.5) {\color{black}{\tiny$7$}};
	\node at (7.5,3.5) {\color{black}{\tiny$8$}};
	\node at (8.5,3.5) {\color{black}{\tiny$9$}};
	
    \node[] at (10.1,1.7){\footnotesize$\st$};
    \node[] at (10.1,1){$\longrightarrow$};

\node at (4.5,-5){$
\overline{5}
\underline{1}
\underline{3}
\overline{9}
\overline{8}
\overline{7}
\underline{6}
\underline{4}
\underline{2}
$};

\node[] at (10.1,-4.3){\scriptsize$\cdot (5\,4)$};
\node[] at (10.1,-5){$\longrightarrow$};
\draw[-stealth, thick] (4.5,-1.3)--(4.5,-4.3);
\node[label={\footnotesize$\delta$}] at (4,-3.6){};
\end{scope}

\begin{scope}[scale=.5, xshift=640]
\vertex[color=gray!40]() at (5.5, 2.5) {};
\vertex[color=gray!40]() at (4.5, 1.5) {};
\vertex[color=gray!40]() at (3.5, 0.5) {};
\vertex[color=gray!40]() at (0.5, -.5) {};

\draw [black,thick,domain=180:270] plot ({6+.5*cos(\x)}, {3+.5*sin(\x)});
\draw [black,thick,domain=180:270] plot ({5+.5*cos(\x)}, {2+.5*sin(\x)});
\draw [black,thick,domain=180:270] plot ({4+.5*cos(\x)}, {1+.5*sin(\x)});
\draw [black,thick,domain=180:270] plot ({1+.5*cos(\x)}, {0+.5*sin(\x)});
\draw [black,thick] (0.5,0) -- (0.5,3);
\draw [black,thick] (1.5,0) -- (1.5,3);
\draw [black,thick] (2.5,0) -- (2.5,3);
\draw [black,thick] (3.5,1) -- (3.5,3);
\draw [black,thick] (4.5,2) -- (4.5,3);
\draw [black,thick] (3,-.5) -- (6,-.5);
\draw [black,thick] (4,0.5) -- (6,0.5);
\draw [black,thick] (5,1.5) -- (6,1.5);

\cross[thick](3,0);
\cross[thick](7,0);
\cross[thick](7,1); 
\cross[thick](7,2);
\cross[thick](8,0); 
\cross[thick](8,3);
\cross[thick](9,3);
\elbow[thick](2,0); 
\elbow[thick](9,0);
\elbow[thick](9,1); 
\elbow[thick](8,1);
\elbow[thick](8,2);
\elbow[thick](9,2);
\elbow[thick](7,3);

\draw[thin, color=blue!50] (0,-1) grid (9,3);
\node at (4.5, 4.5) {\footnotesize$D^{\sLe}(Q)$};

	\node at (0.5,3.5) {\color{black}{\tiny$1$}};
	\node at (1.5,3.5) {\color{black}{\tiny$2$}};
	\node at (2.5,3.5) {\color{black}{\tiny$3$}};
	\node at (3.5,3.5) {\color{black}{\tiny$4$}};
	\node at (4.5,3.5) {\color{black}{\tiny$5$}};
	\node at (5.5,3.5) {\color{black}{\tiny$6$}};
	\node at (6.5,3.5) {\color{black}{\tiny$7$}};
	\node at (7.5,3.5) {\color{black}{\tiny$8$}};
	\node at (8.5,3.5) {\color{black}{\tiny$9$}};

\node at (4.5,-5){$ 
\overline{5}
\underline{1}
\underline{3}
\overline{8}
\overline{9}
\overline{7}
\underline{6}
\underline{4}
\underline{2}
=\overrightarrow{\rho_{A(C)}}(\pi)$};

\draw[-stealth, thick] (4.5,-1.3)--(4.5,-4.3);
\node[label={\footnotesize$\delta$}] at (4,-3.6){};
\end{scope}
\end{tikzpicture}
\end{center}
\caption{Two combinatorial characterizations of nonnegatively representable elementary positroid quotients, from partial flag positroid pipe dreams to decorated permutations.
See Example~\ref{eg.grand_finale}.}
\label{fig.grand_finale}
\end{figure}

Given a rank $k$ positroid $P$ with corresponding decorated permutation $\pi$, we can now compute the decorated permutation of any rank $k+1$ positroid $Q$ such that $P \unlhd_q Q$ via right cyclic shifts.
For completeness, we can also describe the decorated permutation of any rank $k-1$ positroid $N$ such that $P \unlhd_q N$, via left cyclic shifts.

\begin{definition}\label{defn.lcshift}
For $A \subseteq [n]$, let $b_1< \cdots < b_\ell$ be the elements of $[n]\backslash A$ and define the permutation $\tau_{[n]\backslash A} = (b_1 \,\cdots\, b_\ell)$ in cycle notation.
The \defn{left cyclic shift operator} $\overleftarrow{\rho_A}: \fD_n \rightarrow \fD_n$ is defined as follows.
Given a decorated permutation $\pi$ on $[n]$, let $\overleftarrow{\rho_A}(\pi) = \pi \tau_{[n]
\backslash A}$ be the decorated permutation such that if $j\in [n]\backslash A$ is a fixed point of $\overleftarrow{\rho_A}(\pi)$ then $c(j) =1$.
\end{definition}
This is completely analogous to the right cyclic shift operator; values of $\pi$ in $A$ are frozen, all other values of $\pi$ are cyclically shifted one position to the left, and any shifted element that is a fixed point of $\overleftarrow{\rho_A}(\pi)$ is decorated with an underline.

\begin{definition}
Let $\pi$ be a decorated permutation on $[n]$.  
Then $\pi(j)$ is a \defn{left unblocked value} of $\pi$ if $\pi(j)$ is $2$-coloured and for all $j'<j$ such that $\pi(j')$ is $2$-coloured we have $\pi(j') < \pi(j)$.
We also say that $j$ is a \defn{left unblocked position} of $\pi$. 

Letting $S$ be the set of left unblocked positions of $\pi$, then for nonempty $R \subseteq S$, let $O(R)=\{i_1,\ldots, i_m\}$ be the maximal subset of $1$-coloured positions of $\pi$ such that $n\geq i_1 > \cdots > i_m > \max(R)$ and $\pi(\min R) > \pi(i_1) > \cdots > \pi(i_m)$, greedily chosen.
Define $A(R) = [n]\backslash (R \sqcup O(R))$.
\end{definition}

\begin{theorem}
Let $P$ be a rank $k$ positroid with corresponding decorated permutation $\pi$, whose set of left unblocked positions is $S$.
The rank $k-1$ positroids $N$ such that $P \unrhd_q N$ are those whose decorated permutations are $\overleftarrow{\rho_{A(R)}} (\pi)$ for nonempty $R\subseteq S$.
\qed
\end{theorem}

\subsection{The poset of nonnegatively representable positroid quotients}
We summarize some properties of the poset $(\Pi_n, \unlhd_q)$ of nonnegatively representable positroid quotients that we have developed in this article.

\begin{definition}
The elements of the poset $(\Pi_n, \unlhd_q)$ are the positroids on $[n]$.  
Given two positroids $P$ and $Q$, we have the covering relation $P \lessdot_q Q$ if and only if $P \unlhd_q Q$ is a nonnegatively representable elementary quotient. 
\end{definition} 
By definition, $(\Pi_n, \unlhd_q)$ is graded by rank, and is a subposet of the poset $(\Pi_n, \leq_q)$ of elementary positroid quotients.  
See Figure~\ref{fig.3poset}.

Since positroids on $[n]$ are in bijection with partial FPPs $\setFPP^{\sLe}(n)$ and with decorated permutations $\fD_n$, then the poset $(\Pi_n, \unlhd_q)$ can be described in terms of these combinatorial objects.
Theorem~\ref{thm.char_Q} gives a combinatorial characterization of the covering relations in the poset in terms of the nonempty subsets of unblocked columns of the partial FPP, while Theorem~\ref{thm.covering_decperms} gives a characterization in terms of the right cyclic shift sets $C\sqcup T(C)$ of the decorated permutation.
Corollary~\ref{cor.char_Q} states that every element in the poset is covered by $2^{|U|}-1$ elements, where $|U|$ is the number of unblocked columns of $D^{\sLe}$/unblocked positions of $\pi$.

Theorem~\ref{thm.fpp_richardson} shows that maximal chains in $(\Pi_n, \unlhd_q)$ are in bijection with the set $\setFPP(n)$ of flag positroid pipe dreams.
In other words, maximal chains in this poset are in bijection with positive Richardson cells in the decomposition of the nonnegative flag variety $\Fl_n^{\geq0}$.

A rank $k$ \defn{lattice path matroid} (LPM) is a positroid whose bases are of the form $\{B \mid I \leq_G B \leq_G J\}$, for some $k$-subsets $I, J$, where $\leq_G$ denotes the Gale order on $\binom{[n]}{k}$.
Benedetti and Knauer~\cite[Theorem 19]{BK24} gave a complete combinatorial characterization of LPM quotients in terms of ``good pairs'', and studied the poset of LPM quotients.
Since quotients of LPMs are nonnegatively representable~\cite[Corollary 37]{BK24}, then the poset $(\Pi_n, \unlhd_q)$ contains the poset of quotients of LPMs~\cite[Section 3.1]{BK24} as an induced subposet.

LPMs are easily identifiable visually via their partial FPPs $D^{\sLe}$/$\Le$-pipe dreams; a $\Le$-pipe dream corresponds to an LPM if and only if its elbow tiles form a skew partition shape where the lower boundary of the skew partition is the lower boundary of the $\Le$-pipe dream.
For example in Figure~\ref{fig.3NNposet}, the two positroids on $[3]$ that are not LPMs are the ones corresponding to the decorated permutations $\overline{3}\underline{21}$ and $\overline{32}\underline{1}$.

\begin{figure}[ht!]
\begin{tikzpicture}[xscale=1.4, yscale=1.4]

\begin{scope}[xshift=0, yshift=130, scale=0.3, local bounding box=p1bar2bar3bar]
\node[anchor=south]() at (1.5,-0.3) {};
\node[] at (1.5, -0.4){\tiny$\overline{123}$};
\draw [black,thick,domain=180:270] plot ({1+.5*cos(\x)}, {1+.5*sin(\x)});
\draw [black,thick,domain=180:270] plot ({2+.5*cos(\x)}, {2+.5*sin(\x)});
\draw [black,thick,domain=180:270] plot ({3+.5*cos(\x)}, {3+.5*sin(\x)});
\draw [black,thick] (.5,1)--(.5,3){};
\draw [black,thick] (1.5,2)--(1.5,3){};
\draw [black,thick] (2,1.5)--(3,1.5){};
\draw [black,thick] (1,.5)--(3,.5){};
\draw[very thin, color=blue!50] (0,0) grid (3,3);
\end{scope}

\begin{scope}[xshift=120, yshift=80, scale=0.3, local bounding box=p231]
\node[anchor=south]() at (1.5,-0.3) {};
\node[anchor=north]() at (1.5,2.3) {};
\node[] at (1.5, 2.4){\tiny$\overline{23}\underline{1}$};
\draw [black,thick,domain=180:270] plot ({2+.5*cos(\x)}, {2+.5*sin(\x)});
\draw [black,thick,domain=180:270] plot ({1+.5*cos(\x)}, {1+.5*sin(\x)});
\elbow[thick](3,1);
\elbow[thick](3,2);
\draw [black,thick] (1,.5)--(2,.5){};
\draw [black,thick] (.5,1)--(.5,2){};
\draw[very thin, color=blue!50] (0,0) grid (3,2);
\end{scope}

\begin{scope}[xshift=80, yshift=80, scale=0.3, local bounding box=p1bar32]
\node[anchor=south]() at (1.5,-0.3) {};
\node[anchor=north]() at (1.5,2.3) {};
\node[] at (1.5, 2.4){\tiny$\overline{1}\overline{3}\underline{2}$};
\draw [black,thick,domain=180:270] plot ({2+.5*cos(\x)}, {2+.5*sin(\x)});
\draw [black,thick,domain=180:270] plot ({1+.5*cos(\x)}, {1+.5*sin(\x)});
\elbow[thick](3,2);
\cross[thick](3,1);
\draw [black,thick] (1,.5)--(2,.5){};
\draw [black,thick] (.5,1)--(.5,2){};
\draw[very thin, color=blue!50] (0,0) grid (3,2);
\end{scope}

\begin{scope}[xshift=40, yshift=80, scale=0.3, local bounding box=p32bar1]
\node[anchor=south]() at (1.5,-0.3) {};
\node[anchor=north]() at (1.5,2.3) {};
\node[] at (1.5, 2.4){\tiny$\overline{3}\overline{2}\underline{1}$};
\draw [black,thick,domain=180:270] plot ({2+.5*cos(\x)}, {2+.5*sin(\x)});
\draw [black,thick,domain=180:270] plot ({1+.5*cos(\x)}, {1+.5*sin(\x)});
\cross[thick](3,2);
\elbow[thick](3,1);
\draw [black,thick] (1,.5)--(2,.5){};
\draw [black,thick] (.5,1)--(.5,2){};
\draw[very thin, color=blue!50] (0,0) grid (3,2);
\end{scope}

\begin{scope}[xshift=0, yshift=80, scale=0.3, local bounding box=p1bar2bar3]
\node[anchor=south]() at (1.5,-0.3) {};
\node[anchor=north]() at (1.5,2.3) {};
\node[] at (1.5, 2.4){\tiny$\overline{1}\overline{2}\underline{3}$};
\draw [black,thick,domain=180:270] plot ({2+.5*cos(\x)}, {2+.5*sin(\x)});
\draw [black,thick,domain=180:270] plot ({1+.5*cos(\x)}, {1+.5*sin(\x)});
\cross[thick](3,2);
\cross[thick](3,1);
\draw [black,thick] (1,.5)--(2,.5){};
\draw [black,thick] (.5,1)--(.5,2){};
\draw[very thin, color=blue!50] (0,0) grid (3,2);
\end{scope}

\begin{scope}[xshift=-40, yshift=80, scale=0.3, local bounding box=p213bar]
\node[anchor=south]() at (1.5,-0.3) {};
\node[anchor=north]() at (1.5,2.3) {};
\node[] at (1.5, 2.4){\tiny$\overline{2}\underline{1}\overline{3}$};
\draw [black,thick,domain=180:270] plot ({3+.5*cos(\x)}, {2+.5*sin(\x)});
\draw [black,thick,domain=180:270] plot ({1+.5*cos(\x)}, {1+.5*sin(\x)});
\elbow[thick](2,1);
\draw [black,thick] (2,.5)--(3,.5){};
\draw [black,thick] (.5,1)--(.5,2){};
\draw [black,thick] (1.5,1)--(1.5,2){};
\draw[very thin, color=blue!50] (0,0) grid (3,2);
\end{scope}

\begin{scope}[xshift=-80, yshift=80, scale=0.3, local bounding box=p1bar23bar]
\node[anchor=south]() at (1.5,-0.3) {};
\node[anchor=north]() at (1.5,2.3) {};
\node[] at (1.5, 2.4){\tiny$\overline{1}\underline{2}\overline{3}$};
\draw [black,thick,domain=180:270] plot ({3+.5*cos(\x)}, {2+.5*sin(\x)});
\draw [black,thick,domain=180:270] plot ({1+.5*cos(\x)}, {1+.5*sin(\x)});
\cross[thick](2,1);
\draw [black,thick] (2,.5)--(3,.5){};
\draw [black,thick] (.5,1)--(.5,2){};
\draw [black,thick] (1.5,1)--(1.5,2){};
\draw[very thin, color=blue!50] (0,0) grid (3,2);
\end{scope}

\begin{scope}[xshift=-120, yshift=80, scale=0.3, local bounding box=p12bar3bar]
\node[anchor=south]() at (1.5,-0.3) {};
\node[anchor=north]() at (1.5,2.3) {};
\node[] at (1.5, 2.4){\tiny$\underline{1}\overline{2}\overline{3}$};
\draw [black,thick,domain=180:270] plot ({3+.5*cos(\x)}, {2+.5*sin(\x)});
\draw [black,thick,domain=180:270] plot ({2+.5*cos(\x)}, {1+.5*sin(\x)});
\draw [black,thick] (2,.5)--(3,.5){};
\draw [black,thick] (1.5,1)--(1.5,2){};
\draw[very thin, color=blue!50] (0,0) grid (3,2);
\end{scope}

\begin{scope}[xshift=120, yshift=31, scale=0.3, local bounding box=p312]
\node[anchor=south]() at (1.5,-0.3) {};
\node[anchor=north]() at (1.5,1.3) {};
\node[] at (1.5, -0.4){\tiny$\overline{3}\underline{12}$};
\draw [black,thick,domain=180:270] plot ({1+.5*cos(\x)}, {1+.5*sin(\x)});
\elbow[thick](2,1);
\elbow[thick](3,1);
\draw[very thin, color=blue!50] (0,0) grid (3,1);
\end{scope}

\begin{scope}[xshift=80, yshift=31, scale=0.3, local bounding box=p213]
\node[anchor=south]() at (1.5,-0.3) {};
\node[anchor=north]() at (1.5,1.3) {};
\node[] at (1.5, -0.4){\tiny$\overline{2}\underline{13}$};
\draw [black,thick,domain=180:270] plot ({1+.5*cos(\x)}, {1+.5*sin(\x)});
\elbow[thick](2,1);
\cross[thick](3,1);
\draw[very thin, color=blue!50] (0,0) grid (3,1);
\end{scope}

\begin{scope}[xshift=40, yshift=31, scale=0.3, local bounding box=p321]
\node[anchor=south]() at (1.5,-0.3) {};
\node[anchor=north]() at (1.5,1.3) {};
\node[] at (1.5, -0.4){\tiny$\overline{3}\underline{21}$};
\draw [black,thick,domain=180:270] plot ({1+.5*cos(\x)}, {1+.5*sin(\x)});
\cross[thick](2,1);
\elbow[thick](3,1);
\draw[very thin, color=blue!50] (0,0) grid (3,1);
\end{scope}

\begin{scope}[xshift=0, yshift=31, scale=0.3, local bounding box=p1bar23]
\node[anchor=south]() at (1.5,-0.3) {};
\node[anchor=north]() at (1.5,1.3) {};
\node[] at (1.5, -0.4){\tiny$\overline{1}\underline{2}\underline{3}$};
\draw [black,thick,domain=180:270] plot ({1+.5*cos(\x)}, {1+.5*sin(\x)});
\cross[thick](2,1);
\cross[thick](3,1);
\draw[very thin, color=blue!50] (0,0) grid (3,1);
\end{scope}

\begin{scope}[xshift=-40, yshift=31, scale=0.3, local bounding box=p132]
\node[anchor=south]() at (1.5,-0.3) {};
\node[anchor=north]() at (1.5,1.3) {};
\node[] at (1.5, -0.4){\tiny$\underline{1}\overline{3}\underline{2}$};
\draw [black,thick,domain=180:270] plot ({2+.5*cos(\x)}, {1+.5*sin(\x)});
\elbow[thick](3,1);
\draw[very thin, color=blue!50] (0,0) grid (3,1);
\end{scope}

\begin{scope}[xshift=-80, yshift=31, scale=0.3, local bounding box=p12bar3]
\node[anchor=south]() at (1.5,-0.3) {};
\node[anchor=north]() at (1.5,1.3) {};
\node[] at (1.5, -0.4){\tiny$\underline{1}\overline{2}\underline{3}$};
\draw [black,thick,domain=180:270] plot ({2+.5*cos(\x)}, {1+.5*sin(\x)});
\cross[thick](3,1);
\draw[very thin, color=blue!50] (0,0) grid (3,1);
\end{scope}

\begin{scope}[xshift=-120, yshift=31, scale=0.3, local bounding box=p123bar]
\node[anchor=south]() at (1.5,-0.3) {};
\node[anchor=north]() at (1.5,1.3) {};
\node[] at (1.5, -0.4){\tiny$\underline{1}\underline{2}\overline{3}$};
\draw [black,thick,domain=180:270] plot ({3+.5*cos(\x)}, {1+.5*sin(\x)});
\draw[very thin, color=blue!50] (0,0) grid (3,1);
\end{scope}

\begin{scope}[xshift=0, yshift=0, scale=0.3, local bounding box=p123]
\node[anchor=north]() at (1.5,0.3) {};
\node[]() at (1.5,-0.5) {\tiny$\underline{123}$};
\draw[very thin, color=blue!50] (0,0) grid (3,0);
\end{scope}

\draw[thick] (p231.north)--(p1bar2bar3bar.south);
\draw[thick] (p32bar1.north)--(p1bar2bar3bar.south);
\draw[thick] (p1bar32.north)--(p1bar2bar3bar.south);
\draw[thick] (p1bar2bar3.north)--(p1bar2bar3bar.south);
\draw[thick] (p213bar.north)--(p1bar2bar3bar.south);
\draw[thick] (p1bar23bar.north)--(p1bar2bar3bar.south);
\draw[thick] (p12bar3bar.north)--(p1bar2bar3bar.south);

\draw[thick] (p312.north)--(p231.south);
\draw[thick] (p312.north)--(p1bar32.south);
\draw[thick] (p312.north)--(p213bar.south);

\draw[thick] (p213.north)--(p231.south);
\draw[thick] (p213.north)--(p1bar2bar3.south);
\draw[thick] (p213.north)--(p213bar.south);

\draw[thick] (p321.north)--(p1bar23bar.south);

\draw[thick] (p1bar23.north)--(p1bar32.south);
\draw[thick] (p1bar23.north)--(p1bar2bar3.south);
\draw[thick] (p1bar23.north)--(p1bar23bar.south);

\draw[thick] (p132.north)--(p231.south);
\draw[thick] (p132.north)--(p1bar32.south);
\draw[thick] (p132.north)--(p12bar3bar.south);

\draw[thick] (p12bar3.north)--(p32bar1.south);
\draw[thick] (p12bar3.north)--(p1bar2bar3.south);
\draw[thick] (p12bar3.north)--(p12bar3bar.south);

\draw[thick] (p123bar.north)--(p213bar.south);
\draw[thick] (p123bar.north)--(p1bar23bar.south);
\draw[thick] (p123bar.north)--(p12bar3bar.south);

\draw[thick] (p123.north)--(p312.south);
\draw[thick] (p123.north)--(p213.south);
\draw[thick] (p123.north)--(p321.south);
\draw[thick] (p123.north)--(p1bar23.south);
\draw[thick] (p123.north)--(p132.south);
\draw[thick] (p123.north)--(p12bar3.south);
\draw[thick] (p123.north)--(p123bar.south);

\end{tikzpicture}
\caption{The poset of nonnegatively representable elementary positroid quotients $(\Pi_3,\unlhd_q)$, indexed by their partial flag positroid pipe dreams $D^{\sLe}$ and decorated permutations.
This is a subposet of $(\Pi_3, \leq_q)$ shown in Figure~\ref{fig.3poset}.
}
\label{fig.3NNposet}
\end{figure}
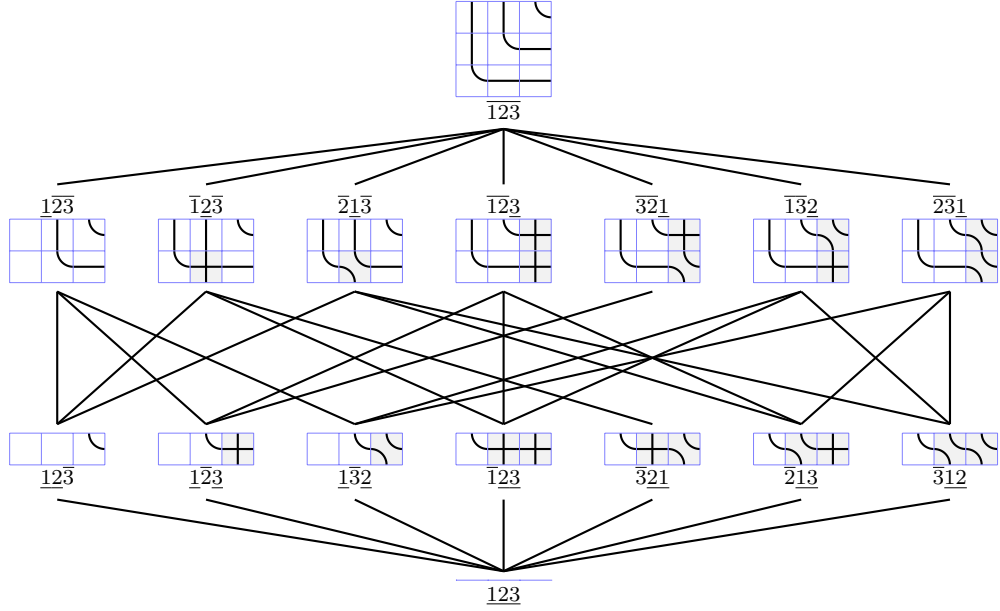

To conclude, we give one last result on the poset of nonnegatively representable elementary positroid quotients.
\begin{proposition}\label{prop.selfdual}
The poset $(\Pi_n, \unlhd_q)$ is self-dual via taking positroid duals (or equivalently, via taking inverses of decorated permutations).
\end{proposition}
\begin{proof}
As observed by Oh~\cite{Oh13}, the result~\cite[Theorem 19]{Oh11} implies that the dual of a positroid $P$ is a positroid $P^*$, and the decorated permutation of $P^*$ is the inverse of the decorated permutation of $P$.
We point out the fact that if the value $\pi(j)$ is $1$-coloured (respectively $2$-coloured) in $\pi$, then it is a $2$-coloured (respectively $1$-coloured) position in $\pi^{-1}$.

We will use the decorated permutation perspective to show that $(\Pi_n,\unlhd_q)$ is self-dual, by showing that $\pi \lessdot_q \pi'$ if and only if $\omega \gtrdot_q \omega'$, where $\omega = \pi^{-1}$ and $\omega'=(\pi')^{-1}$.

Suppose $\pi \lessdot_q \pi'$.
Note that $U$ is the set of (right) unblocked positions of $\pi$ if and only if $S=\pi(U)$ is the set of left unblocked positions of $\omega = \pi^{-1}$.

By Theorem~\ref{thm.covering_decperms} 
there exists a nonempty subset $C\subseteq U$ such that $\pi' = \pi \sigma_{C\sqcup T(C)}$.
Suppose $C\sqcup T(C) = \{c_1,\ldots, c_r\}\sqcup\{t_1,\ldots, t_s\}$ so that 
\[t_1< \cdots<t_s<c_1< \cdots<c_r
\quad\hbox{and}\quad
\pi(c_1) < \cdots<\pi(c_r)<\pi(t_1)<\cdots <\pi(t_s).
\]
Taking inverses means that $R^* = \pi(C) = \{\pi(c_1), \ldots, \pi(c_r) \} \subseteq S^*$ is a nonempty subset of left unblocked positions of $\omega$.
By construction, it follows that $O(R^*) = \pi(T(C)) = \{\pi(t_1),\ldots, \pi(t_s) \}$, and $\tau_{R^* \sqcup O(R^*)} = (\pi(t_1) \, \cdots \, \pi(t_s) \, \pi(c_1)\, \cdots \, \pi(c_r))$.
Then
\begin{align*}
\omega'
&= (\pi')^{-1} = (\pi \sigma_{C\sqcup T(C)})^{-1}
= (t_1\, \cdots\, t_s \, c_1 \, \cdots \, c_r) \cdot \pi^{-1}\\
&= \pi^{-1}\cdot (\pi(t_1)\, \cdots\, \pi(t_s) \, \pi(c_1) \, \cdots \, \pi(c_r))\\
&= \omega \tau_{R^*\sqcup O(R^*)}\\
&=\overleftarrow{\rho_{A(R^*)}}(\omega),
\end{align*}
where the equality in the second line follows from $(i\ j)\cdot \sigma = \sigma \cdot (\sigma^{-1}(i) \ \sigma^{-1}(j))$.
\end{proof}

\begin{corollary}\label{cor.ukncovers}
Let $U_{k,n}$ denote the rank $k$ uniform matroid on $[n]$.
There are $2^k-1$ nonnegatively representable elementary positroid quotients of the form $P\unlhd_q U_{k,n}$.
\end{corollary}
\begin{proof}
The standardized partial flag positroid pipe dream $D^{\sLe}(U_{n-k,n})$ has pivot elbows in the first $n-k$ columns and elbow tiles in its last $k$ columns. 
Therefore, $D^{\sLe}(U_{n-k,n})$ has $k$ unblocked columns and is covered by $2^k-1$ positroids in $(\Pi_n, \unlhd_q)$, by Corollary~\ref{cor.char_Q}.
Since $(\Pi_n, \unlhd_q)$ is self-dual and $U_{n-k,n}^* = U_{k,n}$, it follows that $U_{k,n}$ covers $2^k-1$ positroids in $(\Pi_n, \unlhd_q)$.
\end{proof}
\begin{remark}
Contrast Corollary~\ref{cor.ukncovers} with~\cite[Theorem 28]{BCT22}, which states that $U_{k,n}$ has at least $\sum_{i=0}^{k-1}\binom{n}{i}$ elementary positroid quotients of the form $P\leq_q U_{k,n}$.
To get a rough sense of what proportion of these positroid quotients are nonnegatively representable, note that in the case when $n$ is odd and $k=(n-1)/2$, we have $\sum_{i=0}^{k-1}\binom{n}{i}=2^{n-1}$.
\end{remark}

\begin{corollary}
The subposet of $(\Pi_n, \unlhd_q)$ induced by the lattice path matroids on $[n]$ is self-dual. \qed
\end{corollary}

\begin{example}\label{eg.pi_omega}
The decorated permutations $\pi$ and $\pi'$ from Example~\ref{eg.grand_finale} and their inverses $\omega =\pi^{-1}$ and $\omega'=(\pi')^{-1}$ are shown in Figure~\ref{fig.pi_omega}.
Recall that $\pi\lessdot_q \pi'$.
We verify that $\omega\gtrdot_q \omega'$ by showing that $\omega'$ is a left cyclic shift of $\omega$.

The right cyclic shift set for $\pi$ is $C\sqcup T(C) = \{5,9\}\sqcup \{4\}$ and we have $\pi' = \pi\cdot (9\,5\,4)$.
Taking inverses means that the corresponding left cyclic shift set for $\omega$ is $R^* \sqcup O(R^*) =\{\pi_5,\pi_9\}\sqcup\{\pi_4\} = \{2,8,9\}$, and $\omega' = \omega\cdot (2\,8\,9)$.
\end{example}

\begin{figure}[ht!]
\begin{center}
\begin{tikzpicture}
\begin{scope}[scale=1, xshift=0, yshift=50, local bounding box=piprime]
\node[] at (0,0){
$\pi' = {\setlength\arraycolsep{1.6pt}
\begin{matrix}
\textcolor{red}{1}&\textcolor{blue}{2}&\textcolor{blue}{3}&\textcolor{red}{4}&\textcolor{red}{5}&\textcolor{red}{6}&\textcolor{blue}{7}&\textcolor{blue}{8}&\textcolor{blue}{9}\\
\overline{5}&\underline{1}&\underline{3}&\overline{8}&\overline{9}&\overline{7}&\underline{6}&\underline{4}&\underline{2}
\end{matrix}}$};
\end{scope}

\begin{scope}[scale=1, xshift=0, yshift=0, local bounding box=pi]
\node[] at (0,0){
$\pi = {\setlength\arraycolsep{1.6pt}\begin{matrix}
\textcolor{red}{1}&\textcolor{blue}{2}&\textcolor{blue}{3}&\textcolor{red}{4}&\textcolor{blue}{5}&\textcolor{red}{6}&\textcolor{blue}{7}&\textcolor{blue}{8}&\textcolor{blue}{9}\\
\overline{5}&\underline{1}&\underline{3}&\overline{9}&\underline{2}&\overline{7}&\underline{6}&\underline{4}&\underline{8}
\end{matrix}}$};
\end{scope}

\begin{scope}[scale=1, xshift=170, yshift=50, local bounding box=omega]
\node[] at (0,0){
$\omega = {\setlength\arraycolsep{1.6pt}\begin{matrix}
\textcolor{red}{1}&\textcolor{red}{2}&\textcolor{red}{3}&\textcolor{red}{4}&\textcolor{blue}{5}&\textcolor{red}{6}&\textcolor{blue}{7}&\textcolor{red}{8}&\textcolor{blue}{9}\\
\overline{2}&\overline{5}&\overline{3}&\overline{8}&\underline{1}&\overline{7}&\underline{6}&\overline{9}&\underline{4}
\end{matrix}}$};
\end{scope}

\begin{scope}[scale=1, xshift=170, yshift=0, local bounding box=omegaprime]
\node[] at (0,0){
$\omega' = {\setlength\arraycolsep{1.6pt}\begin{matrix}
\textcolor{red}{1}&\textcolor{red}{2}&\textcolor{red}{3}&\textcolor{red}{4}&\textcolor{blue}{5}&\textcolor{red}{6}&\textcolor{blue}{7}&\textcolor{blue}{8}&\textcolor{blue}{9}\\
\overline{2}&\overline{9}&\overline{3}&\overline{8}&\underline{1}&\overline{7}&\underline{6}&\underline{4}&\underline{5}
\end{matrix}}$};
\end{scope}

\draw[-stealth, thick] (pi.west) .. controls +(left:.5cm) and +(left:.5cm) .. node[left]{$\overrightarrow{\rho_{A(C)}}$} (piprime.west);

\draw[-stealth, thick] (omega.east) .. controls +(right:.5cm) and +(right:.5cm) .. node[right]{$\overleftarrow{\rho_{A(R^*)}}$} (omegaprime.east);

\draw[-stealth, thick] (pi.east) .. controls +(right:.5cm) and +(right:0cm) .. node[above]{$\,^{-1}$} (omega.west);

\draw[-stealth, thick] (piprime.east) .. controls +(right:.5cm) and +(right:0cm) .. (omegaprime.west);

\end{tikzpicture}
\end{center}
\caption{Illustrating that $(\Pi_n,\unlhd_q)$ is self-dual via taking the inverse of the decorated permutations. 
See Example~\ref{eg.pi_omega}.}
\label{fig.pi_omega}
\end{figure}
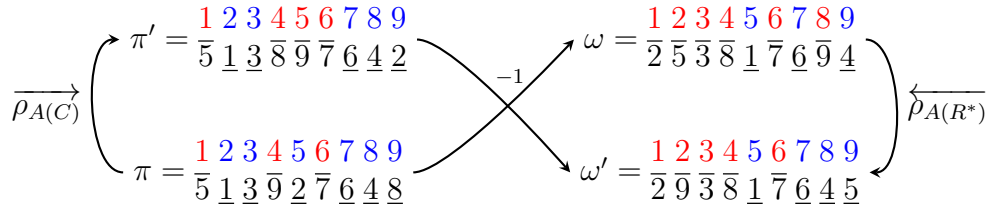

\printbibliography
\end{document}